\documentclass[12pt,a4paper]{amsart}
\usepackage[all]{xy}
\usepackage{tikz}
\usetikzlibrary{arrows,backgrounds,decorations.markings}

\usepackage{graphicx, verbatim,amsmath,amssymb,subfigure,enumerate}

\allowdisplaybreaks

\def\tcr{\textcolor{red}}

\def\newspan{\operatorname{span}}

\def\supp{\operatorname{supp}}
\def\ker{\operatorname{ker}}

\def\dim{\operatorname{dim}}

\def\im{\operatorname{Image}}

\def\Im{\operatorname{Im}}

\def\supp{\operatorname{supp}}

\def\R{\mathbb{R}}
\def\N{\mathbb{N}}
\def\Z{\mathbb{Z}}

\def\LL{\mathcal{L}}

\def\ep{\varepsilon}

\def\pset{\mathcal{P}}
\def\ptile{p}

\def\supp{\operatorname{supp}}
\def\Int{\operatorname{Int}}
\def\rrule{\mathcal{R}}

\def\domain{X}

\def\tilingspace{\Omega}

\def\sub{\omega}
\def\disjoint{\bigsqcup_{p \in \pset}}
\def\disjointelt{\sqcup_{p \in \pset}}

\def\disjointE{\bigsqcup_{e \in E(G)}}
\def\MWedgegraph{\Delta}

\def\Tinf{T^{(\infty)}}
\def\tinf{t^{(\infty)}}

\def\tzero{T^{(0)}}
\def\tone{T^{(1)}}
\def\tn{T^{(n)}}
\def\psik{\psi^{(k)}}
\def\psin{\psi^{(n)}}
\def\wpsin{\widetilde\psi^{(n)}}
\def\wpsiinf{\widetilde\psi^{(\infty)}}
\def\psiinf{\psi^{(\infty)}}
\def\oinf{\Omega^{(\infty)}}
\def\pinf{\pset^{(\infty)}}
\def\Sinf{S^{(\infty)}}
\def\KSinf{A^{(\infty)}}
\def\einf{e^{(\infty)}}
\def\subinf{\omega_{(\infty)}}

\def\MWedgegraph{\Delta}

\def\diam{\operatorname{diam}}

\renewcommand{\int}{\operatorname{int}}

\newtheorem{thm}{Theorem}[section]
\newtheorem{cor}[thm]{Corollary}

\newtheorem{lemma}[thm]{Lemma}
\newtheorem{prop}[thm]{Proposition}

\theoremstyle{definition}
\newtheorem{definition}[thm]{Definition}

\theoremstyle{remark}

\newtheorem{example}[thm]{Example}

\numberwithin{equation}{section}

\tikzstyle{vertex}=[circle,thick]
\tikzstyle{goto}=[->,shorten >=1pt,>=stealth,semithick]

\pgfpathcircle{\pgfpoint{5cm}{0cm}} {2pt}

\newcommand{\chaird}[5][1]{ 
\begin{scope}[scale={#1},shift={({#2},{#3})}, rotate around={#4:(1/2,1/2)}]
\node at (0,0) {$\scriptstyle\bullet$};
\draw[ultra thin] (0,0) -- (1,0) -- (1,1/2) -- (1/2,1/2) -- (1/2,1) -- (0,1) -- (0,0);
\node at (0.3,0.3) {$\scriptstyle{#5}$};
\end{scope}
}
\newcommand{\chair}[5][1]{	
\begin{scope} [scale={#1},shift={({#2},{#3})}, rotate around={#4:(1/2,1/2)}]
\draw[ultra thin] (0,0) -- (1,0) -- (1,1/2) -- (1/2,1/2) -- (1/2,1) -- (0,1) -- (0,0);
\node at (0.3,0.3) {$\scriptstyle{#5}$};
\end{scope}
} 
 
 \title[Fractal dual substitution tilings]{Fractal dual substitution tilings}
 
 \author{Natalie Priebe Frank}
 \address{Department of Mathematics, Vassar College, Box 248, Poughkeepsie, NY  12604, USA}
 \email{nafrank@vassar.edu}

 \author{Samuel B.G. Webster}
 \address{Samuel B.G. Webster, Michael F. Whittaker \\ School of Mathematics and Applied Statistics  \\ The University of Wollongong\\ NSW  2522\\ AUSTRALIA} \email{sbgwebster@gmail.com, mfwhittaker@gmail.com}

 \author{Michael F. Whittaker}
 
 \thanks{This research was partially supported by the following:  the Australian Research Council, the Institute of Mathematics and its Applications at University of Wollongong, the German Research Council (DFG) within the CRC 701, and the Babette Rogol '61 Memorial Fund of Vassar College.}
 

\begin{document}

\begin{abstract}
Starting with a substitution tiling, we demonstrate a method for constructing infinitely many new substitution tilings. Each of these new tilings is derived from a graph iterated function system and the tiles have fractal boundary. We show that each of the new tilings is mutually locally derivable to the original tiling. Thus, at the tiling space level, the new substitution rules are expressing geometric and combinatorial, rather than topological, features of the original. Our method is easy to apply to particular substitution tilings, permits experimentation, and can be used to construct border-forcing substitution rules. For a large class of examples we show that the combinatorial dual tiling has a realization as a substitution tiling. Since the boundaries of our new tilings are fractal we are led to compute their fractal dimension. As an application of our techniques we show how to compute the \v{C}ech cohomology of a (not necessarily border-forcing) tiling using a graph iterated function system of a fractal tiling.
\end{abstract}

\maketitle

\section{Introduction}

A tiling of the plane is a covering of $\R^2$ by closed subsets, called tiles, such that the interiors of two distinct tiles are disjoint. One method of producing tilings is by substitution: a rule which expands each tile by a fixed amount and then breaks each expanded tile into smaller pieces, each of which is an isometric copy of an original tile. The most famous substitution tiling is the Penrose tiling \cite{pentaplexity}; many other substitution tilings have been constructed and can be found online in the Tilings Encyclopedia \cite{Tiling Encyc}.

\subsection{Overview of the field.}
The study of substitution tilings is motivated by several disparate fields of mathematics and science. One thing they all have in common is that tilings are always constructed using a finite set of tiles called \emph{prototiles} as their building blocks.  The study of {\em aperiodic} prototile sets---finite prototile sets that can only form nonperiodic tilings---began with logician Hao Wang \cite{Wang}.  He tied them to the decidability of the Domino Problem, which asks if there is an algorithm that can determine whether any given set of (square) prototiles, with specified edge matching rules, can tile the plane.  He conjectured that the problem was decidable, but his conjecture depended upon non-existence of aperiodic prototile sets. Wang's student Robert Berger \cite{Ber} proved the conjecture false by producing an aperiodic prototile set with 20,426 prototiles.  It was the search for small aperiodic prototile sets that led Penrose to discover his famous tiles.

Research on aperiodic tilings was purely academic until Dan Shechtman's Nobel Prize-winning discovery of quasicrystals \cite{She}. In a sense, Shechtman's discovery was predicted  by Alan L. Mackay \cite{Mackay} when he computed the diffraction of a Penrose tiling. Its  `impossible' symmetry was found again in Shectman's quasicrystal and it became clear that Penrose tilings could be used as a mathematical model. 
Diffraction patterns of quasicrystals are now modeled by the spectral theory of tiling spaces, in particular the dynamical spectrum of the translation operator (see \cite{BG} for a beautiful exposition of the most recent advances in this theory). The dynamical spectrum of tilings is best understood in the case of substitution tilings.

An understanding of tiling spaces at this physical level led Bellissard to his gap-labelling conjecture \cite{Bel}, which connects the spectral gaps of the Schr\"odinger operator associated with a tiling space to the K-theory and cohomology of that space. Bellissard's conjecture was proven  independently in \cite{Bel,BO,KP}.  By the same token, tiling spaces provide interesting examples of $C^*$-algebras, and the geometry imposed by tilings makes them excellent candidates for study using Connes' noncommutative geometry program \cite{Con}.

\subsection{Results.}

The main goal of this paper is to establish a novel method for producing, from some given substitution tiling, infinitely many new substitution tilings.  The method produces finite sets of prototiles which have fractal boundaries, so we call them \emph{fractal realizations} of the original tiling.  The construction depends on a synthesis of combinatorics and geometry in a way that has not previously been capitalized upon in the study of substitution tilings.  Moreover, it allows for hands-on experimentation with an infinitude of choices, and for now it is not completely clear the significance of making one choice over another.  In \cite{FW}, we gave a somewhat ad-hoc method for making a  fractal realization of the Pinwheel tiling. In this paper we show that a similar construction works for every primitive substitution tiling where tiles meet singly edge-to-edge and remove the ad-hoc nature of the construction in \cite{FW}.

One advantage to our method is that it provides access to some of the physical information described above.  For instance, our new substitutions are easily made to \emph{force the border}, an essential property for cohomology computations.  When a tiling does not force the border, our method requires less information than the original approach in \cite{AP} and uses the same amount of information, but is geometrically more elegant, than the most efficient known method \cite{BDHS}.  
The real power behind our approach comes from the geometric information encoded in the new substitution, which we capitalize on by constructing spectral triples on the $C^*$-algebra of the original tiling space in \cite{Mampusti,FMWW}. According to Connes' noncommutative geometry program \cite{Con}, the existence of spectral triples implies that fractal realizations are defining a noncommutative Riemannian geometry on the $C^*$-algebra of the original tiling space (viewed as a noncommutative manifold).

There are further interesting questions invoked by our fractal substitution tilings.  For example, we give a formula for the fractal dimension of the boundary of the prototiles for any of our fractal tilings. But what information does the set of all possible fractal dimensions for a given tiling substitution carry? The answer to this question seems to rely on a deep connection between the combinatorial graph-theoretic properties of the tiling and the geometric property of fractal dimension. Further questions arise from a combinatorial standpoint; for example, which substitution tilings admit {\em combinatorial dual tilings}\footnote{Two tilings are combinatorially dual if there is a one-to-one correspondence between their edge sets, between the tiles of one and the vertices of the other, and vice versa.} that are also substitution tilings?  We provide sufficient conditions but have not yet been able to provide a complete characterization.

Our method of obtaining fractal substitution tilings is different than those that have appeared in the literature. For example, fractal tilings arise in the seminal work of Kenyon \cite{Kenyon} characterizing the possible expansion factors for substitution tilings. Since then fractal tilings have appeared many times, for instance in \cite{Gensubref,Bandt-Gummelt,BV,Me.Boris,FW}. Our construction distinguishes itself from these in a few ways. The first is that we begin with a known substitution rule and construct an infinite family of substitutions whose tiling spaces are mutually locally derivable. In particular, we are able to construct substitution tilings that are mutually locally derivable to their original but with specified geometric properties that lend themselves more easily to computation.  Our construction differs, also, because it is not simply a redrawing of existing tiles but rather a recomposition that creates new tiles.  We can understand, and indeed have some control over, the combinatorics of the new tiles within their tilings.

\subsection{Methodology--an example}

The main tool in our construction is a \emph{recurrent pair} $(G,S)$, which is a combinatorial and geometric structure that is compatible with the original substitution. The best way to illustrate how to construct a recurrent pair is with a simple example.  We begin with a two-dimensional version of the Thue-Morse substitution rule, which we call the ``2DTM" substitution. It has two prototiles, labelled $\alpha$ and $\beta$, that are both unit squares. The substitution expands each prototile by a factor of four and replaces it with the patches shown in Figure \ref{Thue-Morse-substitution}.

\begin{figure}[ht]
\scalebox{0.8}{
\begin{tikzpicture}
\node (tile1) at (0,0) [label=below:{$\alpha$}] {\includegraphics[width=0.75cm]{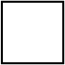}};
\node (tile1sub) at (3,0) [label=below:{$\omega(\alpha)$}] {\includegraphics[width=3cm]{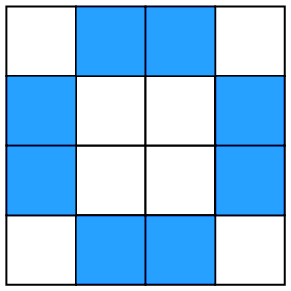}}
	edge[<-] (tile1);	
\begin{scope}[xshift=7cm]
\node (tile2) at (0,0) [label=below:{$\beta$}] {\includegraphics[width=0.75cm]{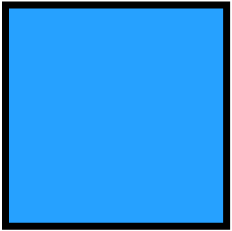}};
\node (tile2sub) at (3,0) [label=below:{$\omega(\beta)$}] {\includegraphics[width=3cm]{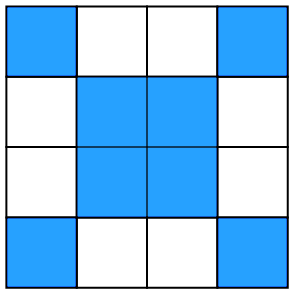}}
	edge[<-] (tile2);	
\end{scope}
\end{tikzpicture}}
\caption{The two-dimensional Thue-Morse substitution rule.}
\label{Thue-Morse-substitution}
\end{figure}

To construct a recurrent pair $(G, S)$, we begin by embedding a planar graph $G$ into each of the prototiles $\alpha$ and $\beta$, as shown on the left of Figure \ref{2DTM_graph_sub}.   Note that we have chosen $G$ so that it has one vertex in the interior of each tile and connects to the interior of each edge.  If $G$ were placed in every tile in an infinite tiling of these tiles, it would construct a new tiling that is combinatorially dual to the original.  In general it is natural to start with such a dual graph for several reasons that will be outlined in this paper.

Next we construct a new graph $S$ on the prototiles by substituting the prototiles, this time without expanding, embedding the initial graph into each sub-tile, then selecting a subgraph $S$ that is combinatorially equivalent to the initial graph if we ignore all vertices of degree two.\footnote{It is important to note that selecting such a graph is not necessarily possible in general and depends on the combinatorics of the tiling, the substitution, and $G$}  In Figure \ref{2DTM_graph_sub}, the embedding stage is labelled $\rrule(G)$, and the selection stage is labelled $S$.   We have made a choice of $S$ that breaks the inherent symmetry of the 2DTM tilings, but this was not necessary.

\begin{figure}[ht]
\[
\begin{tikzpicture}[>=stealth,scale=1.1]
\node (a) at (0,3.5) [label=below:{$G_\alpha$}] {\includegraphics[width=2.5cm]{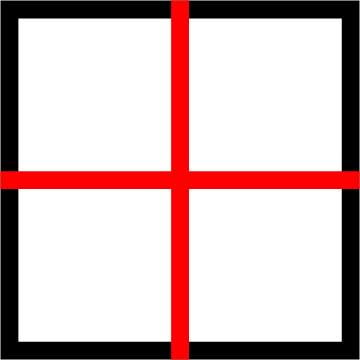}};
\node (b) at (3.5,3.5) [label=below:{$\rrule(G)_\alpha$}] {\includegraphics[width=2.5cm]{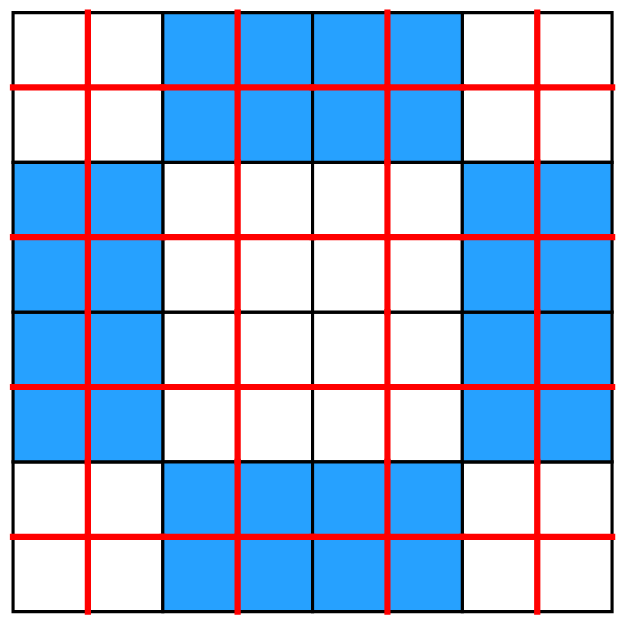}}
	edge[<-] (a);
\node (c) at (7,3.5) [label=below:{$S_\alpha$}] {\includegraphics[width=2.5cm]{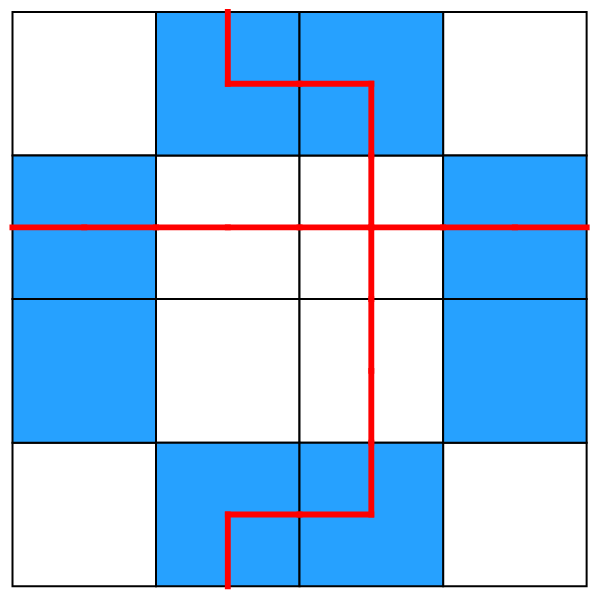}}
	edge[<-] (b);
\node (d) at (10.5,3.5) [label=below:{Fractal $G_\alpha$}] {\includegraphics[width=2.5cm]{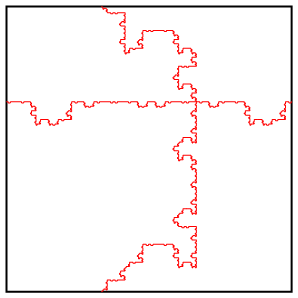}}
	edge[<-,dashed] (c);
\node (e) at (0,0) [label=below:{$G_\beta$}] {\includegraphics[width=2.5cm]{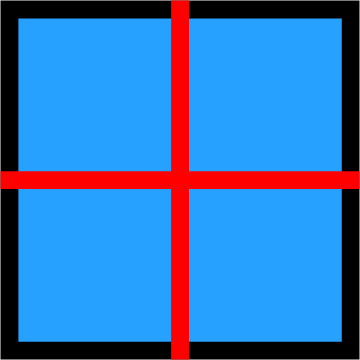}};
\node (f) at (3.5,0) [label=below:{$\rrule(G)_\beta$}] {\includegraphics[width=2.5cm]{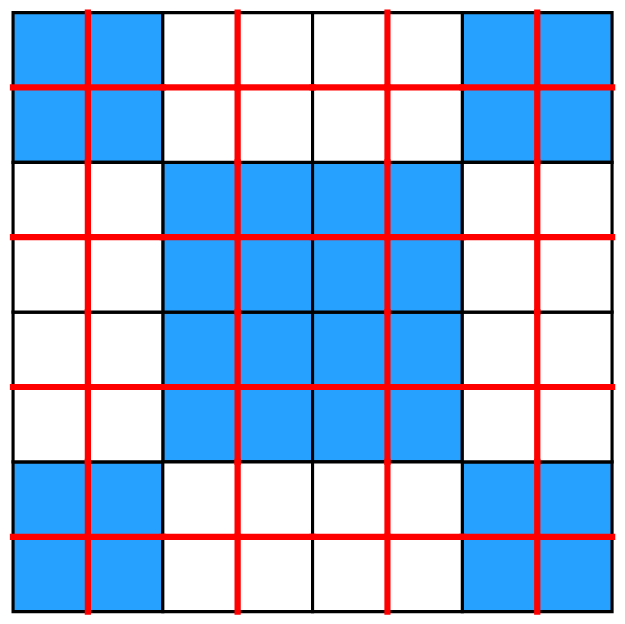}}
	edge[<-] (e);
\node (g) at (7,0) [label=below:{$S_\beta$}] {\includegraphics[width=2.5cm]{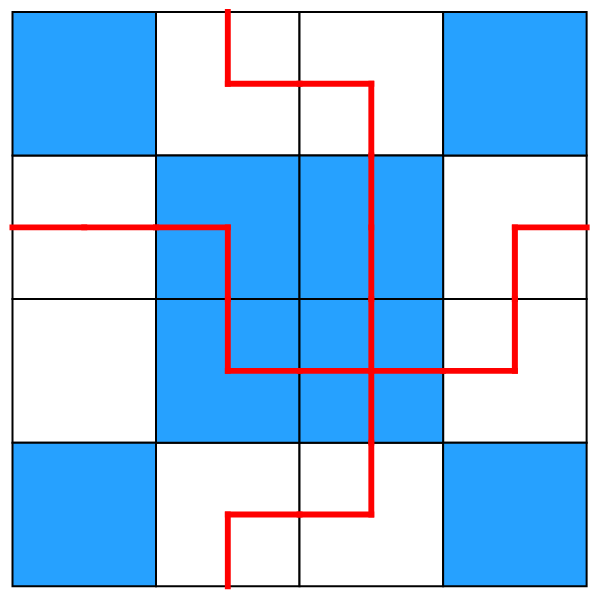}}
	edge[<-] (f);
\node (h) at (10.5,0) [label=below:{Fractal $G_\beta$}] {\includegraphics[width=2.5cm]{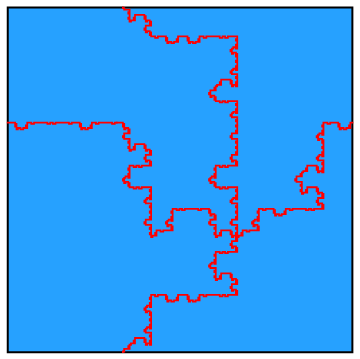}}
	edge[<-,dashed] (g);
\end{tikzpicture}
\]
\caption{A recurrent pair on the $2$-dimensional Thue-Morse substitution leading to a fractal graph.}
\label{2DTM_graph_sub}
\end{figure}

We call $(G,S)$ a recurrent pair, and $S$ can be thought of as of graph substitution of $G$. In fact, a recurrent pair forms a graph iterated function system (GIFS) and has an attractor (which is a fractal). Under the right conditions the attractor is a graph that is combinatorially equivalent to both $G$ and $S$.   We show the attractor for this example in the far right of Figure \ref{2DTM_graph_sub}.

Figure \ref{2DTM_overlay} shows what happens when the attractor is placed into every tile of a 2DTM tiling.
Everything works perfectly in this example and we obtain a combinatorially dual tiling that is itself a substitution tiling. The fact that it is a substitution tiling follows from the fact that the attractor of the recurrent pair $(G,S)$ is invariant under the substitution rule of $T$.  
\begin{figure}[ht]
\[
\includegraphics[width=15cm]{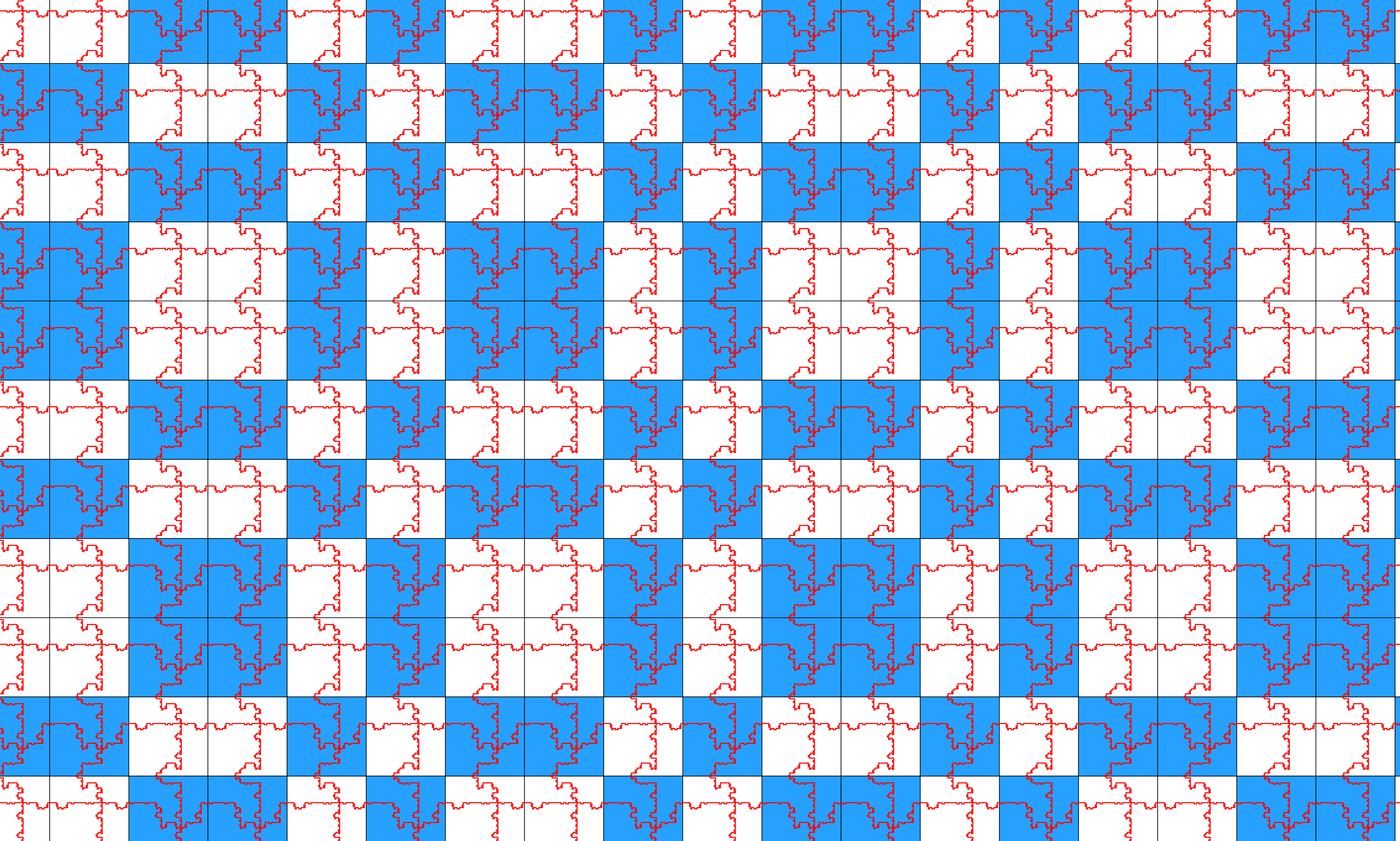}
\]
\caption{A patch of the $2$-dimensional Thue-Morse tiling with a fractal dual tiling overlaid.}
\label{2DTM_overlay}
\end{figure}
%

\subsection{Organziation of the paper.}

Given a substitution tiling, our main result is to construct new substitution tilings using a recurrent pair $(G, S)$. There is a natural map from the edges of $G$ to the edges of the attractor associated with $(G,S)$. If this natural map is injective, then Theorem \ref{prop:T infty is a tiling} says that the resulting tiling is mutually locally derivable to the original substitution tiling. In Theorem \ref{thm: psi infinity is injective} we provide sufficient \emph{injectivity conditions} on the recurrent pair for the map to be injective.  Theorem \ref{thm:existence} implies that every primitive substitution tiling whose tiles meet singly edge-to-edge has an infinite number of distinct recurrent pairs satisfying the injectivity conditions and whose fractal realization is border-forcing. We conjecture that every substitution tiling of the plane with finite local complexity has a recurrent pair with the necessary injectivity conditions. We also believe that our techniques should extend to tilings with infinite rotational symmetry, such as the Pinwheel tiling, but have not addressed that in this paper.

We have organized the paper as follows. In Section \ref{tiling_defn_sec} we introduce substitution tilings along with the definitions we require in the paper. Section \ref{comb_defn_sec} introduces the combinatorics of tilings. We give an alternative description of tiling substitutions in Section \ref{alt_view_sub_sec} using digit sets, and describe how these digit sets give rise to a contraction map on prototiles. Section \ref{sec:psi} introduces the notion of a recurrent pair on a substitution tiling $T$ and shows how a recurrent pair produces a fractal realization of $T$. In Section \ref{sec:exist} we give injectivity conditions on a given recurrent pair and we prove they are what is needed to guarantee existence of a border-forcing fractal realization. We then show that every primitive substitution tiling with tiles meeting singly edge-to-edge has such a recurrent pair. In Section \ref{FractalDimension_sec} we compute the fractal dimension of the new prototile boundaries. Section \ref{cohom_zeta_sec} shows how a recurrent pair can be used to compute the \v{C}ech cohomology of the original tiling space. Finally, in Appendix \ref{examples_appendix_sec} we give several examples of fractal realizations of some famous substitution tilings.

\noindent
\textbf{Acknowledgements:} We are extremely grateful to Michael Baake and Franz G\"{a}hler for  helping to rectify an error in the cohomology computations from an early version of the paper. We are also indebted to Lorenzo Sadun for several helpful conversations and ideas and to Michael Mampusti for the Mathematica code used to produce some of the figures.

\section{Tiling definitions}\label{tiling_defn_sec}

In this section we introduce the basics of tilings. We provide the assumptions that will be used in the paper and introduce substitution rules. We establish the relationship between tiling substitution rules and self-similar tilings so that we may use these two notions interchangeably for the remainder of the paper. 

The tilings in this paper are built out of a finite number of labelled shapes called prototiles. A {\em prototile} $p$ consists of a closed subset of $\R^2$ that is homeomorphic to a topological disk, denoted $\supp(p)$, and a label $\ell(p)$. The purpose of the labels is to distinguish between prototiles that have the same shape, and a common visualization is by color (for instance white vs. black unit squares). A {\em prototile set} $\pset$ is a finite set of prototiles.

A {\em tile} $t$ is any translate of a prototile: for $p \in \pset$ and $x \in \R^2$ we use the notation $t=p+x$ to mean the topological disk $\supp(p)+x$ with label $\ell(t)=\ell(p)$. Two tiles that are translates of the same prototile are said to be {\em equivalent}; we note that equivalent tiles have congruent supports and carry the same label.
\begin{definition} Given  a set of prototiles $\pset$, a \emph{tiling}  is a countable collection of tiles $\{t_1,t_2,\dots \}$ each of which is a translate of a prototile, and such that 
\begin{enumerate}
\item $\bigcup_{i=1}^\infty \supp(t_i) = \R^2$ and
\item $\Int(\supp(t_i)) \cap \Int(\supp(t_j)) = \varnothing$ for $i \neq j$.
\end{enumerate}
\end{definition}
\noindent The first condition implies that the tiles cover the plane and the second that they intersect on the boundaries only. These are sometimes referred to as the covering and packing conditions.

A {\em patch} is a finite set of tiles whose supports cover a connected set that intersect at most on their boundaries.  Connected finite subsets of tilings are patches, and tilings are sometimes thought of as infinite patches.  We denote the set of all patches from a prototile set $\pset$ by $\pset^*$.

Given a patch or tiling $Q$ and $x \in \R^2$, the set $Q+x:=\{t+x \mid t \in Q\}$ is also a patch or tiling.  Like tiles, patches and tilings are called {\em equivalent} if they are translations of the same patch or tiling.

\begin{definition}
A tiling $T$ has {\em finite local complexity (FLC)} if the set of all two-tile patches appearing in $T$ is finite up to equivalence. A tiling is \emph{nonperiodic} if $T+x=T$ implies $x=0$.
\end{definition}

If Q is a patch of tiles or a tiling and $S$ is a subset of $\R^2$, we define the patch
\[
[S]^Q:=\{q \in Q \mid \supp(q) \cap S \neq \varnothing\}.
\]
We note that if $Q$ is a tiling, the support of $[S]^Q$ contains $S$.  Sometimes we may abuse notation and put a tile or patch in place of $S$, in which case it is understood to mean the $Q$-patch intersecting its support.

A very important form of equivalence between tilings with (potentially) different prototile sets is mutual local derivability.

\begin{definition} We say a tiling $T'$ is {\em locally derivable} from a tiling $T$ if there is an $R>0$ such that if $[B_R(x)]^T = [B_R(y)]^T+z$ then $[x]^{T'}=[y]^{T'}+z$.  If both $T$ and $T'$ are locally derivable from each other then we say they are {\em mutually locally derivable (MLD)}.
\end{definition}

{\em Remark on notation and assumptions.}  Throughout this paper, unless otherwise noted, reference to a tiling $T$ implies the presence of a finite prototile set $\pset$ that will only be mentioned explicitly if there is danger of confusion. Tilings will always be assumed to have finite local complexity.

\subsection{Tiling substitutions and self-similar tilings}

\begin{definition}
A function $\omega: \pset \to \pset^*$ is called a {\em tiling substitution} if there exists $\lambda > 1$ such that for every $\ptile \in \pset$, 
\[
\lambda \supp(\ptile) = \supp(\omega(\ptile)).
\]
In this case $\lambda$ is called the {\em expansion factor} of the substitution. \label{subs_def}
\end{definition}

It is natural to extend the substitution $\omega$ to tiles, patches, and tilings.  The substitution of a tile $t=\ptile + x$, for $\ptile \in \pset$ and $x \in \R^2$, is the patch $\omega(t):=\omega(p) + \lambda x$.  The substitution of a patch or tiling is the substitution applied to each of its tiles. Substitution rules naturally give rise to tilings by constructing a self-similar tiling as described in \cite[p.13]{Lorenzo.book}.

An alternative approach to substitution tilings is to consider a self-similar tiling, and then find the substitution $\omega$ it determines. In order to define a self-similar tiling we require some notation.  The boundary of a tiling $T$, denoted $\partial T$, is the subset of $\R^2$ given by the boundaries of all the supports of tiles in $T$, i.e. $\partial T := \bigcup_{t \in T} \partial(\supp(t))$.

\begin{definition}
A tiling $T$ is \emph{self-similar} if
\begin{enumerate}
\item there exists some $\lambda > 1$ such that $\lambda \partial T \subset \partial T$; and
\item if $t_1,t_2 \in T$ are translationally equivalent, then the patches enclosed by $\lambda \supp(t_1)$ and $\lambda \supp(t_2)$ are translationally equivalent.\label{sstcond2}
\end{enumerate}
\end{definition}

The central patch of a self-similar tiling for the 2DTM substitution is shown in Figure \ref{TM-self-similar} with the origin of $\R^2$ marked at the center. It is routine to check that the same patch sits at the centre of a substituted version of this patch, and so on. Thus, we can extend this patch to a tiling by substituting an infinite number of times. Since the (ever expanding) central patch is always invariant under subsequent substitutions, we obtain a self-similar tiling of the plane. We call this particular tiling a self-similar version of the 2DTM tiling.
\begin{figure}[ht]
\includegraphics[width=4cm]{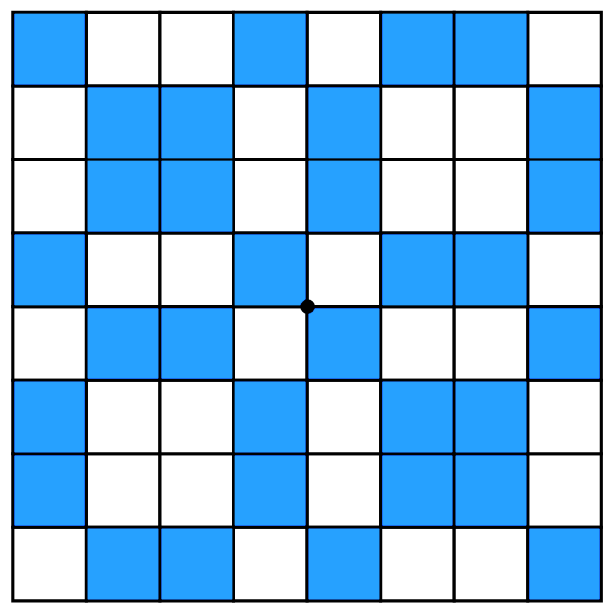}
\caption{A portion of a self-similar 2DTM tiling, with the origin at the center.}
 \label{TM-self-similar}
\end{figure}

To extract $\omega$ from a self-similar tiling $T$, for each prototile $\ptile$ find any equivalent tile $t = \ptile + x \in T$, then define $\omega(\ptile)$ to be the patch with support $\lambda \supp(t)$ translated by $-\lambda x$. For the self-similar version of the 2DTM tiling, notice that the white prototile is sitting in the unit square in the first quadrant, its substitution is sitting at the $4 \times 4$ square in the first quadrant, its second substitution will be the $16 \times 16$ square in the first quadrant, and so on.

\section{The combinatorics of tilings}\label{comb_defn_sec}

Key to our methodology is the connection between the combinatorics of tilings and their geometry. In this section we collect the definitions necessary for our construction.

\subsection{Combinatorial and geometric graphs}
A \emph{combinatorial graph} $K$ consists of a set $V(K)$, whose elements we call {\em vertices}, and a set $E(K)$, whose elements we call {\em edges}.   Each edge in $E(K)$ is defined to be an unordered pair of vertices; we call these vertices the {\em endpoints} of that edge.  
The {\em degree} of a vertex is the number of edges for which it is an endpoint, and a  \emph{dangling} vertex is a vertex of degree $1$.  In this paper we will never need to consider vertices with degree 0.

Following Gross and Tucker \cite{GT}, we define the {\em topological realization} $\widetilde K$  of a combinatorial graph $K$ as follows.  Each edge of $K$ is identified with a copy of $[0,1]$, where $0$ represents one endpoint vertex of the edge and $1$ represents the other.  
Whenever two edges share a common vertex we identify the appropriate endpoints of their copies of $[0,1]$.  The end result is a topological space $\widetilde K$ that encodes the combinatorics of $K$.

It is essential for our purposes to visualize combinatorial graphs in the plane by their topological realizations.  Intuitively, we to do this by first choosing a point set in the plane in one-to-one correspondence with $V(K)$.  Then any pair of points whose associated vertices  make up an edge are connected by a Jordan arc.

\begin{definition}\label{equiv_graphs}
A  \emph{geometric graph} $G$ is an embedding of the topological realization $\widetilde{K}$ of a combinatorial graph $K$ into the plane.  Let $\iota_G$ denote the embedding map. We say that two geometric graphs $G,H$ are \emph{equivalent} if they are embeddings of the same combinatorial graph, and write $G \sim H$. 
\end{definition}  

Combinatorial graphs that can be embedded in the plane are often called \emph{planar} or \emph{plane graphs}, and we call their embeddings geometric graphs. The combinatorial graph that gives rise to a given connected geometric graph $G$ is not unique because degree two vertices cannot be detected.  However, there is always a unique combinatorial graph $K$ with no vertices of degree two for which $G = \iota_G(\widetilde K)$.  This makes the following definition of the edge set $E(G)$ quite natural.

\begin{definition}
Let $G$ be a geometric graph and $K$ a combinatorial graph with no vertices of degree two such that $\iota_G(\widetilde{K})=G$, where $\widetilde{K}$  is the topological realization of $K$.  An {\em edge (resp. vertex)} of $G$ is the image under $\iota_G$ of the topological realization of an edge (resp. vertex) in $K$.
\end{definition}

\subsection{Tilings as geometric graphs} \label{subsec:tilinggraphs}

A tiling $T$ in the plane gives rise to a canonical geometric graph: each point at which three or more tiles meet represents a vertex, and any arc along which two tiles meet represents an edge.   
 Although this combinatorial graph is quite natural, it can cause problems because ideally, prototiles would have well-defined edges and vertices that carry throughout $T$.  However, there are the tilings, such as the chair tiling, whose prototiles will not have a well defined vertex set unless either the prototile set or the vertex set is enlarged to account for them.  Thus we choose to make a more subtle definition for the graph of a tiling.

Suppose $T$ is a tiling with finite local complexity and prototile set $\pset$ with boundary $\partial \pset$.   We form an equivalence relation on $\partial \pset$ by $a \sim_T b$ if there exists $x,y \in \R^2$ and $p,q \in \pset$ such that $p+x, q+y \in T$ and $a+x = b+y$.  A set $F \subset \partial\pset$ is \emph{$T$-consistent} if it is a union of equivalence classes of $\sim_T$.  

\begin{definition} A {\em natural vertex} of $T$ is a point at which three or more tiles meet; the natural vertices of $\pset$ are their representatives on the prototile set. The {\em vertex set of $\pset$} is the $T$-consistent set of points generated by the natural vertices of $\pset$.  An {\em edge} in $\ptile\in \pset$ is a Jordan arc connecting a pairs of neighboring vertices along the boundary of $\ptile$. 
\end{definition}

In this way the boundary of each prototile is associated with a combinatorial and geometric graph.
In Figure \ref{chair.vertices} we show how the vertex set for a chair tile is determined in this perspective.

\begin{figure}[ht]
	\includegraphics[width=3cm]{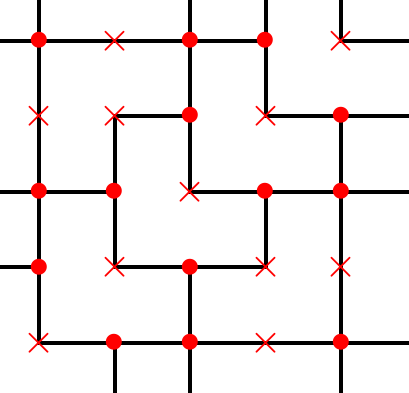}
	\caption{A patch from the third substitution of the chair tiling (see Example \ref{Chair tiling example}). The natural vertices marked with a \tcr{$\bullet$}, and the vertices in the $T$-consistent set generated by the natural vertices are marked with an \tcr{$\times$}.}
    \label{chair.vertices}
\end{figure}

We use this definition of prototile edges and vertices to define a canonical combinatorial graph $K_T$ associated with $T$.  Since every tile in $T$ is a translation of a prototile, the edges and vertices of the tile are inherited from the prototile by translation.  Since the edges and vertices of tiles are embedded on the boundaries of the tiles, they appear in the supports of more than one tile. To define the edge and vertex sets of $K_T$ we include only one vertex and one edge from any given pair of adjacent tiles.  This combinatorial graph is conveniently equipped with the embedding $\iota_T$ that places the vertices and edges where they came from in the first place; the image of $\iota_T$ is $\partial T$, the boundary of $T$. We note that if $T$ and $T'$ are translates of one another, then $K_T$ and $K_{T'}$ are graph isomorphic.

\begin{definition}
A tiling is defined to be {\em edge-to-edge} if any two tile edges intersect either completely, at a common vertex, or not at all.  A tiling is defined to be {\em singly edge-to-edge} provided any two tiles intersect along at most one edge. 
\end{definition}

If $T$ is a FLC tiling whose edges and vertices come from the edges and vertices of $\partial T$, then $T$ is automatically edge-to-edge.  The chair tiling is an instructive example; by adding vertices, as in Figure \ref{chair.vertices}, we have ensured that tile edges always line up.  However, doing so means that chair tilings are no longer singly edge-to-edge.

\subsection{Geometric graphs on prototiles and the tilings they induce}
A key construction used in this work is to make new tilings from a given tiling $T$ by marking all its tiles in a specified way.  This process of turning old tilings into new ones by carving up the prototile set is what is referred to as {\em recomposition} in \cite{GS}. Our method for doing this is by embedding geometric graphs into the supports of the prototiles, then extending to all the tiles of $T$.  

\begin{definition}\label{geometric graph}
A {\em geometric graph $G$ on a prototile set $\pset$} is the disjoint union $\disjoint G_p$ of  finite geometric graphs such that $G_\ptile \subset \supp (\ptile)$ for all $\ptile \in \pset$.   We denote by $K_p$ the underlying combinatorial graph of $G_p$ that has no degree 2 vertices; we call its embedding map $(\iota_G)_p$.
A vertex of $G_p$ that is contained on the boundary of $\supp(p)$ is called a \emph{boundary vertex} of $G$ and a vertex contained in the interior of $\supp(p)$ is called an \emph{interior vertex}.
\end{definition} 

We now introduce a condition ensuring that edges of geometric graphs always meet when we translate prototiles to make a tiling.

\begin{definition} \label{consistent}
Suppose $T$ is a FLC tiling and let $G$ be a geometric graph on its prototile set $\pset$.  We say that $G$ is \emph{$T$-consistent} if 
\begin{enumerate}[(i)]
\item\label{T-consistent graph vertices}  for all $p \in \pset$, $G_p$  only intersects the boundary of $\supp(p)$ at boundary vertices and
\item\label{T-consistent graph bdary} the boundary vertices of $G$ form a $T$-consistent set.\footnote{For readers familiar with the Anderson-Putnam complex \cite{AP} of a tiling space, a $T$-consistent embedding is a geometric graph $G$ on the Anderson-Putnam complex with no dangling vertices.}
\end{enumerate}
\end{definition}

The next definition gives sufficient (but not necessary) conditions on $G$ so that it induces a tiling $T_G$ by recomposition.

\begin{definition} \label{quasi}
Suppose $T$ is a FLC tiling and let $G$ be a $T$-consistent geometric graph on its prototile set $\pset$. We say that $G$ is a {\em quasi-dual graph} if
\begin{enumerate}[(i)]
\item\label{T-consistent graph tree} the graph $G_p$ is a connected tree for all $p \in \pset$, 
\item\label{T-consistent graph edge interior} the interior of each prototile edge contains exactly one boundary vertex of $G$, situated on the interior of that prototile edge, and
\item\label{T-consistent graph danglers} every interior vertex of $G$ has degree at least 3.
\end{enumerate}
If $G$ is a quasi-dual graph such that every prototile contains exactly one interior vertex then we call $G$ a \emph{dual} graph.
\end{definition}

An example of a dual graph for the 2DTM tiling is given on the left hand side of Figure \ref{2DTM_graph_sub} of the introduction. 

A $T$-consistent graph $G$ on $\pset$ extends to a graph in $\R^2$ as follows.  For each $t \in T$, denote by $x_t$ the translation vector of $t$ for which $t = \ptile + x_t$ for its $\ptile \in \pset$. Then we define
\begin{equation}
\Gamma(T,G) = \bigcup_{\ptile \in \pset} \bigcup_{\{t\in T : t \text{ is of type } \ptile\}} \left(G_\ptile+x_t \right) \label{graph_from_prototiles}
\end{equation}
This is a union of Jordan arcs that can also be seen as a geometric graph in $\R^2$.   It has no dangling vertices since $G$ is $T$-consistent; the only way it can fail to be a tiling is if some of the Jordan arcs do not close up to become Jordan curves.

An {\em empty} Jordan curve in $\Gamma(T,G)$ is a Jordan curve that contains no portion of $\Gamma(T,G)$  in its interior. In the following Lemma we define labelled empty Jordan curves to be the tiles of a new tiling induced by a $T$-consistent graph.

\begin{lemma}\label{tiling_from_graph}
Suppose $G$ is a $T$-consistent graph such that each arc in $\Gamma(T,G)$ is part of an empty Jordan curve.   Then $\Gamma(T,G)$ is the boundary of a tiling $T_G$ such that every empty Jordan curve is the boundary of a tile in $T_G$.
\end{lemma}

\begin{proof}
We begin by constructing a prototile set $\pset_G$ which we will use to define the tiling $T_G$.
A label set $\LL$ is defined by
\[
\LL:=\{[J+x]^T : J \text{ is an empty Jordan curve and } J+x \in \Gamma(T,G)\}/ \sim,
\]
where $\sim$ denotes translational equivalence on patches of tiles in $T$. Since $T$ has FLC the set $\LL$ is finite. A set of prototiles $\pset_G$ is defined by taking the closed set bounded by a representative from each translational equivalence class of empty Jordan curves for each label in $\LL$. Then each empty Jordan curve in $\Gamma(T,G)$ uniquely defines a tile that is a translation of a prototile in $\pset_G$. The union of all the tiles defined by empty Jordan curves in $\Gamma(T,G)$ defines a tiling $T_G$, as required.
\end{proof}

If $G$ is a quasi-dual graph on a substitution tiling $T$, then the hypotheses of Lemma \ref{tiling_from_graph} are satisfied and $T_G$ is a tiling. We call tilings arising from quasi-dual graphs \emph{quasi-dual tilings}. Moreover, if $G$ is a dual graph, then $T_G$ is a labelled combinatorial dual of $T$.

\begin{prop}\label{quasi-dual tiling} Suppose $T$ is a tiling with FLC and $G$ is a quasi-dual graph on $\pset$, then the quasi-dual tiling $T_G$ has FLC and is mutually locally derivable to $T$.  Moreover, the vertex patterns in $T$ are in one-to-one correspondence with the tiles in $T_G$.
\end{prop}

\begin{proof}
Let $v$ be a vertex in $T$.  Consider the edges emanating from $v$ in clockwise order; there is a unique path through $\Gamma(T,G)$ connecting each edge to the next by conditions \eqref{T-consistent graph tree} and \eqref{T-consistent graph edge interior} of Definition \ref{quasi}. These edges form the boundary of the tile in $T_G$ that corresponds to $v$.  Moreover, Definition \ref{quasi}\eqref{T-consistent graph tree} guarantees that there can be no additional tiles in $T_G$.

Mutual local derivability follows by the fact that vertex patches in $T$ give rise to tiles in $T_G$ while tiles in $T_G$ specify their corresponding patches in $T$ by their label.
\end{proof}


 
\section{Alternate views of tiling substitutions}\label{alt_view_sub_sec}
 
In this section we describe tiling substitutions using digit sets and digit matrices.  We then use digit sets to define a contraction map on substitution tilings, which turns out to be essential for our construction.  At the end of the section we calculate the digit matrix and contraction map for the chair substitution.  


\subsection{Matrices associated with tiling substitutions.} We follow the method of Moody and Lee \cite[Section 2]{LM}.  Consider a fixed tiling substitution $\omega$ on prototile set $\pset$ with expansion factor $\lambda$ as in Definition \ref{subs_def}.  Suppose $|\pset|=m$ and that when $\pset$ is used as an index set for the rows and columns of a matrix we keep some fixed order.

\begin{definition}
The {\em substitution matrix} of $\omega$ is the $m \times m$ matrix $M$ indexed by $\pset$ whose $(p,q)$-entry is equal to the number of copies of the prototile $q$ occurring in $\omega(p)$. 
\end{definition}

The substitution matrix contains basic information about $\omega$. 
A matrix that keeps more geometric information about the substitution is the \emph{digit matrix} $D = (D_{pq})$, which contains the translation vectors required to implement the tiling substitution. The patch given by $\omega(p)$ is a collection of translates of the prototiles $q\in\pset$, where the number of copies of each $q$ is $M_{pq}$.  The translation vectors for $q$ in $\omega(p)$ are denoted $d_{pq}^1, d_{pq}^2, ..., d_{pq}^{M_{pq}}$ and  we define
the entries of the digit matrix $D$ to be these sets of translation vectors:
 \[
 D_{pq} = 
 \begin{cases}
 \{d_{pq}^k \in \R^2 : k= 1,\dots, M_{pq}\} &\text{if $M_{pq} \neq 0$, and}\\
  \varnothing &\text{ if $M_{pq} = 0$}
 \end{cases}
 \]
 
We will use $D$ to translate both subsets of $\R^2$ (especially geometric graphs on $\pset$) and prototiles $q$ as follows.  For $S \subset \R^2$ and $D_{pq} \neq \varnothing$ define $S + D_{pq} := \bigcup_{k=1}^{M_{pq}} (S + d_{pq}^k)$.  If $D_{pq}= \varnothing$ we take the convention that $S + \varnothing = \varnothing$. For a tile $q$ and $D_{pq} \neq \varnothing$, we define $q + D_{pq}$ to be the set of tiles $\{q+ d_{pq}^k  : k \in 1, 2, ..., M_{pq}\}$, and again use the convention that $q + \varnothing = \varnothing$. 
With this notation we can write
\begin{equation}
\omega(p) = \bigcup_{q \in \pset} \left( q + D_{pq}\right).
\label{omega_with_digits}
\end{equation}

\subsection{Tiling contraction maps}

In this section we will define a tiling contraction map $\rrule$. When restricted to prototiles we want $\rrule$ to be equal to $\lambda^{-1}\omega$, which subdivides each prototile into subtiles rather than inflating and then subdividing. However, we want $\rrule$ to be a more general map and we will see that the domain of $\rrule$ plays an important role in its definition. In particular, we would like $\rrule$ to be a `shrink-and-replace' rule on patches of tiles and on disjoint unions of compact subsets of $\R^2$, though our prototypical application is to geometric graphs embedded in tiles. We will use the notation $\rrule$ regardless of what type of object is being acted on, with the understanding that the output is always the same type as the input.

Before defining $\rrule$ we will define its domain. In order to define the domain of $\rrule$ we need to decide whether we want it to act on disjoint unions of compact subsets of $\R^2$ or on patches of scaled tiles.  For the former, let $H(\R^2)$ denote the set of nonempty compact subsets of $\R^2$ and define the domain of $\rrule$ to be $\domain:=\disjoint H(\R^2)$. The latter is slightly trickier. Given a prototile set $\pset$, consider the set of all nonempty patches of translated prototiles, denoted  $\pset^*$. For any $0 < \kappa \leq 1$ the set $\kappa \pset^*$ is the set of all nonempty patches of prototiles scaled by $\kappa$. When we want $\rrule$ to act on scaled patches of tiles the domain is then $\displaystyle{X_\pset := \bigcup_{0< \kappa \leq 1} \disjoint \kappa\pset^*}$. 


\begin{definition}  The \emph{tiling contraction map} $\rrule$ is defined as follows. Let  $B = \disjointelt B_p$ be an element of $\domain$ or $X_\pset$. For each $p \in \pset$ let
\begin{equation}
\rrule(B)_p =\bigcup_{q \in \pset} \lambda^{-1} \left(B_q + D_{pq} \right),
\label{rrule_def}
\end{equation}
and then $\rrule(B) = \disjointelt \rrule(B)_p$.
\end{definition}

Since $\rrule$ has the same definition regardless of the domain we now illustrate the difference. The first case is if $B \in X$, and then $\rrule(B) \in X$ as well. By comparing equation (\ref{rrule_def}) to equation \eqref{omega_with_digits} we see that if $B=\disjointelt \, \supp (p)$ then we have
\[
\rrule(B)=\disjoint \bigcup_{q \in \pset} \lambda^{-1}(\supp(q) + D_{pq}) = \disjoint \supp(p)=  B.
\]
Thus the disjoint union of prototile supports is fixed under the action of $\rrule$ and is therefore the attractor of the dynamical system $(X,\rrule)$. In the special case that $G$ is a geometric graph on $\pset$, then $\rrule(G)$ is also a geometric graph on $\pset$.  If $G$ is $T$-consistent then $\rrule(G)$ will be $T$-consistent with no dangling vertices in its interior.  Moreover, if $G$ is quasi-dual then $\rrule(G)$ will be connected but will have too many boundary vertices to be quasi-dual itself. In the next section we will be selecting subgraphs of $\rrule(G)$ to define recurrent pairs.

On the other hand, if $B \in X_\pset$, then $\rrule(B)$ may fail to be in $X_\pset$. This can happen when the patches that make up $\rrule(B)_p$ have overlapping interiors. To deal with this annoyance we will restrict our attention to those $B$ for which $\supp(B_p) = \supp(p)$ for all $p \in \pset$, in which case both $\rrule(B)$ and $\rrule^n(B)$ are elements of $X_\pset$ for all $n \in \N$.  With a slight abuse of terminology we write $\rrule^n(\pset)$ for $\rrule^n(\disjointelt \, p)$. Then $\rrule(\pset)_p = \lambda^{-1}\sub(p)$ for all $p \in \pset$.  We call a tile in $\rrule(\pset)_p$ a {\em subtile}; and $\rrule$ can be applied repeatedly so that $\rrule^n(\pset)$ consists of patches of $\pset$-tiles scaled by $\lambda^{-n}$ that we call {\em $n$-subtiles}.

\begin{example}[The Chair tiling]\label{Chair tiling example}
Let $\pset$ be the set of four prototiles with long side length $1$ and the origin in the corner marked by the dot, as depicted in Figure \ref{chair_prototiles}. These are the prototiles of the Chair tiling.
\begin{figure}[ht]
\[
\begin{tikzpicture}[scale=0.8]
\foreach \x in {1,2,3,4}
	{
		\chaird{2*\x}{0}{90*(1-\x)}{p_\x}
	}
\end{tikzpicture}
\hspace{2cm}
\begin{tikzpicture}[scale=0.8]
\chaird{0}{0}{0}{p_1}
\begin{scope}[xshift=3cm]
\chaird{0}{0}{0}{p_1}
\chair{1/2}{1/2}{0}{p_1}
\chair{0}{1}{-90}{p_2}
\chair{1}{0}{-270}{p_4}
\end{scope}

\node (1.5,1/2) {}
	edge[->] node[auto] {} (2.5,1/2);
\end{tikzpicture}
\]
\caption{The $4$ prototiles of the chair tiling and the substitution of prototile $p_1$.}
\label{chair_prototiles}
\end{figure}
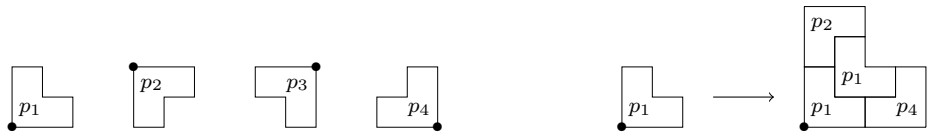
The expansion factor for the chair substitution is $\lambda = 2$, and the substitution of prototile $p_1$ also appears in Figure \ref{chair_prototiles}; the substitution of each of its three rotations are just the rotation of this substitution about the origin.
The digit matrix for the Chair substitution is 
\[
D = 
\begin{pmatrix}
\big\{(0,0),(\frac12, \frac12)\big\} & \{(0,2)\} & \varnothing & \{(2,0)\} \\
\{(0,-2)\} & \big\{(0,0),(\frac12, -\frac12)\big\} & \{(2,0)\} & \varnothing  \\
\varnothing & \{(-2,0)\} & \big\{(0,0),(-\frac12, -\frac12)\big\} &  \{(0,-2)\} \\
\{(-2,0)\} & \varnothing &  \{(0,2)\} & \big\{(0,0),(-\frac12, \frac12)\big\} 
\end{pmatrix}.
\]
Thus, for $B=\sqcup_{i=1}^4 B_i$ in $X$ or $X_\pset$, the tiling contraction map $\rrule$ is given by
\[
\rrule \begin{pmatrix}B_1 \\ B_2 \\ B_3 \\ B_4\end{pmatrix}
= \frac12 \begin{pmatrix}
B_1 \cup \big(B_1 + (\frac12,\frac12) \big) \cup \big(B_2 + (0,2)\big) \cup \big(B_4 + (2,0)\big) \\
\big(B_1 + (0,-2)\big) \cup B_2 \cup \big( B_2 +(\frac12,-\frac12) \big) \cup \big(B_3 + (2,0)\big) \\
\big(B_2 + (-2,0)\big) \cup B_3 \cup \big(B_3 + (-\frac12,\frac12)\big) \cup \big(B_4 + (0,-2)\big) \\
\big(B_1 + (-2,0)\big) \cup \big(B_3 + (0,2)\big) \cup B_4 \cup \big(B_4 + (-\frac12,\frac12)\big)
\end{pmatrix}.
\]
\end{example}

\section{Constructing fractal tilings from self-similar tilings}\label{sec:psi}

In this section we construct fractal substitution tilings. Our method uses a self-similar tiling to construct the fractal tiling explicitly. We use this method rather than standard techniques from fractal geometry because we obtain more information about the relationship between the the original tiling and a fractal realization. In particular, our techniques allow for the possibility of experimentation and we are able to pinpoint a necessary condition for a recurrent pair on a substitution tiling to give rise to a fractal substitution tiling.

Here is a brief overview of the method. Suppose $T$ is a fixed self-similar tiling with prototile set $\pset$ and substitution map $\omega$.  Given a geometric graph $G$ on $\pset$ satisfying the consistency conditions of Definition \ref{consistent}, we let $\tzero:= T_G$ be the tiling defined in Definition \ref{tiling_from_graph} obtained by embedding $G$ in all the tiles of $T$. We then select a geometric graph $S \subset \rrule(G)$ with the same combinatorics as $G$, forming a recurrent pair $(G,S)$.  The map taking $G$ to $S$ acts as a substitution rule on edges, and we define the map $\psi$ to be the extension of the substitution to the boundary of $\tzero$ and obtain a tiling $T^{(1)}:=\psi(T^{(0)})$.  Assuming certain injectivity conditions on the recurrent pair $(G,S)$ and iterating this process leads to a limiting tiling $\Tinf$, which is also a self-similar tiling.


%

\subsection{Recurrent pairs and edge refinement maps}

Recall that two geometric graphs $B$ and $C$ are said to be equivalent, and we write $B \sim C$,  provided there exists a combinatorial graph $K$ and embeddings $\iota_B$ and $\iota_C$ of the topological realization of $K$ such that $\im(\iota_B)=B$ and $\im(\iota_C)=C$.  In this case there is a homeomorphism between $B$ and $C$ given by $\iota_C \circ\iota_B^{-1}$.

\begin{definition}
A pair of $T$-consistent geometric graphs $(G,S)$ on $\pset$ is said to be a \emph{recurrent pair} for $(T,\sub)$ if the following conditions hold:
\begin{enumerate}[(i)]
\item $S \subset \rrule^N(G)$ for some $N \in \N$,
\item for all $p \in \pset$, $G_p \sim S_p$ with homeomorphism denoted by $\widetilde \psi_p : G_p \to S_p$.
\item \label{recurrent iii}$v$ is a boundary vertex of $G_p$ if and only if $\widetilde \psi_p(v)$ is a boundary vertex of $S_p$, in which case both lie in the same edge of $p$.
\end{enumerate}
\end{definition}

Notice also that we can always take $N=1$ in the definition of a recurrent pair by replacing $\rrule$ with $\rrule^N$. So after we find a recurrent pair $(G,S)$ we will typically assume $S \subset \rrule(G)$.

Figure \ref{2DTM_graph_sub} of the introduction shows how a recurrent pair for the 2DTM substitution is constructed. The leftmost images show the prototiles $\alpha$ and $\beta$ with dual graphs $G_\alpha$ and $G_\beta$ inscribed. The disjoint union of $G_\alpha$ and $G_\beta$ is the graph $G$.  The images labelled $\rrule(G)_\alpha$ and $\rrule(G)_\beta$ show the graph $G$ inscribed in each tile of the substitution.  The images labelled $S_\alpha$ and $S_\beta$ show the selection of a graph $S$ that makes $(G,S)$ a recurrent pair for the 2DTM substitution.  

Notice that 
in a recurrent pair, since $S \subset \rrule(G)$, an edge of $S$ is made up of a union of edges from $G$ rescaled by $\lambda^{-1}$,
and $S$ passes through many vertices of $\rrule(G)$ that do not affect the mutual underlying combinatorial graph $K$ since they are of degree 2. There is a very meaningful sense in which we think of $S$ as being the substitution of the graph $G$; describing how to iterate the substitution to produce fractals is our next task.\footnote{We will make this definition with the aid of our self-similar tiling $T$; a different but equivalent formulation can be made using the prototile set $\pset$ only.}

In all that follows, we suppose that when $G$ is drawn in all the tiles of $T$, 
it forms the boundary of a tiling called $T_G$, defined by Lemma \ref{tiling_from_graph}, which we take to be our initial tiling $\tzero$.  Since the boundary graph of $\tzero$ consists of translated elements of $G$ we can
 extend the $\widetilde\psi_p$ maps to a map $\psi: \partial \tzero \to \R^2$ that redraws an edge of $\tzero$ the same way $\widetilde\psi_p$ redrew it in its corresponding prototile. That is, if $z \in \partial \tzero$ is in the support of the tile $t = p + x \in T$ for  $p \in \pset$ and $x\in\R^2$, then $\psi(z) = \widetilde\psi_p(z-x)+x$.   
 
\begin{lemma}
The map $\psi: \partial \tzero \to \R^2$ is a homeomorphism onto its image.
\end{lemma}

\begin{proof}
Since $\psi$ restricted to the interior of any $T$-tile is the translation of a homeomorphism, the only question is what happens on $\partial \tzero \cap \partial T$.  Part (\ref{recurrent iii}) in the definition of a recurrent pair along with the fact that $S$ is $T$-consistent imply that $\psi$ is well-defined and continuous.  The fact that $G$ is $T$-consistent implies that $\psi$ is injective. Thus $\psi$ is a homeomorphism.
\end{proof}

We define the tiling $\tone$ to be the tiling with boundary $\psi(\partial \tzero)$; its tiles are given by the empty Jordan curves in this graph and we allow a tile in $\tone$ to inherit the label of its corresponding tile in $\tzero$.
It should be noted that $\partial \tone=\psi(\partial \tzero)$ is the boundary of the tiling $T_S$ induced by $S$ on $T$.
Like $\tzero$, the tiling $T^{(1)}$ is typically not self similar but is pseudo-self-similar \cite{Me.Boris} and the two tilings are combinatorially equivalent and mutually locally derivable.

The set $\psi(\partial \tzero) = \partial (T^{(1)})$ is not contained in $\partial \tzero$, but because it is induced from the self-similar tiling $T$ it is contained in $\lambda^{-1}(\partial \tzero)$.
Thus if we wish to apply $\psi$ again, we must do so on $\lambda \psi(\partial \tzero)$.  Taking this into account we define $\psi^{(n)}: \partial \tzero \to \R^2$ by
\begin{equation}
\psi^{(n)}(z)=\lambda^{-n}(\lambda \psi)^n(z)
\end{equation}
That is, the edges in $\partial \tzero$ are redrawn and then inflated by $\lambda$ so that the result is again in $\partial \tzero$, after which they are redrawn and inflated again until this has happened $n$ times.  Rescaling the result by $\lambda^{-n}$ brings it back to the original scale.  The map $(\lambda \psi)^n$ can be thought of as a kind of `inflate-and-subdivide' rule for edges of $G$, taking an edge in $\partial \tzero$ to a sequence of edges in $\partial \tzero$ of length approximately $\lambda^n$ times as long.

The map $\psin$ is  a homeomorphism for any finite $n$ and we define $\tn$ to be the tiling with boundary graph $\psin(\partial \tzero)$ whose tiles inherit the labels of their corresponding tiles in $\tzero$:
\begin{definition}
For each tile $t \in \tzero$ we define the tile $t^{(n)}$ in $\tn$ as being supported by the set enclosed by $\psin(\partial t)$ and carrying the label of $t$.
\end{definition}

We define $\psi^{(\infty)}$ to be the pointwise limit of $\psi^{(n)}$, whose existence is proved in the following lemma.
\begin{lemma}
The sequence $(\psi^{(n)})_{n =1}^\infty$ is uniformly Cauchy. Hence the pointwise limit, denoted $\psi^{(\infty)}$, is continuous.
\end{lemma}
\begin{proof}
Fix $\epsilon > 0$. Let $C = \max_{t \in T} \diam t$, and choose $N \in \N$ such that $C\lambda^{-N} < \epsilon$. Fix $m,n > N$, and assume without loss of generality that $m < n$. For $z \in \partial T$ 
we have
\begin{align*}
|\psi^{(m)}(z) - \psi^{(n)}(z) | &= |\lambda^{-m}(\lambda\psi)^m(z) - \lambda^{-n}(\lambda\psi)^n(z)|\\
&= \lambda^{-m}|(\psi\lambda)^m(z) - \lambda^{-(n-m)}(\lambda\psi)^{n-m}(\lambda\psi)^m(z)|\\
&= \lambda^{-m}|y - \psi^{(n-m)}(y)| \quad \left(\text{letting } y = (\lambda\psi)^m(z)\right)\\
&\leq \lambda^{-m} C \quad \left(\text{since } y \text{ and } \psi^{(n-m)}(y) \text{ are in the same tile of } T\right) \\
&< \epsilon.
\end{align*}
So $(\psi^{(n)})_{n =1}^\infty$ is uniformly Cauchy, and $\psi^{(\infty)}$ is continuous.
\end{proof}


\subsection{The limiting fractal tiling.}
When $\psiinf$ is injective its image forms the boundary of a tiling which we denote $\Tinf$ and which we will show is self-similar.   Under the condition of injectivity we make the formal definition of a fractal realization.


\begin{definition}
For each tile $t \in \tzero$ we define the tile $t^{(\infty)}$ in $\Tinf$ as being supported by the set enclosed by $\psiinf(\partial t)$ and carrying the label of $t$.  
 We call $\Tinf$ a {\em fractal realization} of $\tzero$ (or of $G$).
\label{definition_tinfty}
\end{definition}

We collect a few observations about $\psin$, $\psiinf$, $\tn$ and $\Tinf$ in the following lemma.

\begin{lemma}\label{MLD_lemma}
When $\psiinf$ is injective,
\begin{enumerate}
\item $\psi^{(1)} = \psi$
\item If $z \in \partial \tzero$ and $z \in \supp(t)$ for a $T$-tile $t$, then $\psin(z) \in \supp(t)$ for all $n$.
\item\label{MLD_part} $T$, $\tzero$, $\tn$, and $\Tinf$ are MLD for all $n \in \N$, and
\item For any $z \in \partial \tzero$ and $n \in \N$, \begin{equation}
\lambda^n \psiinf(z) = \psiinf((\lambda\psi)^n(z)). \label{tinf_invariant_lambda}
\end{equation}
\end{enumerate}
\label{psifacts}
\end{lemma}

\begin{proof}   The first part follows directly from the definition of $\psik$.  The second follows from the construction of $\psin$ by induction, since it is true for $\psi = \psi^{(1)}$.
The fact that $\tzero, \tn,$ and $\Tinf$ are MLD is immediate from the definitions.  Since every tile in $\tzero$ corresponds to a patch in $T$ we see that $\tzero$ and $T$ are MLD.
Part \eqref{MLD_part} can be seen by considering only $k > n$ in the following:
\begin{align*}
\lambda^n \psiinf(z) &= \lim_{k \to \infty} \lambda^{n-k}(\lambda \psi)^k(z)\\
&=\lim_{k \to \infty} \lambda^{n-k}(\lambda \psi)^{k-n}((\lambda \psi)^n(z))\\
&=\psiinf((\lambda\psi)^n(z)) \qedhere
\end{align*}
\end{proof}

\begin{thm}\label{prop:T infty is a tiling}
Let $T$ be self-similar tiling with expansion factor $\lambda$, and let $(G,S)$ be a recurrent pair for $T$. If $\psi^{(\infty)}$ is injective then $T^{(\infty)}$ is a self-similar tiling with expansion factor $\lambda$ and finite local complexity.
\end{thm}
\begin{proof}
To show $\Tinf$ has finite local complexity we show that it has finitely many types of $2$-tile patches up to translation. Since $\psi$ is injective, patches in $\Tinf$ are in one-to-one correspondence with patches in $\tzero$. The number of two-tile patches in $\tzero$ is governed by the number of types of patches in the original self-similar tiling $T$ that intersect a two-tile patch in $\tzero$.  Since $T$ is FLC, this number is finite.  Hence both $\tzero$ and $\Tinf$ (and all other $\tn$) have FLC.

To prove $T^{(\infty)}$ is self-similar we show that its boundary is invariant under scaling by $\lambda$ and that equivalent tiles inflate to equivalent patches.  The former follows directly from Lemma \ref{psifacts}, since if $x=\psiinf(z) \in \partial \Tinf$,  $\lambda x = \psiinf((\lambda \psi)(z)) \in \partial \Tinf$, as desired.  To prove the latter suppose $t_1$ and $t_2$ are equivalent tiles in $\tzero$ so that $t_1^{(\infty)}$ and $t_2^{(\infty)}$ are equivalent tiles in $\Tinf$ . 
Denote by $[t_1]^{T}$ and $[t_2]^T$ the patches in $T$ that intersect the supports of $t_1$ and $t_2$, respectively; these are also equivalent, by definition, since $t_1$ and $t_2$ carry the same label.  Since $T$ is self-similar the $T$-patches in $\lambda \supp [t_1]^{T}$ and $\lambda \supp [t_2]^{T}$ 
are equivalent.  This means that the $\tzero$- and $\Tinf$-patches contained completely inside them are equivalent as well. Since $\supp(t_1^{(\infty)}) $ is contained in  $[t_1]^{T}$ the same way $\supp(t_2^{(\infty)})$ is contained in $[t_2]^{T}$, $\lambda \supp(t_1^{(\infty)})$ and $\lambda \supp(t_2^{(\infty)})$ support equivalent $\Tinf$-patches, as required.
\end{proof} 

Under these circumstances we will let $\pinf$ denote the set of prototiles for the self-similar tiling $\Tinf$, and let $\subinf$ denote its substitution map.

\section{The existence of fractal dual and quasi-dual tilings}\label{sec:exist}

The process described in the previous section can be used to find recurrent pairs with injective $\psiinf$ maps by experimentation.  This section contains technical results related to the existence of such pairs in general.

In the first part of this section we provide sufficient conditions on a recurrent pair $(G,S)$ for the map $\psiinf$ to be injective.  These conditions result in fractal tilings that are  quasi-dual to their original self-similar tilings.
In the second part of the section we show that every singly edge-to-edge self-similar tiling has a recurrent pair satisfying the injectivity conditions.
In the third part of this section we prove that if a singly edge-to-edge self-similar tiling has convex prototiles, then its labelled combinatorial dual has a fractal realization.

\subsection{Conditions under which $\psiinf$ is injective.}

Next we introduce conditions that ensure that the combinatorics of $\psiinf(\partial \tzero)$ are identical to the combinatorics of $\tzero$.  One way this can fail if the interior of an edge in $\psin(\partial \tzero)$ approaches the boundary of a tile in the original self-similar tiling $T$ as $n \to \infty$.  If $\psiinf(z) \in \partial T$ and $z \notin \partial T$, it can happen that an additional tile (or tiles) can arise in $\Tinf$ either by two edges coming together that were apart in $\tzero$ or by an edge doubling back on itself. Infinitely many arbitrarily small tiles are also a danger in this situation.  We can avoid these troubles by keeping $\psin(\partial \tzero)$ away from the boundary of $T$ except at $T$-consistent boundary vertices.

Another way the combinatorics of $\psiinf(\partial\tzero)$ can  fail to be those of $\partial \tzero$ is if $\psin(z)$ approaches the vertex of a tile in $T$ as $n \to \infty$.  In this situation there may still be a perfectly good self-similar tiling $\Tinf$, but we've lost control of its combinatorics and its substitution rule is not guaranteed to be border forcing.

Before we give the injectivity conditions we set some notation.  In what follows, for a real number $k>0$ and any tiling $T$, the notation $kT$ represents the tiling obtained by rescaling all the tiles of $T$ by a factor of $k$ and keeping the labels the same.  Recall that if $B$ is any subset of $\R^2$ or patch of tiles and $Q$ is a tiling or patch the notation $[B]^Q$ represents the patch in $Q$ whose support intersects $B$.  In particular, for $p \in \pset$ note that $[B]^{\rrule^n(\pset)_p}$ is the patch of $n$-subtiles of $\rrule^n(\pset)$ in $p$ that intersect $B$.

\begin{definition}[Injectivity conditions]\label{injectivity_conditions}
Let $S$ be a geometric graph on $\pset$.  If there is an $N \in \N$ for which the following four conditions hold, we say that $S$ {\em satisfies the injectivity conditions for $N$-subtiles}.
\begin{enumerate}
\renewcommand{\theenumi}{I{\arabic{enumi}}}
\item\label{pc1} For any $p \in \pset$ and edges $e \neq f$ in $S_p$, the patches $[e]^{\rrule^N(\pset)_p}$ and $[f]^{\rrule^N(\pset)_p}$ have intersecting interiors if and only if $e$ and $f$ share a common vertex $v$;
\item\label{pc2} If $e \neq f \in S_p$ share a common vertex $v$ then $[e]^{\rrule^N(\pset)_p} \cap [f]^{\rrule^N(\pset)_p}=[v]^{\rrule^N(\pset)_p}$, the single subtile containing $v$, and this subtile is contained in the interior of $\supp(p)$;
\item\label{pc3} For each $p \in \pset$, $[S_p]^{\rrule^N(\pset)_p}$ does not contain any vertex of $p$ in its support; and
\item\label{pc4} For $e$ in $S_p$,  the support of $[e]^{\rrule^N(\pset)_p}$ intersects the boundary of $p$ if and only if $e$ does, in which case  $\supp[e]^{\rrule^N(\pset)_p} \cap \partial p$ is connected.
\end{enumerate}
\end{definition}

Notice that if $(G,S)$ is a recurrent pair, $G$ is a quasi-dual graph, and $S$ satisfies the injectivity conditions for $N$-subtiles,  then condition (\ref{pc3}) along with the definition of a recurrent pair implies that $S$ is also a quasi-dual graph.  In general these conditions imply that the interior of the patch of subtiles that $S$ runs through retracts to a geometric graph that has nearly the same combinatorics as $S$.   If the combinatorics differ it is at interior vertices of $S$, several of which could be collapsed into a single vertex in the retraction.

Given a recurrent pair $(G, S)$ we make two assumptions on the embeddings $\iota_G$ and $\iota_S$, which can be considered ``without loss of generality".  The first is that $G$ is a piecewise linear graph.  If for some reason it is not, it can be redrawn as one: the topological realization of the combinatorial graph $K$ corresponding to the geometric graph $G$ can always be embedded into $\supp (\pset)$ in a piecewise linear fashion. We assume, then,  that such an embedding has been chosen for $G$.  The second is on the embedding $\iota_S$ of $S$, which we can control since $S$, being a subgraph of $\rrule(G)$, is also a piecewise linear graph.  Letting $B \subset e \in E(G)$ or $E(S)$, denote by $|B|$ the arc length of the smallest sub-arc of $e$ containing $B$.  In this case we assume that $\iota_S$ has been chosen so that for any $e \in E(G)$ there is a $K_e$ such that for all $B \subset e$ we have $|\iota_S\circ\iota_G^{-1}(B)| = K_e|B|$.  In other words, $\widetilde \psi$ is a {\em piecewise constant-speed parameterization} taking $G$ to $S$. 

One further technical note is on the edge set of $\partial(\tzero)$, which can be thought of in two different ways.  One way is to consider the vertex set endowed by the tiling $\tzero$, and another is to allow $\partial(\tzero)$ to inherit the edge and vertex sets of $G$.  The latter has extra vertices wherever the boundary of $\tzero$ intersects the boundary of $T$. In all the proofs in this section we consider an edge of $\tzero$ to mean a copy of an edge in $G$.

\begin{prop}
Suppose $(G,S)$ is a recurrent pair such that $G$ is quasi-dual and the parameterization $\widetilde \psi: G\to S$ is  piecewise constant-speed.  If $S$ satisfies the injectivity conditions on $N$-subtiles for some $N$, then $\psiinf: \partial \tzero \to \partial \Tinf$ is injective.
\label{thm: psi infinity is injective}
\end{prop}

We prove Proposition~\ref{thm: psi infinity is injective} using a sequence of lemmas, each of which implicitly has the assumptions of Proposition \ref{thm: psi infinity is injective}.   For convenience of notation we will assume $N = 1$; the general case can be surmised by replacing the substitution rule $\sub$ by $\sub^N$, which changes the expansion factor to $\lambda^N$ and the tiling contraction map to $\rrule^N$.

The injectivity conditions (\ref{pc1})-(\ref{pc4}) are designed to keep $ \psi^{(n)}(e)$  away from the boundaries of both of $\tzero$ and $T$ except where the combinatorics of $\partial T$ explicitly require it, while the constant-speed assumption keeps edges from collapsing onto themselves in the limit.

A key observation in everything that follows is that if  $e$ is an edge in $\partial \tzero$ for which $e = e_p + z$ for some edge $e_p \in G_p$ and $z \in \R^2$, then $[\psi(e)]^{\lambda^{-1} T} = [\widetilde \psi(e_p)]^{\rrule(\pset)_p} + z$  by the self-similarity of $T$.  Thus edges is $\psi(\partial \tzero)$ follow the same injectivity conditions, when intersected with $\lambda^{-1}T$, as edges in $S$ do when intersected with $\rrule(\pset)$.

\begin{lemma}\label{lem:patch boundary}
Suppose that $e \in E(\partial \tzero)$ and $B$ is an edge in $\partial T$ such that $B \cap e = \varnothing$. Then $[\lambda\psi(e)]^{T} \cap \lambda B = \varnothing$.
\end{lemma}
\begin{proof}
If $e$ has no vertex in $\partial T$ then by (\ref{pc4}),  $[\psi(e)]^{\lambda^{-1}T} \cap B = \varnothing$.  If $e$ has a vertex in $\partial T$ then by (\ref{pc3}) and (\ref{pc4}), $[\psi(e)]^{\lambda^{-1}T}$ must intersect $\partial T$ only on the edge that $e$ does. However,  $B \cap e =\varnothing$, and $[\lambda\psi(e)]^{T} \cap \lambda B = \varnothing$.
\end{proof}

\begin{lemma}\label{lem:disjt edges}
Suppose that $x,y \in \partial \tzero$ and that there exists $n \in \N$ such that $(\lambda\psi)^n(x)$ and $(\lambda \psi)^n(y)$ are on disjoint edges in $\partial \tzero$. Then $\psi^{(\infty)}(x) \neq \psi^{(\infty)}(y)$.
\end{lemma}
\begin{proof}
Denote by $e_x$ and $e_y$ the edges in $\partial \tzero$ containing $x$ and $y$, respectively. If $x$ and $y$ are in distinct tiles in $T$, then even if those tiles share a common boundary segment, that segment cannot intersect both $e_x$ and $e_y$ without violating either the fact that $G$ is quasi-dual or that $e_x \cap e_y = \varnothing$.  Thus, in this case, it follows that $\psi^{(\infty)}(x) \neq \psi^{(\infty)}(y)$.

Next we consider the case where $e_x$ and $e_y$ are in the same $T$-tile.  Since $\lambda\psi$ is injective, we have that $(\lambda\psi (e_x)) \cap (\lambda\psi (e_y)) = \varnothing$. Then~\eqref{pc1} implies that $\int(\supp[\lambda\psi(e_x)]^{T}) \cap \int(\supp[\lambda\psi(e_y)]^{T}) = \varnothing$. If $\supp[\lambda\psi(e_x)]^{T} \cap \supp[\lambda\psi(e_y)]^{T} = \varnothing$, then we are done since $\psiinf(x) \in \supp\lambda^{-1}[\lambda\psi(e_x)]^{T}$ and $\psiinf(y) \in \supp\lambda^{-1}[\lambda\psi(e_y)]^{T}$.
If not, then $[\lambda\psi(e_x)]^{T}$ and $[\lambda\psi(e_y)]^{T}$ share some part $B$ of their boundary.  Since $G$ is quasi-dual, neither $\lambda\psi (e_x)$ nor $\lambda\psi (e_y)$ intersect $B$, for otherwise $\lambda\psi(e_x)$ would pass through $B$ into the interior of $[\lambda\psi(e_y)]^{T}$ or vice versa.
By Lemma~\ref{lem:patch boundary}, we have that $[(\lambda\psi)^2(e_x)]^{T} \cap \lambda B = \varnothing$. So there exist $t_x, t_y \in T$ satisfying $(\lambda\psi)^2(x) \in t_x$, $(\lambda\psi)^2(y) \in t_y$ and $t_x \cap t_y = \varnothing$. Since $\lambda^2\psi^{(\infty)}(x) = \psi^{(\infty)} ((\lambda\psi)^2(x))\in \supp(t_x)$  and $\lambda^2\psi^{(\infty)}(y) = \psi^{(\infty)} ((\lambda\psi)^2(y)) \in \supp (t_y)$, we have $\psi^{(\infty)}(x) \neq \psi^{(\infty)}(y)$.

Now suppose that there exists $n \in \N$ such that $(\lambda\psi)^n(x)$ and $(\lambda \psi)^n(y)$ are on disjoint edges in $E(\partial \tzero)$. Then by the previous argument, $ \psi^{(\infty)} ((\lambda \psi)^n(x)) \neq \psi^{(\infty)} ((\lambda \psi)^n(y)) $, making $\lambda^n \psi^{(\infty)}(x) \neq \lambda^n \psi^{(\infty)}(y)$, and in turn $\psi^{(\infty)}(x) \neq \psi^{(\infty)}(y)$, as desired.
\end{proof}

Lemma \ref{lem:disjt edges} implies that edges that are disjoint in $\tzero$ remain disjoint in $\Tinf$.  The next lemma shows that $\psiinf$ restricted to edges of $\tzero$ is injective, meaning that an edge in $\tzero$ cannot loop back onto itself in $\Tinf$.  This is the part of the proof that uses the assumption that the parameterizations are piecewise constant-speed.

\begin{lemma}\label{lem:inj on edges}
For each $e \in E(\partial \tzero)$, $\psi^{(\infty)}|_e$ is injective.
\end{lemma}
\begin{proof}
Fix $x,y \in e \in E(\partial \tzero)$, and suppose that $\psi^{(\infty)}(x) = \psi^{(\infty)}(y)$. We will show $x = y$. Write $x_n := (\lambda\psi)^n (x)$ and $y_n := (\lambda \psi)^n(y)$. In order for $\psi^{(\infty)}(x) $ and $\psi^{(\infty)}(y)$ to be equal,  Lemma~\ref{lem:disjt edges} implies that for each $n \in \N$, $x_n$ and $y_n$ are either on the same or adjacent edges in $E(\partial \tzero)$.  This breaks into two cases: either they are on the same edge for all $n$ or there is an $N$ such that if $n \ge N$, $x_n$ and $y_n$ are on adjacent edges.

We first suppose that $x_n$ and $y_n$ are on the same edge $e_n$ in $\partial \tzero$ for all $n$. For each $n$, we have $e_{n+1} \subset \lambda\psi(e_n)$, and we let $l_n$ be the arc in $e_n$ for which $\lambda\psi(l_n) = e_{n+1}$. 
Set $\rho = \max_{n \in \N} (|l_n| / |e_n|)$. We can deduce that $\rho < 1$ by injectivity condition (\ref{pc2}) because for every $e \in E(\partial T)$, $\lambda\psi(e)$ is a union of at least $2$ edges in $E(\partial T)$. 
The constant-speed assumption on $\widetilde \psi$ extends to guarantee that $\lambda \psi$ is piecewise constant-speed as well.

%

For each $n \in \N$ and $k \le n$ we know that  $(\lambda \psi)^{-(n-k)}(e_n)$ and $(\lambda \psi)^{-(n-k)}(l_n)$ are in the same edge in $\partial \tzero$ and the constant-speed assumption implies that $$\frac{|(\lambda\psi)^{-n}(l_n)|}{|(\lambda\psi)^{-n}(e_n)|}=\frac{|(\lambda\psi)^{-(n-1)}(l_n)|}{|(\lambda\psi)^{-(n-1)}(e_n)|} = \cdots = \frac{|l_n|}{|e_n|}\le \rho.$$

Since $x, y \in (\lambda \psi)^{-n}(l_n)$ for all $n$, we have that
\begin{align*}
|x - y| \leq |(\lambda\psi)^{-n}(l_n)| &\leq \rho |(\lambda\psi)^{-n}(e_n)| = \rho |(\lambda\psi)^{-(n-1)}(l_{n-1})|\\
&\leq \rho^2 |(\lambda\psi)^{-(n-1)}(e_{n-1})| = \rho^2|(\lambda\psi)^{-(n-2)}(l_{n-2})| \leq \cdots \leq \rho^n|e_0|
\end{align*}
Since this is true for all $n$ and $\rho$ is strictly less than one we have shown $x = y$ when $x_n$ and $y_n$ are on the same edge of $\partial \tzero$ for all $n$.

%

Now suppose there is an $N$ such that if $n \ge N$ then $x_n$ and $y_n$ are on adjacent edges that share a vertex $v_n$.
Since $\psi$ is continuous, the edges containing $(\lambda \psi)^{n}(x_N)$ and $(\lambda \psi)^{n}(y_N)$ share the vertex $(\lambda \psi)^n (v_N)$ for all $n \in \N$. Then 
\[
\| (\lambda\psi)^{n}(x_N) - (\lambda\psi)^{n}(v_N) \| \leq \max_{t \in T} \diam (t)
\]
for all $n \in \N$. Multiplying both sides by $\lambda^{-n}$ then gives 
\[
\| \psi^{(n)}((\lambda \psi)^N(x)) - \psi^{(n)}(v_N)\| \leq \lambda^{-n} \max_{t\in T} \diam (t)
\]
for all $n \in \N$. This implies that $\psiinf(x_N) = \psiinf(v_N)$, and since $x_N$ and $v_N$ are always on the same edge the first case in this proof shows that $x_N =v_N$.  This argument shows that also $y_N = v_N$.  Since $(\lambda \psi)^N$ is injective this implies that $x = y$ and we have shown that $\psiinf$ is injective on edges of $\partial \tzero$.
\end{proof}


\begin{proof}[Proof of Proposition~\ref{thm: psi infinity is injective}]
Fix $x,y \in \partial T$ and suppose that $x \neq y$. If $\psin(x)$ and $\psin(y)$ are on the same or on disjoint edges for some $n$, then Lemmas \ref{lem:disjt edges} and \ref{lem:inj on edges} imply that $\psi^{(\infty)}(x) \neq \psi^{(\infty)}(y)$.  The only other option is that $\psin(x)$ and $\psin(y)$ are on adjacent edges for all $n$, but the proof of Lemma~\ref{lem:inj on edges} shows that this cannot happen if $x \neq y$. 
Thus $\psi^{(\infty)}$ is injective.
%
%
\end{proof}

\subsection{Existence of recurrent pairs with injective edge refinements}\label{subsec:exist}

In this section we prove that every singly edge-to-edge, primitive self-similar tiling $T$ with finite local complexity has a recurrent pair $(G,S)$ satisfying the conditions of Proposition \ref{thm: psi infinity is injective}. In this case the map $\psiinf$ is injective and the tiling $T$ gives rise to a fractal substitution tiling $\Tinf$. The construction of the recurrent pair often requires multiple iterations of the substitution $\sub$, so the expansion factor of $\Tinf$ is always a power of the expansion factor of $T$. 

\begin{thm}\label{thm:existence}
Suppose $T$ is a singly edge-to-edge, primitive self-similar tiling with FLC. Let $T$ have substitution $\omega$ and prototile set $\pset$.  Then $T$ has a recurrent pair $(G,S)$ satisfying the conditions of Proposition \ref{thm: psi infinity is injective}, and hence $T$ has a fractal realization.
\end{thm}


The proof of the Theorem \ref{thm:existence} is an algorithm for constructing the recurrent pair $(G,S)$. The algorithm begins with Lemma~\ref{lem:existdualish} followed by a finite number of iterations of Lemma~\ref{lem:existence iteration}. Lemma~\ref{lem:existdualish} and \ref{lem:existence iteration} are implicitly assumed to have the assumptions of Theorem \ref{thm:existence}. 

Recall that we say that a graph $G$ satisfies the injectivity conditions for $N$-subtiles if conditions \eqref{pc1} - \eqref{pc4} of Definition \ref{injectivity_conditions} are satisfied with respect to the subtiles $\rrule^N(\pset)$. Also recall that we say a graph $G_0$ is $T$-consistent if it satisfies Definition \ref{consistent}.

\begin{lemma}\label{lem:existdualish}
Let $G_0$ be a $T$-consistent dual graph in $\pset$. Then there exists $N \in \N$ and a $T$-consistent quasi-dual graph $G_1 \subset \rrule^N(G_0)$ such that $G_1$ satisfies the injectivity conditions for $N$-subtiles.
\end{lemma}

\begin{proof}
By primitivity, there exists $n \in \N$ such that for each $p \in \pset$, there is a copy of $\ptile$ interior to $\rrule^n(\pset)_{\ptile}$. More precisely, there exists $x \in \R^2$ such that the $n$-subtile $p_0 = \lambda^{-n} p+x$ has support contained in the interior of $\rrule^n(\pset)_{\ptile}$.

Now fix $\ptile \in \pset$ and let $V_p$ denote the vertices of $p$. Since $p_0 \subset \int (p)$, there exists a geometric graph $H_p \subset p \setminus \int (p_0)$ consisting of $|V_p|$ disjoint edges, where the edge associated to $v \in V_p$ has endpoints $v \in p$ and $v^0=\lambda^{-n}v + x \in p_0$. Let $V_p^0$ denote the collection of such vertices. Note that the geometric graph $H_p$ connects the vertices of $p$ with their corresponding vertices $p_0$, and the existence of $H_p$ is guaranteed since $p \setminus p_0$ is an annulus. Let $H$ be the disjoint union $H:=\disjoint H_p$.

Let $M \in \N$ be large enough so that for each $p \in \pset$ there are at least three $M$-subtile edges along each edge of $p$ and three $M$-subtile edges along each edge of $p_0$. Now let $E(H_p)$ denote the set of edges of $H_p$. Set $\varepsilon$ to be less than the minimum of
\begin{align*}
\varepsilon_1 &= \min_{p \in \pset} \{ d(e,f) \mid e\neq f \in E(H_p)\} \text{ and} \\
\varepsilon_2 &= \min_{p \in \pset} \{ d(e,f) \mid e \in E(H_p), f \in \partial p \setminus \partial ([V_p \cup V_p^0]^{\rrule^M(\pset)_p}) \}
\end{align*}
Then since $\lambda > 1$, we can choose $N \geq M$ such that $\lambda^{-N} \max_{p \in \pset} \diam p < \varepsilon / 3$.

Let $O_B=\cup_{p \in \pset}\{\partial p \setminus \partial ([V_p]^{\rrule^M(\pset)_p})\}$ and let $I_B=\partial p_0 \setminus \partial ([V_p^0]^{\rrule^M(\pset)_p})$. 
Choose a $T$-consistent set of vertices contained in the set $O_B \cap \rrule^N(G_0)$ with exactly one vertex for each edge $e \in \pset$, and denote the set of vertices by $O_V$. By our choice of $N$ with respect to $\ep >0$, for every edge $e \in \pset$ there is at least one path in $\rrule^N(G_0)$ connecting $O_V$ with any vertex in $I_B \cap \rrule^N(G_0)$ that does not cross any edge of $H_p$. We choose one such path, which we denote $x_e$, with the additional restriction that once it enters an interior $N$-subtile it continues through interior subtiles until it reaches its destination on the boundary of $p_0$. Note that $x_e$ must connect the edge $e \in E(p)$ to its counterpart $\lambda^{-n}e + x$ in $p_0$, and $[x_e]^{\rrule^N(\pset)_p}$ must intersect the boundary of $p$ in a connected set.  Moreover, for $e \neq f \in p$, we have $[x_e]^{\rrule^N(\pset)_p} \cap [x_f]^{\rrule^N(\pset)_p}=\varnothing$. Let $X_p=\cup_{e \in p} x_e$.  Although the component graphs of  $X = \disjoint X_p$ are disconnected, $X$ satisfies the injectivity conditions for $N$-subtiles.

In order to complete the proof, for each $p \in \pset$, we need to connect the points $X_p \cap I_B \subset \partial p_0$ by a tree in $p_0$. To that end, consider the connected graph $\rrule^{N}(G_0)_{p_0}$, which we can prune until we have a tree connecting the points $X_p \cap I_B$ with no dangling vertices. Let $Y_p$ denote this tree.

Set $G_1 := \disjoint X_p \cup Y_p$. Then, by construction, the graph $G_1$ is quasi-dual and satisfies the injectivity conditions for $N$-subtiles.
\end{proof}

When $G' \subset \rrule^N(G)$ is a geometric graph we have been interested in the patches of $N$-subtiles intersecting the edges of $G'$.  The next lemma states that the paths in $\rrule^N(G')$ running through those same patches are unique when $G$ and $G'$ are both quasi-dual.

\begin{lemma}\label{lem:existence iteration}
Suppose that $G$ and $G' \subset \rrule^N(G)$ are quasi-dual graphs in $\pset$ such that, for some $M \in \N$, $G'$ satisfies the injectivity conditions for $M$-subtitles. Then there is a unique quasi-dual graph $H \subset \rrule^N(G')$ such that $[G'_p]^{\rrule^N(\pset)_p}=[H_p]^{\rrule^N(\pset)_p}$ for all $p \in \pset$, which satisfies the injectivity conditions for $(M+N)$-subtiles.
\end{lemma}

\begin{proof}
For each $p \in \pset$, consider the patch of $N$-subtiles $Q_p:=[G'_p]^{\rrule^N(\pset)_p}$ and boundary $N$-subedges $E_p:=\{e \in \rrule^N(\pset)_p \cap \partial p \mid e \cap G' \neq \varnothing\}$. Since $G'$ is quasi-dual  and $T$ is singly edge-to-edge there is a unique quasi-dual graph $H_p$ with edges in $\rrule^N(G')_p$ through the patch $Q_p$ with dangling vertices exactly on the edges $E_p$. Let $H:=\disjoint H_p$. By construction, $H$ is a quasi-dual graph satisfying $[G'_p]^{\rrule^N(\pset)_p}=[H_p]^{\rrule^N(\pset)_p}$ for all $p \in \pset$.

Since $G'$ satisfies the injectivity conditions for $M$-subtiles and $H \subset \rrule^N(G')$, it follows that $H$ satisfies the injectivity conditions for $(M+N)$-subtiles.
\end{proof}

We can now prove Theorem~\ref{thm:existence}.

\begin{proof}[Proof of Theorem~\ref{thm:existence}]
Let $G_0$ and $G_1$ be the quasi-dual graphs from Lemma \ref{lem:existdualish}. If $G_0 \sim G_1$, the proof is complete with $G=G_0$ and $S=G_1$. If not, apply Lemma~\ref{lem:existence iteration} with $G_0=G$ and $G_1=G'$ to obtain $G_2:=H$. 

If $G_1 \sim G_2$ then we are done with $G=G_1$ and $S=G_2$. If not, we continue iterating this process of applying Lemma~\ref{lem:existence iteration} to the last two graphs in the sequence. We claim that there exists $n \geq 2$ such that $G_n \sim G_{n+1}$. 

Since each prototile has only a finite number of edges and each graph $G_j$ is quasi-dual, the number of interior vertices (with degree greater than two) is bounded by half the number of edges of the prototile. Suppose $G_i \nsim G_{i+1}$, for $i \geq 2$, then there exists $p \in \pset$ such that $(G_{i})_p \nsim (G_{i+1})_p$. Then the construction in Lemma \ref{lem:existence iteration} implies that the number of interior vertices of $(G_{i+1})_p$ is greater than the number of interior vertices of $(G_{i})_p$. Since the number of interior vertices of each prototile is bounded, this process must end. So there exists $n \in \N$ such that $G_n \sim G_{n+1}$, and we let $G:=G_n$ and $S:=G_{n+1}$.

Lemma~\ref{lem:existdualish} shows that $G_1$ satisfies the injectivity conditions for $N$-subtiles. Lemma \ref{lem:existence iteration} implies that each graph $G_i$ in the algorithm satisfies the injectivity conditions for $iN$-subtiles. Thus, every graph in the sequence continues to satisfy the injectivity conditions. In particular, $S \subset \rrule^{N}(G)$ satisfies the injectivity conditions for $(n+1)N$-subtiles. Moreover, since $G_0$ is a dual tiling and $G \subset \rrule^{nN}(G_0)$ and $S \subset \rrule^{(n+1)N}(G_0)$ the map $\widetilde{\psi}$ can be chosen to be piecewise constant-speed. Thus, $(G,S)$ satisfies the conditions of Proposition \ref{thm: psi infinity is injective} and Theorem \ref{thm:existence} is proved.
\end{proof}

\subsection{Self-similar combinatorial dual tilings} \label{subsec:self-similar dual}

The construction in Section \ref{subsec:exist} invites the question of when the $(G,S)$ recurrent pair are both dual graphs. We give a sufficient condition in the following result.

\begin{thm}\label{thm:dual self-similar}
Let $T$ be a singly edge-to-edge, primitive self-similar tiling whose tiles are all convex.  Then the combinatorial dual tiling of $T$ has a self-similar realization.
\end{thm}

\begin{proof}
Let $G$ be a $T$-consistent dual graph $\pset$, embedded such that $G_p$ has exactly one boundary vertex in the interior of each edge of $p$ and there is one interior vertex that is connected by a straight-line edge to each boundary vertex.  
By primitivity we can choose $N$ such that
\begin{enumerate}
\item $\rrule^N(\pset)_p$ contains a tile of type $p$ in its interior for each $p \in \pset$ and
\item there is a $T$-consistent set of boundary vertices in $\rrule^N(G)$ such that if $v$ is in this set and in $p\in \pset$, then 
$\supp[v]^{\rrule^N(\pset)_p}$ does not contain any vertex of $p$.  Denote by $V_B$ such a $T$-consistent set that has exactly one vertex per prototile edge.
\end{enumerate}

We let the vertex set of $S_p$ be $(V_B)_p$ along with a single interior vertex $v_{p,int}$ chosen from the interior of an $N$-subtile $p_0$ of type $p$ lying in the interior of $\rrule^N(\pset)_p$.  We construct the edges of $S_p$ as follows.
Let $l(v)$ be the line connecting the boundary vertex $v \in S_p$ to its counterpart in $p_0$. We take as the edge in $S_p$ a path in $\rrule^N(G)$ through $[l(v)]^{\rrule^N(\pset)_p}$ connecting $v$ to $v_{p,int}$.  

 Notice that if $v \neq w$ are boundary vertices in $p$, by convexity the $\rrule^N(\pset)_p$-patches intersecting $l(v)$ and $l(w)$ cannot intersect except at $p_0$ and perhaps on their boundaries, and thus (\ref{pc1}) and (\ref{pc2}) are satisfied.  $G$ was already dual and thus quasi-dual and (\ref{pc3}) and (\ref{pc4}) follow from the construction and by convexity. The map $\widetilde \psi$ can be chosen to be piecewise constant-speed so that all the conditions of Proposition \ref{thm: psi infinity is injective} hold.
 
 Thus we have a recurrent pair $(G,S)$ for which $T_G = \tzero$ is a labelled dual tiling of $T$. Since $\Tinf$ is combinatorially equivalent to $T_G$ the result follows.
\end{proof}

%
%

\section{Edge contraction map and Hausdorff dimension of $\partial \Tinf$}\label{FractalDimension_sec}

We mentioned earlier that it is possible to see the edge refinement map $\widetilde\psi$ as a substitution rule on prototile supports.  We make this precise here, using the formalism of digit sets and the tiling contraction map used to describe a substitution rule.  For simplicity we assume that $S \subset \rrule(G)$ but all results extend to the general case of $S \subset \rrule^N(G)$.

Once this is done it is possible to see the edges of $\Tinf$, or more precisely their counterparts in $\pset$, as the invariant set of a graph iterated function system.  This system satisfies the strong open set condition, so \cite{MW} implies that the Hausdorff dimension of its attractor (our edges) is the Mauldin-Williams dimension. In this way we obtain a geometric invariant on the set of fractal realizations of a particular tiling.  

\subsection{Edge refinement as a contraction map $\rrule^E$.}\label{sec:edge contraction}
Suppose that $(G,S)$ is a recurrent pair for a substitution $\sub$ with expansion $\lambda$.

Let the edge set of $G$, denoted $E(G)$, be the collection of all edges in $G$:  $E(G) = \bigcup_{p \in \pset} E(G_p)$.  Analogous to the tiling contraction case, we assign a copy of $H(\R^2)$ for each $e \in E(G)$ and define $X^E = \disjointE H(\R^2)$.
Consider a fixed $e \in E(G_p)$ and consider its counterpart $\widetilde\psi(e) \subset S$.
By definition we know that  $\widetilde \psi(e)$ is a subset of $\rrule(G)_p$ and thus is comprised of subedges of the form $\lambda^{-1}(f + d)$, where $f$ is an edge in a prototile $q$ and $d$ is some element of the digit set $D_{pq}$.  We define the matrix $M^E$ by letting $M^E_{ef}$ be the number of copies of $f$ appearing in $\widetilde\psi(e)$ for all $f \in E(G)$.  This makes $M^E$ a  nonnegative integer matrix whose rows and columns are indexed by $E(G)$.

Suppose that $f$ is an edge in the prototile $q$.  The digit set $D^E_{ef}$ is obtained by taking all of the digits in $D_{pq}$ giving copies of $f$ in $\widetilde\psi(e)$:
\begin{equation}
D^E_{ef} = \{d \in D_{pq} \,\,|\,\, \widetilde\psi(e) \cap \lambda^{-1}\left(f + d \right) = \lambda^{-1}\left(f + d \right) \}
\end{equation}

For $B \in X^E$ and $e \in E(G)$ we define
\begin{equation}
\rrule^E(B)_e = \bigcup_{q \in \pset}\bigcup_{f \in E(G_q)} \lambda^{-1}\left(B_f + D^E_{ef} \right)
\end{equation}

With this definition we see that for $e \in E(G)$, 

\begin{equation*}
\widetilde\psi(e) = \bigcup_{q \in \pset}\bigcup_{f \in E(G_q)} \lambda^{-1}\left(f + D^E_{ef} \right)= \left(\rrule^E\left(\sqcup_{f \in E(G)} f\right)\right)_e
\end{equation*}

We can iterate $\rrule^E$ but not $\widetilde\psi$, since the latter is only defined for edges $e \in E(G)$. However we can define $\wpsin(e)$ for $n \in \N \cup \{\infty\}$ using $\psin$: take any edge $e' \in \tzero$ for which $e' = e + x$, then $\wpsin(e) = \psin(e') - x$.  It is a tedious, but not difficult, chase through notation to verify that  if $B = \sqcup_{f \in E(G)} f$, then
\begin{equation*}
\wpsin(e) = \left((\rrule^E)^n(B)\right)_e
\end{equation*}

Let $\KSinf \subset X^E$ represent the attractor of $\rrule^E$, and let $\Sinf \subset X$ represent the canonical projection of $\KSinf$ onto the support of the prototiles: $\Sinf_p = \bigcup_{e \in E(G_p)} \KSinf_e$.

\begin{prop}
Let $(G,S)$ be a recurrent pair and $T$ a self-similar tiling for the substitution $\sub$.  Then $\psiinf(\partial \tzero) = \Gamma(T, \Sinf)$.  If $\psiinf$ is injective, then $\partial \Tinf = \partial T_\Sinf$.
\end{prop}

The first part says that even if $\Gamma(T, \Sinf)$ is not the boundary of a tiling, it still corresponds to $\psiinf(\partial \tzero)$.  The second part does not quite extend to imply that  $T_{\Sinf}$ and $\Tinf$ are actually equal, because there is the possibility that the tile labels differ slightly if the tiles of $T_{\Sinf}$ intersect larger $T$-patches than those that $\tzero$ did, since, by definition, those endow the tiles of $\Tinf$ with their labels.

\begin{proof}
We must show that if $e \in E(G)$, then $\wpsiinf(e)=\KSinf_e$. We have
\[
\disjointE \wpsiinf(e) = \disjointE\left(\lim_{n \to \infty} \wpsin(e)\right) = \lim_{n \to \infty} \disjointE \wpsin(e) =  \lim_{n \to \infty} \rrule_E^n\left(\disjointE e\right) = \KSinf,
\] since $\KSinf$ is the attractor of $\rrule_E$.
\end{proof}

\subsection{Hausdorff dimension}

We will prove the following theorem.

\begin{thm}\label{thm:dim}
Suppose that $(G,S)$ is a recurrent pair for a tiling $T$ with substitution $\sub$ having  expansion factor $\lambda$ which satisfies the injectivity conditions of Definition \ref{injectivity_conditions}.
Let $M^E$ be the edge substitution matrix with largest eigenvalue $\lambda_E$. Then the Hausdorff dimension of $\partial \Tinf$ is $\frac{\ln \lambda_E}{\ln \lambda}$.
\end{thm}

It is important to notice that while $M$ is always a primitive matrix, by assumption, it often happens that $M^E$ is not.  Theorem \ref{thm:dim} can be deduced almost directly from the definition of the Hausdorff dimension when $M^E$ is primitive.  In the interest of completeness we prove Theorem~\ref{thm:dim} by describing a fractal realization as the fixed point of a graph-directed iterated function system (GIFS), as defined in \cite{Edgar-Golds,MW}. In particular, \cite{MW} allows us to compute the Hausdorff dimension of the boundary of $\Tinf$ by computing the Mauldin-Williams dimension of the GIFS. 

We begin by defining a directed combinatorial graph $\MWedgegraph$ that has one vertex, denoted $v_e$, for each edge $e \in E(G)$.
The edge set $E(\MWedgegraph)$ contains an edge $\epsilon$ from $v_e$ to $v_f$ for each copy of $f$ that appears as a subedge in $\widetilde\psi(e)$.  That is, there are $M^E_{ef}$ edges pointing from $v_e$ to $v_f$.   To each vertex $v_e$ we associate a copy of $\R^2$. To each edge $\epsilon \in E(\MWedgegraph)$ from $e$ to $f$ we assign a digit $d \in D^E_{ef}$ and construct the map $x \mapsto\lambda^{-1}(x + d)$ and call this map $h_\epsilon$.  These are the components necessary to define a GIFS; since the maps $h_\epsilon$ are taken directly from the definition of $\rrule^E$ we refer to this as the GIFS given by $\rrule^E$.

Our fractal graph $\KSinf$ (as a vector of edges) is invariant for this GIFS in the sense that for each vertex $v_e \in V(\MWedgegraph)$, $\KSinf_e = \bigcup_{v_f \in V(\MWedgegraph)} \bigcup_{\epsilon \in E_{ef}} h_\epsilon(\KSinf_{e})$, where $E_{ef}$ denotes edges from $v_e$ to $v_f$. 

The Hausdorff dimension of the invariant set of a GIFS has bounds that depend on the contraction factors of the maps $h_\epsilon$.  In our case, each contraction factor is exactly $\lambda^{-1}$ and this leads to the lower and upper bounds being equal, and in fact $\ln \lambda_E / \ln \lambda$ as soon as we can establish that the GIFS satisfies the strong open set condition given here (and adapted to our notation.)

\begin{definition}[\cite{Edgar-Golds}, Definition 3.11]\label{def:sosc}
A GIFS with attractor set $\KSinf$ satisfies the \emph{strong open set condition} if, for each $v_e \in V(\MWedgegraph)$, there exists an open set $U_e \subset \R^2$ satisfying:
\begin{enumerate}
\item\label{sosc1} For all vertices $v_e, v_f$ and $\epsilon \in E_{ef}$, $h_\epsilon(U_f) \subset U_e$;
\item\label{sosc2} For all vertices $v_e, v_f$ and $v_{f'}$, $\epsilon \in E_{ef}$ and $\epsilon' \in E_{ef'}$ with $\epsilon \neq \epsilon'$, we have $h_\epsilon(U_f) \cap h_{\epsilon'}(U_{f'}) = \varnothing$; and
\item\label{sosc3} For each vertex $v_e$, $U_e \cap \KSinf_e \neq \varnothing$.
\end{enumerate}
\end{definition}

\begin{prop}\label{prop:sosc}
Suppose $(G,S)$ is a recurrent pair satisfying the injectivity conditions (\ref{pc1})-(\ref{pc4}) of Definition \ref{injectivity_conditions}.  Then the GIFS  given by $\rrule^E$ satisfies the strong open set condition.
\end{prop}
\begin{proof}
Fix an edge $e \in E(G)$ in the prototile $p$; $e$ is associated to a unique vertex of our GIFS. Denote by $\einf=\KSinf_e$ its counterpart in the attractor and define $U_e$ as follows. Let $a_e$ be the set of interior vertices of $\einf$ that are endpoints of $\einf$ (i.e., endpoints of $\einf$ interior to the support of $p$; by the injectivity conditions this has either one or two elements). Let $U_{e,n} = \int([\einf]^{\rrule^n(\pset)_p}\setminus[a_e]^{\rrule^n(\pset)_p} )$. Then set $U_e = \bigcup_{n \in \N} U_{e,n}$. Then the injectivity conditions \eqref{pc1}--\eqref{pc4} then guarantee that these $U_e$ satisfy conditions \eqref{sosc1}--\eqref{sosc3} of Definition~\ref{def:sosc}.
\end{proof}

The \emph{Mauldin-Williams dimension} \cite{MW} of the invariant set $F$ of a GIFS is defined as follows. Suppose each $h_\epsilon$ has similiarity ratio $r_\epsilon<1$. Define a matrix $M^E(s)$ by $M^E_{ef}(s) = \sum_{\epsilon \in E_{ef}} r_\epsilon^s$, where $e$ and $f$ represent edges in $G$ and therefore vertices in $\MWedgegraph$. Let $\Phi(s)$ denote the spectral radius of $M^E(s)$.  
The Mauldin-Williams dimension $\dim_{MW}(F)$ of $F$ is the unique $s_1$ for which $\Phi(s_1) = 1$.

\begin{proof}[Proof of Theorem~\ref{thm:dim}]
In \cite[Section 3]{Edgar-Golds} it is shown that if $\|h_\epsilon(x) - h_\epsilon(y)\| = r_\epsilon \|x - y \|$ for each $\epsilon \in E_{ef}$ and $x,y$ in the copy of $\R^2$ associated with $v_f$, and if the GIFS satisfies the strong open set condition, then the Hausdorff dimension $\dim_H\KSinf_e = \dim_{MW}(\KSinf_e)$ for each vertex $v_e$. Our GIFS satisfies the strong open set condition by Proposition~\ref{prop:sosc} and has $r_\epsilon = \lambda^{-1}$ for all $\epsilon$. This means that $M^E(s) = \lambda^{-s} M^E$, where $M^E$ is the edge substitution matrix given by $(G,S)$.  This implies that $\Phi(s) = \lambda^{-s}\lambda_E$, where $\lambda_E$ is the spectral radius of $M^E$.
Solving $\lambda^{-s}\lambda_E = 1$ gives $s = \ln \lambda_E / \ln \lambda$.
\end{proof}

\section{\v{C}ech Cohomology of a tiling space}\label{cohom_zeta_sec}

In this section we relate our fractal tiling construction to the Anderson-Putnam complex and \v{C}ech cohomology. Given an nonperiodic and primitive substitution tiling along with a recurrent pair satisfying the injectivity conditions of Definition \ref{injectivity_conditions}, we show that defining an orientation on the recurrent pair $(G,S)$ gives rise to the substitution maps and boundary maps for the associated fractal realization. If the recurrent pair is quasi-dual then the fractal realization forces its border and is mutually locally derivable to the original substitution tiling. Putting this together, a recurrent pair $(G,S)$ can be used to compute the cohomology of the original tiling. The upshot of our construction is that \v{C}ech cohomology is readily computable for tilings that do not force the border. We illustrate our method by computing the cohomology for the 2-dimensional Thue-Morse (2DTM) tiling. We begin with some background on substitution tiling spaces and their cohomology.

\subsection{A brief description of the tiling space of a substitution}

A substitution $\sub$ with expansion $\lambda$ generates a topological space $\tilingspace$ of tilings often called the \emph{hull}.  This consists of all tilings $T$ such that for every finite patch of tiles in $T$ there exists a prototile $\ptile \in \pset$ and an $N$ for which a copy of the patch appears inside $\sub^N(\ptile)$.  That is, every patch in $T$ is admissible by the substitution.

The topology on $\tilingspace$ is generated by the `big ball' metric, which measures how close two tilings are by how little they differ on big balls around the origin.\footnote{This is a continuous analogue to the standard metric used on shift spaces in symbolic dynamics.} We will not define it precisely here but refer readers to \cite[p.5]{Sol} for details.  It is possible to define $\tilingspace$ as the orbit closure of a fixed self-similar tiling $T$ under the action of translation, and in this view $\tilingspace$ is often called the hull of $T$.

The substitution $\sub$ extends to all tilings in $\tilingspace$ in a natural way.  Given a tiling $T \in \tilingspace$ and a tile $t \in T$, we place the patch $\sub(t)$ atop $\lambda \supp(t)$.  Doing this for all the tiles of $T$ yields a tiling we call $\sub(T)$ which also lives in $\tilingspace$.  When the substitution on $\tilingspace$ is invertible we call the substitution {\em recognizable}.  This means that in any location of any tiling $T$ it is possible to determine the exact type and location of the supertile $\sub(t)$ at that spot using only local information.  Since all the tilings considered in this paper are nonperiodic, \cite{Sol2} implies that they are recognizable. 

\subsection{Border-forcing}

For topological analysis of tiling spaces, it is crucial that the substitution have the property of forcing its border.  This means that whenever a supertile $\sub^N(t)$ appears in $T\in \tilingspace$, the patch $[\sub^N(t)]^T$ of tiles (which includes those immediately adjacent to  $\sub^N(t)$ in $T$) are independent of $T$ and appear in identical fashion anywhere else $\sub^N(t)$ appears in any tiling in $\tilingspace$.  Formally, the substitution forces the border if whenever $t$ and $t'$ are equivalent tiles in $T$, then $[\sub^N(t)]^T$ and $[\sub^N(t')]^T$ are equivalent  patches.

When $\sub$ is border-forcing its tiling space is homeomorphic to the inverse limit of its approximants as discussed below, which is key to cohomology computations on $\tilingspace$.
The following theorem gives conditions for border-forcing that are sufficient but not necessary; it is often possible to tell by inspection whether a given recurrent pair will produce a substitution with the desired result.

\begin{prop}\label{Prop:border-forcing}
Let $(G,S)$ be a recurrent pair for a self-similar tiling $T$ satisfying the conditions of Proposition \ref{thm: psi infinity is injective}. Then the substitution map $\subinf$ for the tiling $\Tinf$ is border-forcing.\end{prop}

\begin{proof}
The fact that $G$ and $S$ are quasi-dual and form a recurrent pair means that for any $t \in \tzero$ and its counterpart $\tinf \in \Tinf$ we have $[t]^T = [\tinf]^T$, and that both $t$ and $\tinf$ are on the interior of the same patch of $T$-tiles.  By definition, the equivalence class of this patch forms the label for both $t$ and $\tinf$.
Let $k_1 = \max_{t \in \tzero} \diam [t]^{T}$. Then every ball of radius $2k_1$ contains a patch of tiles in $T$ which determine at least one tile in $\Tinf$.

There is a strictly positive minimum distance $k_2$ between $\partial \tinf$ and $\partial\supp[\tinf]^{T}$, where the minimum is taken over all $\tinf \in \Tinf$. Choose $N$ such that $\lambda^N k_2 > 2k_1$. For this $N$, we know that $\sub^N([\tinf]^T)$ determines all of the $T$-tiles on its interior, and this includes all of the $T$-tiles a distance of $2k_1$ or less from
$\lambda^N(\supp\tinf)$.  By choice of $k_1$ this patch determines $[\subinf^N(\tinf)]^{\Tinf}$, as desired.
\end{proof}

\begin{cor}
Suppose $T$ is a self-similar tiling satisfying the conditions of Theorem \ref{thm:existence} and let $\tilingspace$ denote its hull.  Then there are infinitely many border-forcing substitutions $\subinf$ such that $(\oinf,\subinf)$ is topologically conjugate to $(\Omega,\sub)$.
\end{cor}
\begin{proof}
The proof of Theorem \ref{thm:existence} can be adapted to produce an infinite number of distinct recurrent pairs. In particular, Lemma \ref{lem:existdualish} can be adapted by
\begin{enumerate}
\item using a quasi-dual graph in place of $G_0$ or
\item allowing $N$ to increase, giving several distinct $G_1$ graphs for each $N$.
 \qedhere
\end{enumerate}
\end{proof}

\subsection{\v{C}ech cohomology of a substitution tiling}\label{AP_background}
We begin with a very brief description of the Anderson-Putnam complex and \v{C}ech cohomology of a border forcing tiling space in two dimensions. We refer the reader to \cite{AP} or \cite{Lorenzo.book} for further details. Suppose $\Omega$ is the tiling space of a border forcing nonperiodic substitution tiling with finite local complexity and prototile set $\pset$. A finite CW complex $\Gamma$, called the approximant, is defined by identifying edges and vertices in two prototiles whenever those edges or vertices are common in any two translates that occur in a patch in the tiling space. Anderson and Putnam \cite{AP} have shown that the tiling space $\Omega$ is the inverse limit of the approximant for tilings that force the border
\begin{equation}\label{inverse limit}
\Omega = \underset{\longleftarrow}{\lim} (\Gamma,f),
\end{equation}
where $f$ is the forgetful map described in \cite[Section 2.5]{Lorenzo.book}.

Since the approximant is a finite CW complex, we obtain a chain complex in simplicial homology. However, homology is not well-behaved with respect to inverse limits so we use cohomology instead, and consider the dual cochain complex. Let $\Gamma^0$, $\Gamma^1$ and $\Gamma^2$ denote the functions from the vertices, edges, and tiles into the group of integers, respectively. There are coboundary maps $\delta_0: \Gamma^0 \rightarrow \Gamma^1$ and $\delta_1: \Gamma^1 \rightarrow \Gamma^2$ defined by taking the transpose of the matrix defining the homology boundary maps. We then have the following cochain complex:
\begin{align*}
0 \overset{}{\longrightarrow} \Gamma^0 & \overset{\delta_0}{\longrightarrow} \Gamma^1 \overset{\delta_1}{\longrightarrow} \Gamma^2 \overset{}{\longrightarrow} 0.
\end{align*}
From \cite[Theorem 3.4]{Lorenzo.book}, \v{C}ech cohomology and singular cohomology are equal on finite CW complexes. Thus we obtain the cohomology groups of the approximant $\Gamma$ by
\begin{equation*}
\check{H}^0(\Gamma)= \ker(\delta_0), \quad \check{H}^1(\Gamma)= \ker(\delta_1)/\Im(\delta_0), \quad \text{and} \quad \check{H}^2(\Gamma)= \Gamma^2/\Im(\delta_1).
\end{equation*}

Since we have a substitution tiling there are substitution maps on vertices, edges, and tiles, which we denote by $A_0$, $A_1$, and $A_2$ respectively. Since cohomology is contravariant, the inverse limit appearing in \eqref{inverse limit} turns into direct limits on the cohomology groups of the approximant. Thus, the \v{C}ech cohomology of $\Omega$ is given by
\begin{equation}\label{inv_limit}
\check{H}^i(\Omega)=\underset{\longrightarrow}{\lim}(H^i(\Gamma),A_i^*)
\end{equation}
where $A_i^*$ denotes the map induced by the substitution on the singular cohomology group $H^i(\Gamma)$.

\subsection{Cohomology from a recurrent pair $(G,S)$}\label{Cohomology_G-S}

We now relate the Anderson-Putnam complex to our situation in order to compute the cohomology  of a tiling $T$ using a fractal realization of $T$. In fact, it is not necessary to see the fractal at all: once a recurrent pair $(G,S)$ satisfying the injectivity conditions is identified, it is possible to do all computations using the graphs $G$ and $S$. We illustrate the construction in Section \ref{2DTM_cohomology} by computing the \v{C}ech cohomology of the two dimensional Thue-Morse tiling using a recurrent pair $(G,S)$. We note that it may be useful to have the example in Section \ref{2DTM_cohomology} at hand while reading this section.

Suppose $T$ is a strongly nonperiodic tiling with finite local complexity and prototile set $\pset$ admitting a  recurrent pair $(G,S)$ satisfying the injectivity conditions of Definition \ref{injectivity_conditions}. Let $\Omega$ be the tiling space associated with $T$. 
Using the construction in Section \ref{sec:psi} we obtain a fractal tiling $\Tinf$ with tiling space $\oinf$ and prototile set $\pinf$. Proposition \ref{prop:T infty is a tiling} implies that $\Tinf$ is a self-similar nonperiodic tiling and Proposition \ref{Prop:border-forcing} implies that $\Tinf$ forces the border.

We will show that the Anderson-Putnam complex is completely determined by the recurrent pair $(G,S)$. First the approximant $\Gamma$ of $\oinf$. 

Let $V_I$ denote the set of interior vertices of $G$ and for each $v \in V_I$ let $f_v$ denote the function that takes the value one on $v$ and zero on $V_I \setminus v$. Then $\Gamma^0$ is generated by $\{f_v \mid v \in V_I\}$ and $\Gamma^0 \cong \Z^{|v_I|}$.

We now look at $\Gamma^1$. Since we have assumed that $(G,S)$ satisfies the conditions of Proposition \ref{thm: psi infinity is injective}, $\psiinf$ is a bijective map from $\tzero$ to $\Tinf$, which induces a bijection taking edges in $G$ to fractal edges whose translations make up $\Tinf$. More specifically, there are two types of edges involved in computing the cohomology of the approximant. The first are pairs of edges in $G$ that form a single edge when we take the quotient by the equivalence relation $\sim_T$ described in section \ref{subsec:tilinggraphs} (two boundary vertices are identified if they ever meet in any patch in $T$). We call pairs of these edges a (single) boundary edge. The second type are edges that connect two interior vertices of $G$ within a prototile, and we call these edges interior edges. 
Let $E$ be the union of the boundary edges and the interior edges. For $e \in E$ let $f_e$ denote the function that takes the value one on $e$ and zero on $E \setminus e$. Then $\Gamma^1$ is generated by $\{f_e \mid e \in E\}$ and $\Gamma^1 \cong \Z^{|E|}$.

For $\Gamma^2$, recall that the prototile set $\pinf$ is in bijective correspondence with distinct vertex patterns in $T$. For $\ptile \in \pinf$, let $f_{\ptile}$ denote the function that takes the value one on $\ptile$ and zero on $\pinf \setminus \ptile$. Then $\Gamma^2$ is generated by $\{f_{\ptile} \mid \ptile \in \pinf\}$ and $\Gamma^2 \cong \Z^{|\pinf|}$.

To compute the cohomology of the approximant it remains to determine the coboundary maps. We accomplish this by finding the boundary maps in homology and then taking adjoints. The first step is to assign an orientation to each edge in $E$ (forming the dual space of $\Gamma^1$). As in \cite[Section 3.2]{Lorenzo.book}, let $\partial_2$ be the matrix with a row for each oriented edge in $E$ and a column for each prototile in $\pinf$. For $e\in E$ and $p \in \pinf$, a $+1$ is added to the $(e,p)$ entry if edge $e$ appears as an edge in $p$ in the same orientation and a $-1$ is added if it appears in the opposite orientation. The coboundary map $\delta_1:\Gamma^1 \to \Gamma^2$ is the transpose matrix $\delta_1=\partial_2^t$. Let $\partial_1$ be the matrix with a row for each interior vertex $V_I$, and a column for each oriented edge in $E$. For each $e \in E$, place a $+1$ in the $(v,e)$ entry if $v$ is the range vertex of $e$ and place a $-1$ in the $(w,e)$ entry if $w$ is the source vertex of $e$. The coboundary map $\delta_0:\Gamma^0 \to \Gamma^1$ is the transpose matrix $\delta_0=\partial_1^t$. Putting all of this together leads to the cohomology groups of the approximants $\check{H}(\Gamma^i)$ for $i=0,1,2$.

In order to compute the \v{C}ech cohomology of the tiling space $\oinf$ the final ingredients are the substitution maps. The matrix $A_0$ is the $|V_I| \times |V_I|$ substitution matrix on interior vertices, described entirely by the substitution of interior vertices from $G$ to $S$. The matrix $A_1$ is an $|E| \times |E|$ matrix with a $+1$ added to the $(e,f)$ entry if $f$ appears in the substitution of $e$ in the same orientation as $e$ and a $-1$ if $f$ appears in the opposite orientation, where the substitution is determined by the graph $S$. The matrix $A_2$ is a $|\pinf| \times |\pinf|$ matrix where entry $(p,q)$ is the number of translations of prototile $q$ appearing in the substitution of prototile $p$.

Computing the \v{C}ech cohomology of $\oinf$ is now an exercise in linear algebra. Lemma \ref{MLD_lemma}\eqref{MLD_part} implies that $\oinf$ is mutually locally derivable to $\Omega$, and hence the \v{C}ech cohomology groups of the two tiling spaces are equal.

%
%
%

\subsection{Cohomology of the two dimensional Thue-Morse tiling}\label{2DTM_cohomology}

We compute the cohomology of the 2-dimensional Thue-Morse (2DTM) tiling using the recurrent pair shown in Figure \ref{2DTM_graph_sub}  of the introduction. 

We label the edges of the graph $G$ and define an orientation on these edges as in Figure \ref{fig:2DTM_tiles} and extend the orientation to $S$. It is essential that the orientation be consistent across pairs of edges connected by boundary vertices, since they become edges in the AP complex of $\Tinf$.

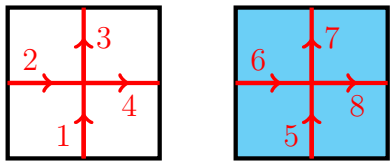
\begin{figure}[ht]\label{fig:edges_tiles}
\[
\begin{tikzpicture}
\draw[ultra thick] (0,0) rectangle (2,2);

\draw[ultra thick,red,->] (1,0) -> node[auto] {1} (1,0.6);
\draw[ultra thick,red,->] (1,1) ->  node[auto,pos=1,swap] {3} (1,1.6);
\draw[ultra thick,red] (1,0) -- (1,2);

\draw[ultra thick,red,->] (0,1) -> node[auto] {2} (0.6,1);
\draw[ultra thick,red,->] (1,1) ->  node[auto,pos=1,swap] {4} (1.6,1);
\draw[ultra thick,red] (0,1) -- (2,1);

\begin{scope}[xshift=3cm]

\draw[ultra thick,fill=cyan!50] (0,0) rectangle (2,2);

\draw[ultra thick,red,->] (1,0) -> node[auto] {5} (1,0.6);
\draw[ultra thick,red,->] (1,1) ->  node[auto,pos=1,swap] {7} (1,1.6);
\draw[ultra thick,red] (1,0) -- (1,2);

\draw[ultra thick,red,->] (0,1) -> node[auto] {6} (0.6,1);
\draw[ultra thick,red,->] (1,1) ->  node[auto,pos=1,swap] {8} (1.6,1);
\draw[ultra thick,red] (0,1) -- (2,1);
\end{scope}
\end{tikzpicture}
\]
\caption{The prototiles of the 2DTM tiling with an orientated dual graph $G$ inscribed.} 
\end{figure}
\begin{figure}[ht]
\begin{tikzpicture}
\begin{scope}[xshift=0cm,yshift=0cm,scale=1.1]
      \node at (0,0) {\includegraphics[width=0.75in]{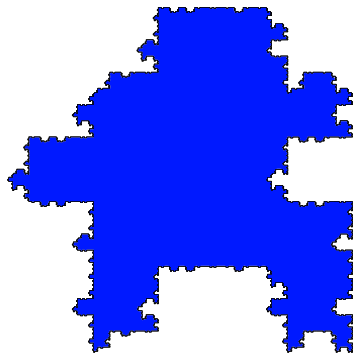}};
      \draw[black,->] (-1.0,-1.0) --  (-1.0,-0.5) node[left] {$\scriptstyle 7$};
      \draw[black,->] (-1,-0.5) -- (-1,0.5) node[left] {$\scriptstyle 5$};
      \draw[black] (-1,0.5) -- (-1,1);
      \draw[black,->] (-1,1) --  (-0.5,1) node[above] {$\scriptstyle 8$};
      \draw[black,->] (-0.5,1) -- (0.5,1) node[above] {$\scriptstyle 6$};
      \draw[black] (0.5,1) -- (1,1);
      \draw[black,->] (1.0,-1.0) --  (1.0,-0.5) node[right] {$\scriptstyle 7$};
      \draw[black,->] (1,-0.5) -- (1,0.5) node[right] {$\scriptstyle 5$};
      \draw[black] (1,0.5) -- (1,1);
      \draw[black,->] (-1,-1) --  (-0.5,-1) node[below] {$\scriptstyle 8$};
      \draw[black,->] (-0.5,-1) -- (0.5,-1) node[below] {$\scriptstyle 6$};
      \draw[black] (0.5,-1) -- (1,-1);
      \node at (-1,-1) {$\scriptstyle \bullet$};
      \node at (-1,1) {$\scriptstyle \bullet$};
      \node at (1,-1) {$\scriptstyle \bullet$};
      \node at (1,1) {$\scriptstyle \bullet$};
\end{scope}
\begin{scope}[xshift=3cm,yshift=0cm,scale=1.1]
      \node at (0,0) {\includegraphics[width=0.75in]{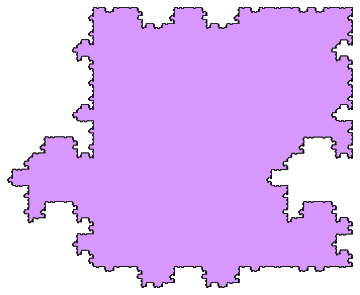}};
      \draw[black,->] (-1.0,-1.0) --  (-1.0,-0.5) node[left] {$\scriptstyle 3$};
      \draw[black,->] (-1,-0.5) -- (-1,0.5) node[left] {$\scriptstyle 1$};
      \draw[black] (-1,0.5) -- (-1,1);
      \draw[black,->] (-1,1) --  (-0.5,1) node[above] {$\scriptstyle 4$};
      \draw[black,->] (-0.5,1) -- (0.5,1) node[above] {$\scriptstyle 2$};
      \draw[black] (0.5,1) -- (1,1);
      \draw[black,->] (1.0,-1.0) --  (1.0,-0.5) node[right] {$\scriptstyle 3$};
      \draw[black,->] (1,-0.5) -- (1,0.5) node[right] {$\scriptstyle 1$};
      \draw[black] (1,0.5) -- (1,1);
      \draw[black,->] (-1,-1) --  (-0.5,-1) node[below] {$\scriptstyle 4$};
      \draw[black,->] (-0.5,-1) -- (0.5,-1) node[below] {$\scriptstyle 2$};
      \draw[black] (0.5,-1) -- (1,-1);
      \node at (-1,-1) {$\scriptstyle \bullet$};
      \node at (-1,1) {$\scriptstyle \bullet$};
      \node at (1,-1) {$\scriptstyle \bullet$};
      \node at (1,1) {$\scriptstyle \bullet$};
\end{scope}
\begin{scope}[xshift=6cm,yshift=0cm,scale=1.1]
      \node at (0,0) {\includegraphics[width=0.75in]{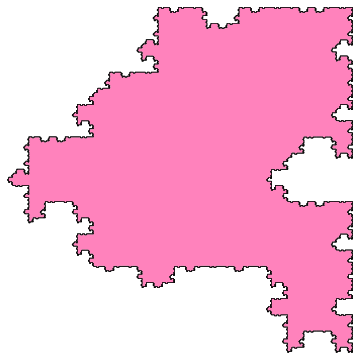}};
      \draw[black,->] (-1.0,-1.0) --  (-1.0,-0.5) node[left] {$\scriptstyle 3$};
      \draw[black,->] (-1,-0.5) -- (-1,0.5) node[left] {$\scriptstyle 5$};
      \draw[black] (-1,0.5) -- (-1,1);
      \draw[black,->] (-1,1) --  (-0.5,1) node[above] {$\scriptstyle 8$};
      \draw[black,->] (-0.5,1) -- (0.5,1) node[above] {$\scriptstyle 2$};
      \draw[black] (0.5,1) -- (1,1);
      \draw[black,->] (1.0,-1.0) --  (1.0,-0.5) node[right] {$\scriptstyle 7$};
      \draw[black,->] (1,-0.5) -- (1,0.5) node[right] {$\scriptstyle 1$};
      \draw[black] (1,0.5) -- (1,1);
      \draw[black,->] (-1,-1) --  (-0.5,-1) node[below] {$\scriptstyle 4$};
      \draw[black,->] (-0.5,-1) -- (0.5,-1) node[below] {$\scriptstyle 6$};
      \draw[black] (0.5,-1) -- (1,-1);
      \node at (-1,-1) {$\scriptstyle \bullet$};
      \node at (-1,1) {$\scriptstyle \bullet$};
      \node at (1,-1) {$\scriptstyle \bullet$};
      \node at (1,1) {$\scriptstyle \bullet$};
\end{scope}
\begin{scope}[xshift=9cm,yshift=0cm,scale=1.1]
      \node at (0,0) {\includegraphics[width=0.75in]{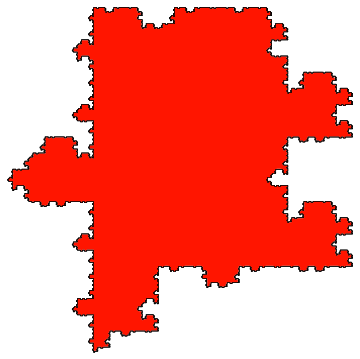}};
      \draw[black,->] (-1.0,-1.0) --  (-1.0,-0.5) node[left] {$\scriptstyle 7$};
      \draw[black,->] (-1,-0.5) -- (-1,0.5) node[left] {$\scriptstyle 1$};
      \draw[black] (-1,0.5) -- (-1,1);
      \draw[black,->] (-1,1) --  (-0.5,1) node[above] {$\scriptstyle 4$};
      \draw[black,->] (-0.5,1) -- (0.5,1) node[above] {$\scriptstyle 6$};
      \draw[black] (0.5,1) -- (1,1);
      \draw[black,->] (1.0,-1.0) --  (1.0,-0.5) node[right] {$\scriptstyle 3$};
      \draw[black,->] (1,-0.5) -- (1,0.5) node[right] {$\scriptstyle 5$};
      \draw[black] (1,0.5) -- (1,1);
      \draw[black,->] (-1,-1) --  (-0.5,-1) node[below] {$\scriptstyle 8$};
      \draw[black,->] (-0.5,-1) -- (0.5,-1) node[below] {$\scriptstyle 2$};
      \draw[black] (0.5,-1) -- (1,-1);
      \node at (-1,-1) {$\scriptstyle \bullet$};
      \node at (-1,1) {$\scriptstyle \bullet$};
      \node at (1,-1) {$\scriptstyle \bullet$};
      \node at (1,1) {$\scriptstyle \bullet$};
\end{scope}
\begin{scope}[xshift=0cm,yshift=-3cm,scale=1.1]
      \node at (0,0) {\includegraphics[width=0.75in]{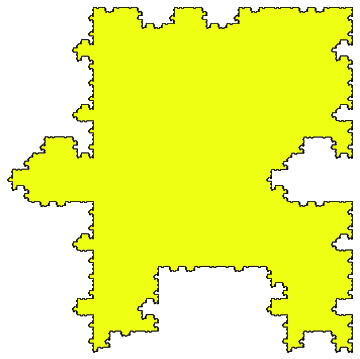}};
      \draw[black,->] (-1.0,-1.0) --  (-1.0,-0.5) node[left] {$\scriptstyle 7$};
      \draw[black,->] (-1,-0.5) -- (-1,0.5) node[left] {$\scriptstyle 1$};
      \draw[black] (-1,0.5) -- (-1,1);
      \draw[black,->] (-1,1) --  (-0.5,1) node[above] {$\scriptstyle 4$};
      \draw[black,->] (-0.5,1) -- (0.5,1) node[above] {$\scriptstyle 2$};
      \draw[black] (0.5,1) -- (1,1);
      \draw[black,->] (1.0,-1.0) --  (1.0,-0.5) node[right] {$\scriptstyle 7$};
      \draw[black,->] (1,-0.5) -- (1,0.5) node[right] {$\scriptstyle 1$};
      \draw[black] (1,0.5) -- (1,1);
      \draw[black,->] (-1,-1) --  (-0.5,-1) node[below] {$\scriptstyle 8$};
      \draw[black,->] (-0.5,-1) -- (0.5,-1) node[below] {$\scriptstyle 6$};
      \draw[black] (0.5,-1) -- (1,-1);
      \node at (-1,-1) {$\scriptstyle \bullet$};
      \node at (-1,1) {$\scriptstyle \bullet$};
      \node at (1,-1) {$\scriptstyle \bullet$};
      \node at (1,1) {$\scriptstyle \bullet$};
\end{scope}
\begin{scope}[xshift=3cm,yshift=-3cm,scale=1.1]
      \node at (0,0) {\includegraphics[width=0.75in]{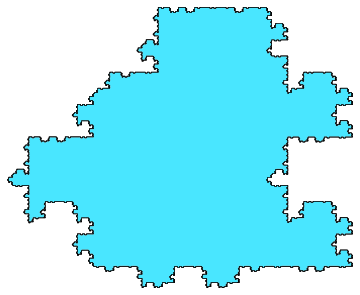}};
      \draw[black,->] (-1.0,-1.0) --  (-1.0,-0.5) node[left] {$\scriptstyle 3$};
      \draw[black,->] (-1,-0.5) -- (-1,0.5) node[left] {$\scriptstyle 5$};
      \draw[black] (-1,0.5) -- (-1,1);
      \draw[black,->] (-1,1) --  (-0.5,1) node[above] {$\scriptstyle 8$};
      \draw[black,->] (-0.5,1) -- (0.5,1) node[above] {$\scriptstyle 6$};
      \draw[black] (0.5,1) -- (1,1);
      \draw[black,->] (1.0,-1.0) --  (1.0,-0.5) node[right] {$\scriptstyle 3$};
      \draw[black,->] (1,-0.5) -- (1,0.5) node[right] {$\scriptstyle 5$};
      \draw[black] (1,0.5) -- (1,1);
      \draw[black,->] (-1,-1) --  (-0.5,-1) node[below] {$\scriptstyle 4$};
      \draw[black,->] (-0.5,-1) -- (0.5,-1) node[below] {$\scriptstyle 2$};
      \draw[black] (0.5,-1) -- (1,-1);
      \node at (-1,-1) {$\scriptstyle \bullet$};
      \node at (-1,1) {$\scriptstyle \bullet$};
      \node at (1,-1) {$\scriptstyle \bullet$};
      \node at (1,1) {$\scriptstyle \bullet$};
\end{scope}
\begin{scope}[xshift=6cm,yshift=-3cm,scale=1.1]
      \node at (0,0) {\includegraphics[width=0.75in]{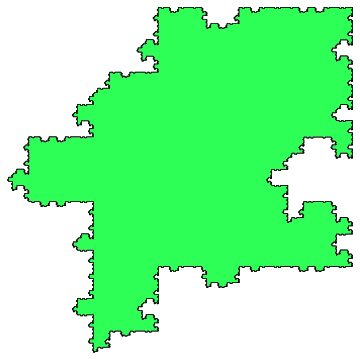}};
      \draw[black,->] (-1.0,-1.0) --  (-1.0,-0.5) node[left] {$\scriptstyle 7$};
      \draw[black,->] (-1,-0.5) -- (-1,0.5) node[left] {$\scriptstyle 5$};
      \draw[black] (-1,0.5) -- (-1,1);
      \draw[black,->] (-1,1) --  (-0.5,1) node[above] {$\scriptstyle 8$};
      \draw[black,->] (-0.5,1) -- (0.5,1) node[above] {$\scriptstyle 2$};
      \draw[black] (0.5,1) -- (1,1);
      \draw[black,->] (1.0,-1.0) --  (1.0,-0.5) node[right] {$\scriptstyle 3$};
      \draw[black,->] (1,-0.5) -- (1,0.5) node[right] {$\scriptstyle 1$};
      \draw[black] (1,0.5) -- (1,1);
      \draw[black,->] (-1,-1) --  (-0.5,-1) node[below] {$\scriptstyle 8$};
      \draw[black,->] (-0.5,-1) -- (0.5,-1) node[below] {$\scriptstyle 2$};
      \draw[black] (0.5,-1) -- (1,-1);
      \node at (-1,-1) {$\scriptstyle \bullet$};
      \node at (-1,1) {$\scriptstyle \bullet$};
      \node at (1,-1) {$\scriptstyle \bullet$};
      \node at (1,1) {$\scriptstyle \bullet$};
\end{scope}
\begin{scope}[xshift=9cm,yshift=-3cm,scale=1.1]
      \node at (0,0) {\includegraphics[width=0.75in]{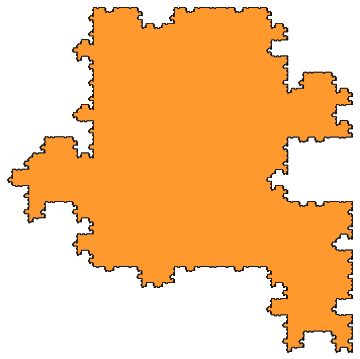}};
      \draw[black,->] (-1.0,-1.0) --  (-1.0,-0.5) node[left] {$\scriptstyle 3$};
      \draw[black,->] (-1,-0.5) -- (-1,0.5) node[left] {$\scriptstyle 1$};
      \draw[black] (-1,0.5) -- (-1,1);
      \draw[black,->] (-1,1) --  (-0.5,1) node[above] {$\scriptstyle 4$};
      \draw[black,->] (-0.5,1) -- (0.5,1) node[above] {$\scriptstyle 6$};
      \draw[black] (0.5,1) -- (1,1);
      \draw[black,->] (1.0,-1.0) --  (1.0,-0.5) node[right] {$\scriptstyle 7$};
      \draw[black,->] (1,-0.5) -- (1,0.5) node[right] {$\scriptstyle 5$};
      \draw[black] (1,0.5) -- (1,1);
      \draw[black,->] (-1,-1) --  (-0.5,-1) node[below] {$\scriptstyle 4$};
      \draw[black,->] (-0.5,-1) -- (0.5,-1) node[below] {$\scriptstyle 6$};
      \draw[black] (0.5,-1) -- (1,-1);
      \node at (-1,-1) {$\scriptstyle \bullet$};
      \node at (-1,1) {$\scriptstyle \bullet$};
      \node at (1,-1) {$\scriptstyle \bullet$};
      \node at (1,1) {$\scriptstyle \bullet$};
\end{scope}
    \end{tikzpicture}
\caption{The $8$ prototiles of the fractal realization of the 2DTM tiling from $(G,S)$. Each orientated (single) edge is labelled by a pair of edges from $G$ as pictured in the bounding squares (which are only present to label the edges).} 
\label{fig:2DTM_tiles}
\end{figure}
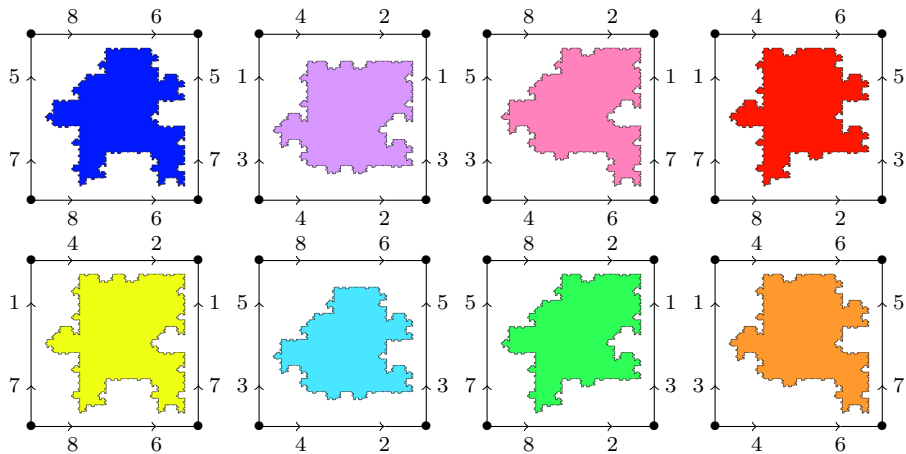

Applying the machinery from Section \ref{sec:psi} we obtain the tiles appearing in Figure \ref{fig:2DTM_tiles} along with their substitution appearing in Figure \ref{fig:2DTM_substitution}.\begin{figure}[ht]
\[
	\includegraphics[width=1.0in]{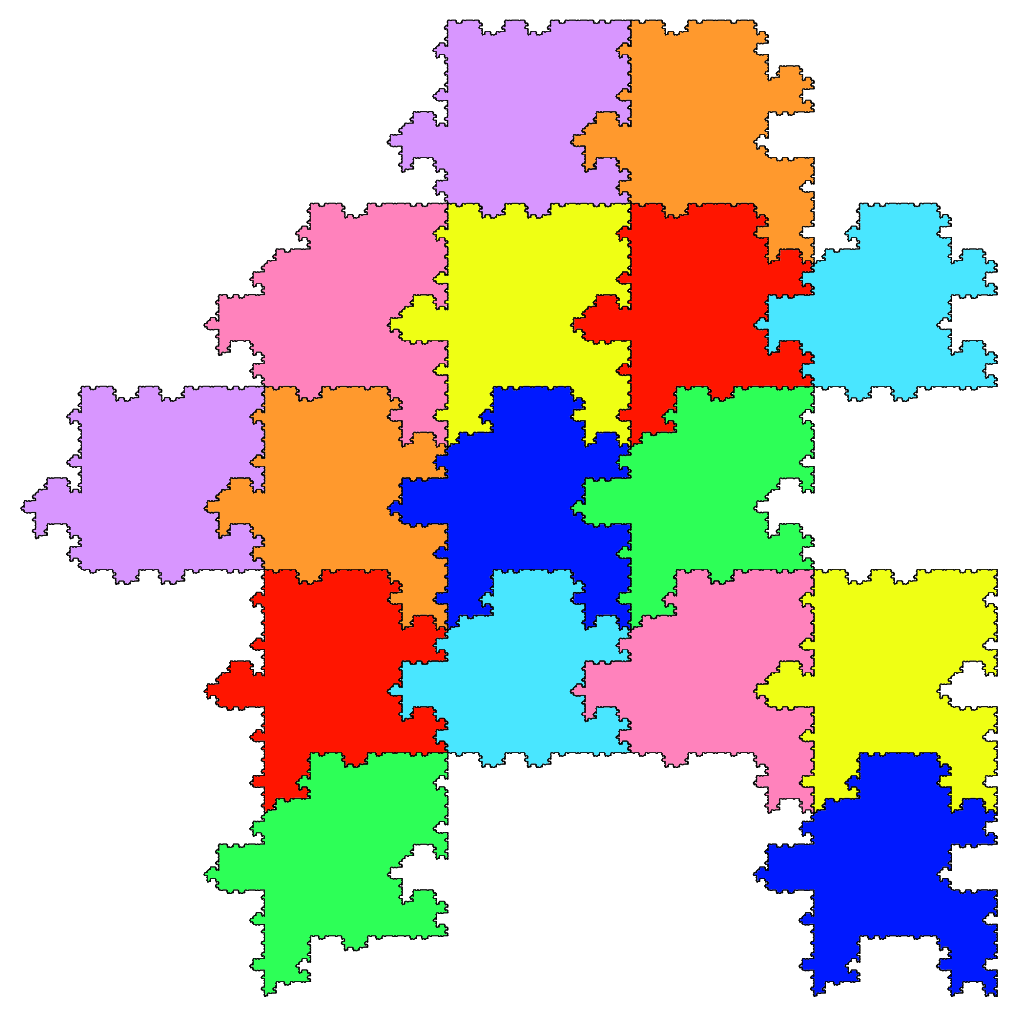}
	\includegraphics[width=1.0in]{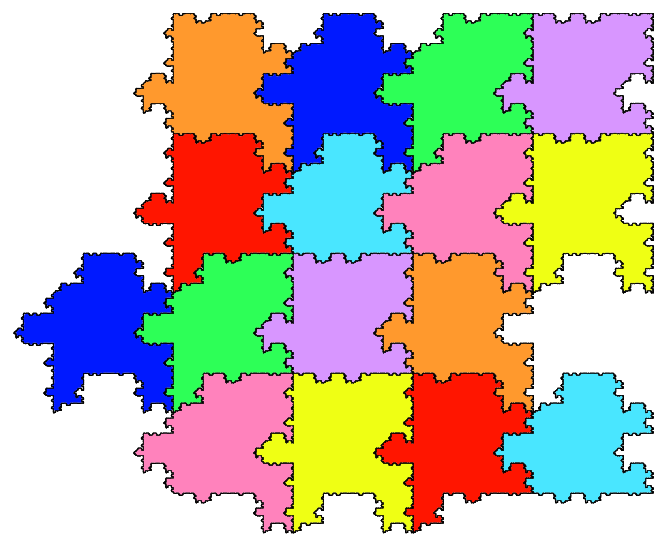} 
	\includegraphics[width=1.0in]{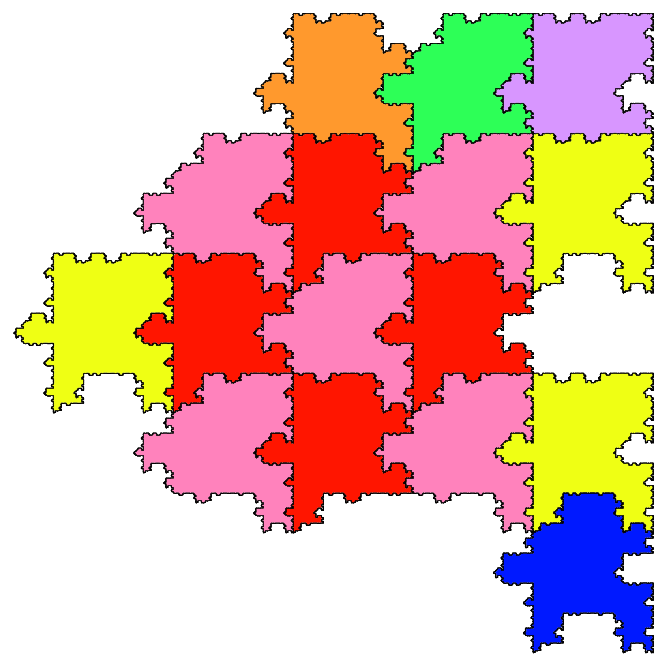} 
	\includegraphics[width=1.0in]{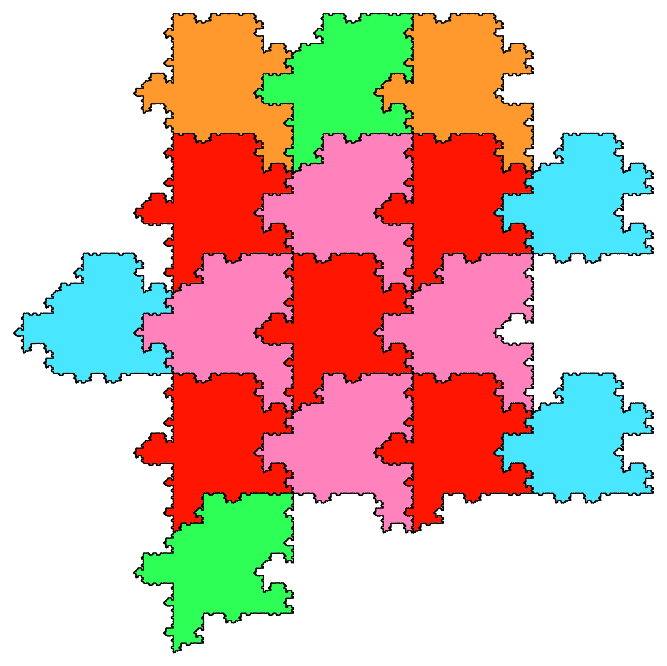}	
\]
\[
	\includegraphics[width=1.0in]{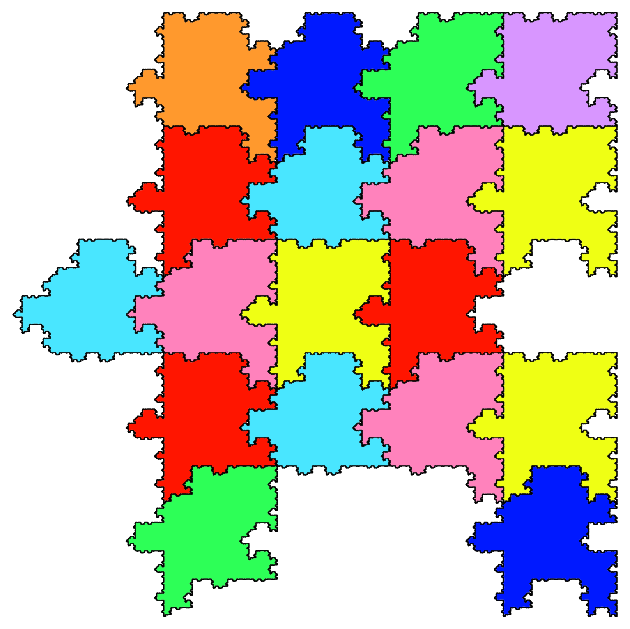}
	\includegraphics[width=1.0in]{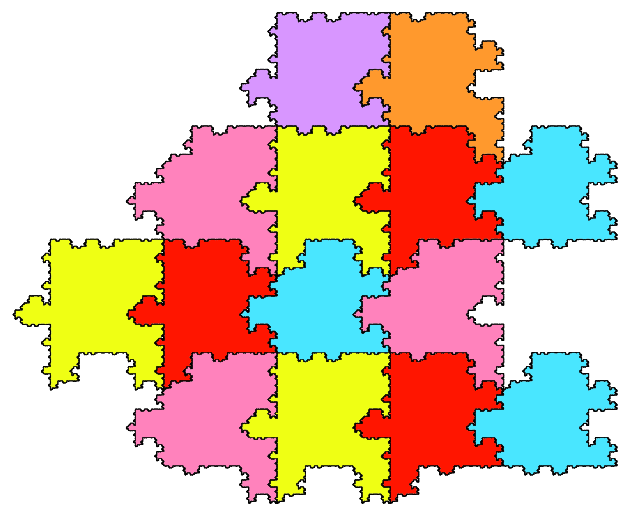} 
	\includegraphics[width=1.0in]{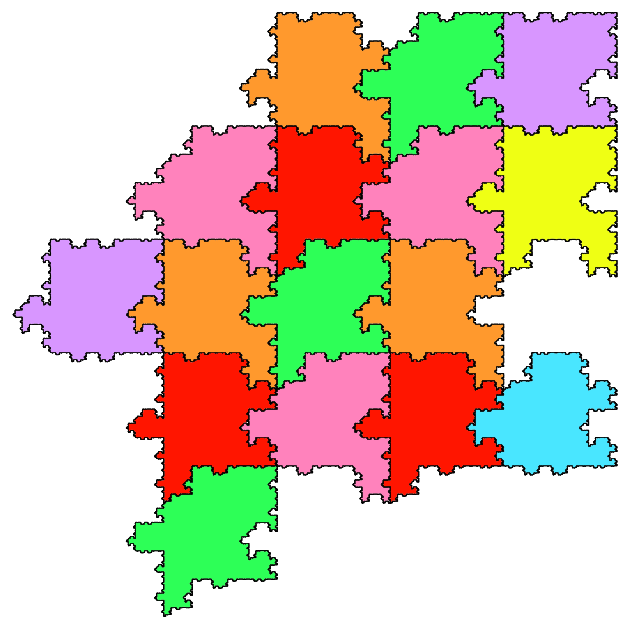} 
	\includegraphics[width=1.0in]{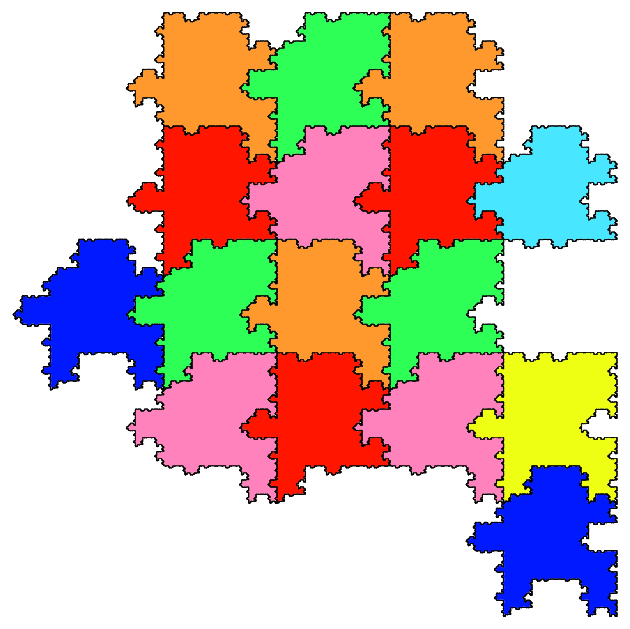}	
\]
\caption{The substitutions of the $8$ fractal prototiles (not to scale).}
\label{fig:2DTM_substitution} 
\end{figure}

Our first goal is to define the matrices used to compute the cohomology $\check{H}^i(\oinf)$. The fractal edges of the tiles are in bijective correspondence with pairs of edges in $G$ whose endpoints are identified in $S$ by boundary vertices. Since we will be interested in the substitution matrices of these edges, we label them in the ordered set $\{42, 46, 82, 86, 31, 35, 71, 75\}$. The pictures above give rise to the matrices:

\begin{small}
\[
A_0=\left( \begin{matrix}
1 & 0 \\
0 & 1 \\
\end{matrix} \right)
\quad
\delta_1=\left( \begin{matrix}
0 & 0 \\
-1 & 1\\
1 & -1 \\
0 & 0 \\
0 & 0 \\
-1 & 1 \\
1 & -1 \\
0 & 0 \\
\end{matrix} \right)
\quad
\delta_2=\left( \begin{matrix}
0 & 0 & 0 & 0 & 0 & 0 & 0 & 0 \\
0 & 0 & 0 & 0 & 0 & 0 & 0 & 0 \\
0 & -1 & 1 & 0 & 0 & 1 & -1 & 0 \\
0 & 1 & -1 & 0 & 0 & -1 & 1 & 0 \\
1 & 0 & 0 & -1 & 0 & 0 & 0 & 0 \\
-1 & 0 & 0 & 1 & 0 & 0 & 0 & 0 \\
0 & 0 & 0 & 0 & -1 & 0 & 0 & 1 \\
0 & 0 & 0 & 0 & 1 & 0 & 0 & -1 \\
\end{matrix} \right)
\]
\[
A_1=\left( \begin{matrix}
1 & 1 & 1 & 1 & 0 & 0 & 0 & 0 \\
0 & 2 & 1 & 1 & 0 & 0 & 0 & -1 \\
1 & 1 & 2 & 0 & 1 & 0 & 0 & 0 \\
1 & 1 & 1 & 1 & 1 & 0 & 0 & -1 \\
0 & 0 & 0 & 0 & 1 & 1 & 1 & 1 \\
1 & 0 & 0 & -1 & 0 & 2 & 1 & 0 \\
-1 & 0 & 0 & 1 & 1 & 1 & 2 & 1 \\
0 & 0 & 0 & 0 & 1 & 1 & 1 & 1 \\
\end{matrix} \right)
\quad
A_2=\left( \begin{matrix}
2&2&2&2&2&2&2&2\\
2&2&2&2&2&2&2&2\\
1&1&5&4&3&0&1&1\\
0&0&4&5&0&3&2&2\\
2&1&3&3&3&3&2&1\\
0&1&3&3&3&3&0&1\\
0&2&3&3&1&1&3&3\\
2&0&3&3&1&1&3&3\\
\end{matrix} \right)
\]
\end{small}

The zeroth cohomology of the approximant $\Gamma$ is generated by $\ker(\delta_0)=(1 1)^t$, which is viewed as the function assigning the value $1$ to each vertex in $\Gamma$. So $\check{H}^0(\Gamma) \cong \Z$. Since $A_0$ is the identity matrix, it follows that $\check{H}^0(\oinf) \cong \Z$ as well.

The first cohomology of the approximant $\Gamma$ is given by $\check{H}^1(\Gamma)= \ker(\delta_1)/\Im(\delta_0)$, and routine linear algebra shows that
\begin{small}
\[
\check{H}^1(\Gamma)=\newspan \left\{ 
\left( \begin{matrix} 1\\0\\0\\1\\0\\0\\0\\0\\ \end{matrix} \right),
\left( \begin{matrix} 0\\1\\1\\0\\0\\0\\0\\0\\ \end{matrix} \right),
\left( \begin{matrix} 0\\0\\0\\0\\1\\0\\0\\1\\ \end{matrix} \right),
\left( \begin{matrix} 0\\1\\0\\0\\0\\1\\0\\0\\ \end{matrix} \right)
\right\} \cong \Z^4.
\]
\end{small}
The induced matrix $A_1^*$ on $\check{H}^1(\Gamma)$ has eigenvalues $\{4,4,1,1\}$ and eigenvectors $EV(A_1^*)$:
\begin{small}
\[
A_1^*=\left( \begin{matrix}
2 & 1 & 0 & 0 \\
2 & 3 & 0 & 0 \\
0 & -1 & 2 & 2 \\
1 & 0 & 1 & 3 \\
\end{matrix} \right)
\quad
EV(A_1^*)=\left\{ 
\left( \begin{matrix} 1\\ 2\\0\\ 1\\ \end{matrix} \right),
\left( \begin{matrix} -1\\ -2\\1\\ 0\\ \end{matrix} \right),
\left( \begin{matrix} -2\\ 2\\0\\ 1\\  \end{matrix} \right),
\left( \begin{matrix} -1\\ 1\\1\\ 0\\  \end{matrix} \right)
\right\}.
\]
\end{small}
Thus, we have $\check{H}^1(\oinf) \cong \Z [1/4]^2 \oplus \Z^2$.

Finally, $\check{H}^2(\Gamma)= \Z^8/\Im(\delta_1)$ and we see that
\begin{small}
\[
\check{H}^2(\Gamma)=\newspan \left\{ 
\left( \begin{matrix} 1\\1\\0\\0\\0\\0\\0\\0\\ \end{matrix} \right),
\left( \begin{matrix} 0\\0\\1\\1\\0\\0\\0\\0\\ \end{matrix} \right),
\left( \begin{matrix} 0\\0\\0\\0\\1\\1\\0\\0\\ \end{matrix} \right),
\left( \begin{matrix} 0\\0\\0\\0\\0\\0\\1\\1\\ \end{matrix} \right),
\left( \begin{matrix} 1\\-1\\0\\0\\0\\0\\0\\0\\ \end{matrix} \right)
\right\}.
\]
\end{small}
The induced matrix $A_2^*$ on $\check{H}^2(\Gamma)$ has nonzero eigenvalues $\{16,4,4,1\}$ with eigenvectors $EV(A_2^*)$:
\begin{small}
\[
A_2^*=\left( \begin{matrix}
4 & 1 & 2 & 2 & 0 \\
4 & 9 & 6 & 6 & 0 \\
4 & 3 & 6 & 2 & 0 \\
4 & 3 & 2 & 6 & 0 \\
0 & 0 & 0 & 0 & 0 \\
\end{matrix} \right)
\quad
EV(A_2^*)=\left\{ 
\left( \begin{matrix} 1\\ 4\\2\\ 2\\ 0\\ \end{matrix} \right),
\left( \begin{matrix} 4\\ 9\\6\\ 6\\ 0\\ \end{matrix} \right),
\left( \begin{matrix} 4\\ 3\\6\\ 2\\ 0\\  \end{matrix} \right),
\left( \begin{matrix} 4\\ 3\\2\\ 6\\ 0\\  \end{matrix} \right)
\right\}.
\]
\end{small}
Thus, we have $\check{H}^2(\oinf) \cong \Z [1/16] \oplus \Z [1/4]^2 \oplus \Z$.

\appendix

\section{Examples}\label{examples_appendix_sec}

\subsection{A fractal dual for the Ammann-Beenker (``octagonal") tiling}

We consider the version of this tiling using two labelled right triangles and a parallelogram, where the labels keep track of the handedness of the tiles.  These tiles come in rotations by $\pi/4$, all of which we treat the same way.  
In Figure \ref{octagonalrecur} we show the parallelogram and one copy of the triangle; all rotations and reflections will carry along geometric graphs.
Denoting $\pset=\{\alpha, \beta\}$ we define $G$ to be a $T$-consistent dual graph in both $\alpha$ and $\beta$. We show $G, \rrule(\pset), \rrule(G),$ and select $S$ as in Figure \ref{octagonalrecur}. The fractal realization of the Ammann-Beenker Tiling induced by $(G,S)$ is shown in Figure \ref{bigoct}.
\begin{figure}[ht]
\[
\scalebox{0.9}{
\begin{tikzpicture}
%

\node (a) at (0,3) [label=below:{$\alpha$ and $G_\alpha$}] {\includegraphics[width=3cm]{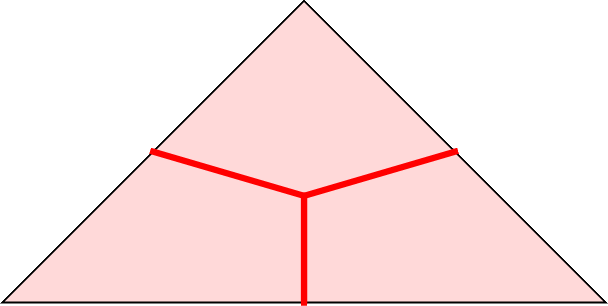}};
\node (b) at (4,3) [label=below:{$\rrule(G)_\alpha$}] {\includegraphics[width=3cm]{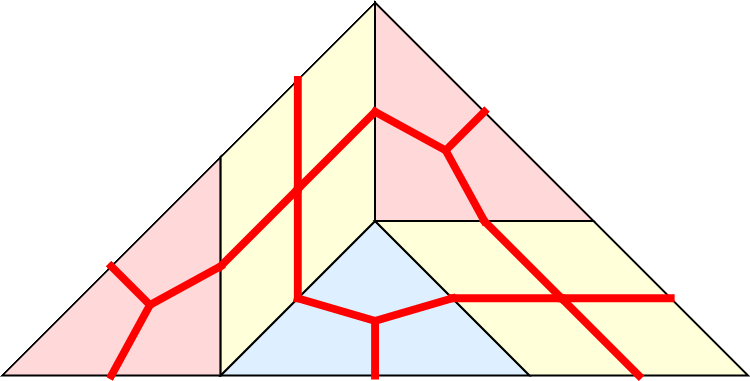}}
	edge[<-] (a);
\node (c) at (8,3) [label=below:{$S_\alpha$}] {\includegraphics[width=3cm]{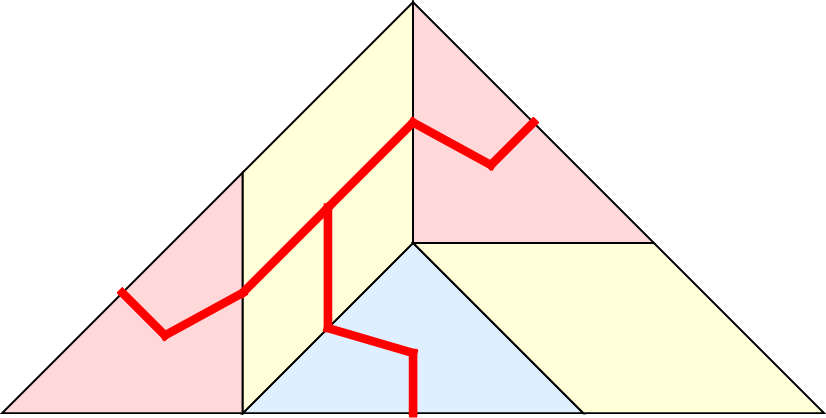}}
	edge[<-] (b);
\node (d) at (12,3) {\includegraphics[width=3cm]{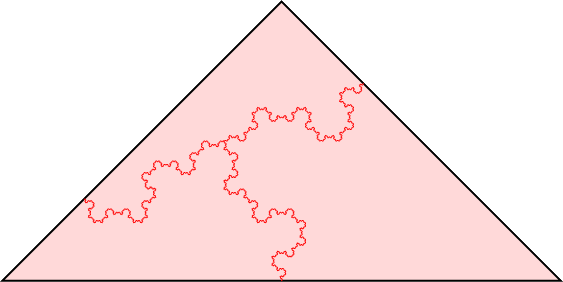}}
	edge[<-,dashed] (c);

\node (e) at (0,0) [label=below:{$\beta$ and $G_\beta$}] {\includegraphics[width=3cm]{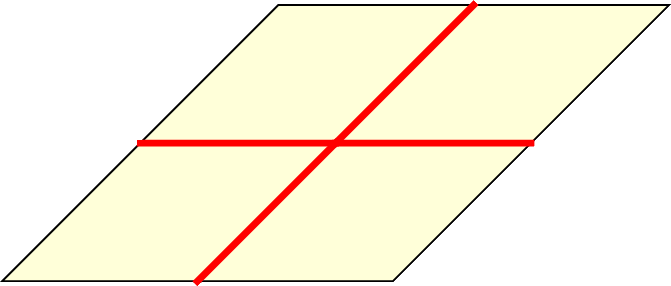}};
\node (f) at (4,0) [label=below:{$\rrule(G)_\beta$}] {\includegraphics[width=3cm]{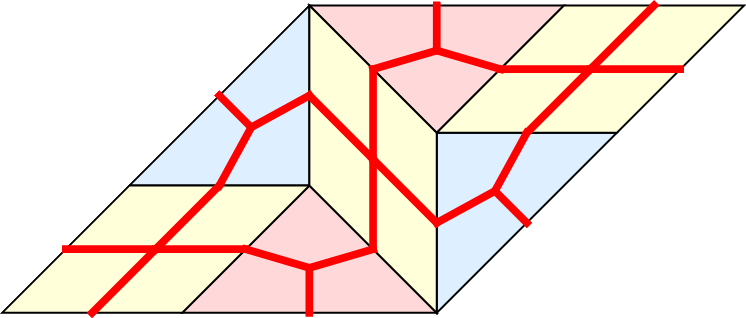}}
	edge[<-] (e);
\node (g) at (8,0) [label=below:{$S_\beta$}] {\includegraphics[width=3cm]{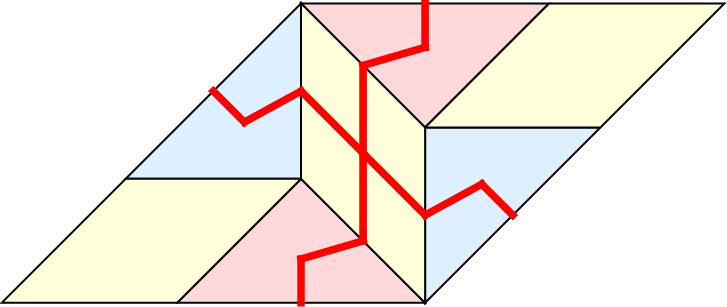}}
	edge[<-] (f);
\node (h) at (12,0)  {\includegraphics[width=3cm]{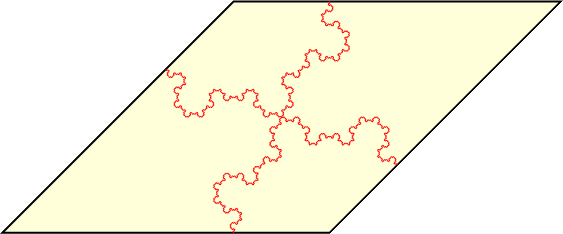}}
	edge[<-,dashed] (g);
\end{tikzpicture}}
\]
\caption{A recurrent pair and resulting fractal for the dual graph of the Ammann-Beenker tiling}
\label{octagonalrecur}
\end{figure}
\begin{figure}[ht]
\includegraphics[width=5in]{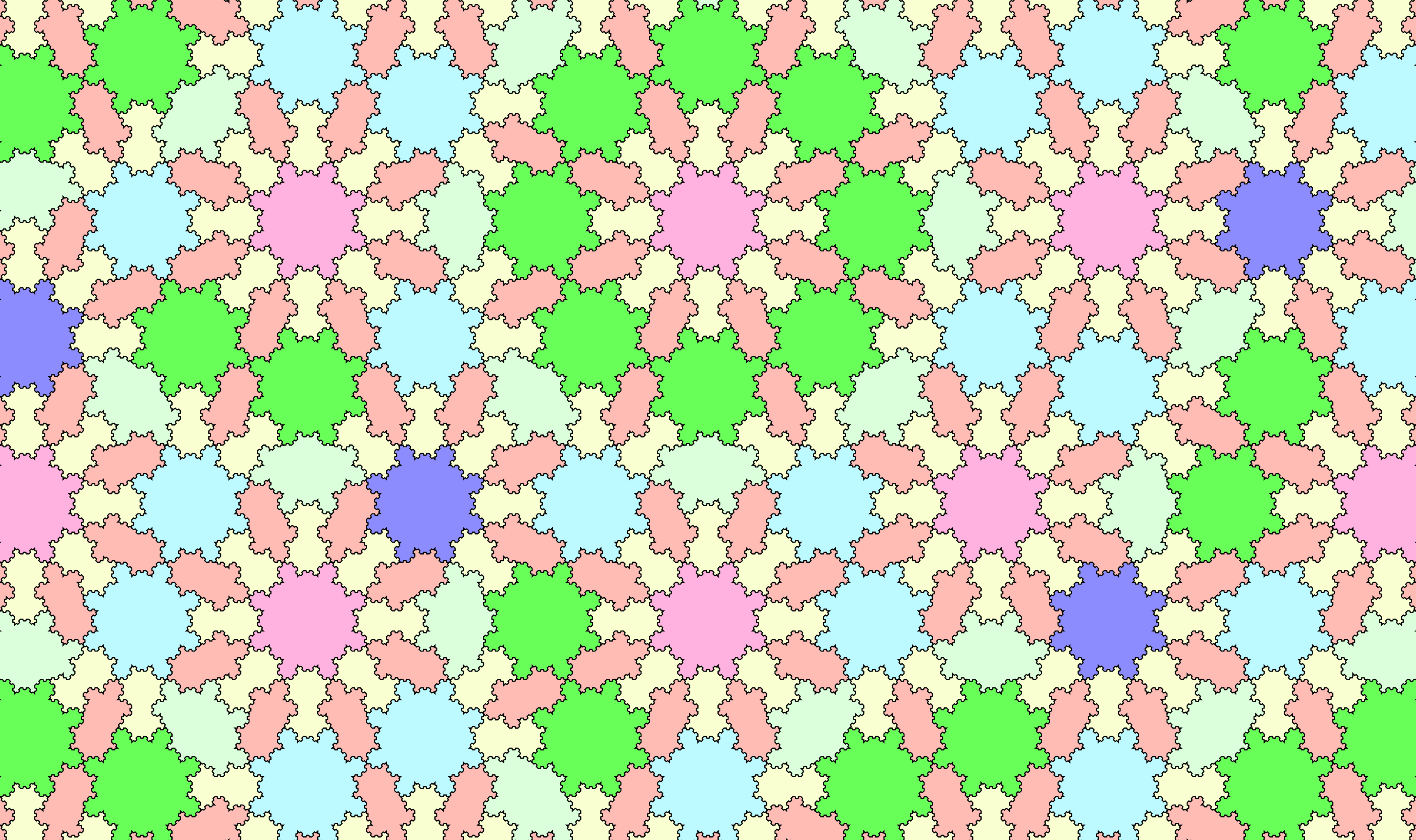}
\caption{A patch of a fractal dual Ammann-Beenker tiling}
\label{bigoct}
\end{figure}

In Figure \ref{octsubs} we show the substitution rule for the prototiles of the fractal realization.
\begin{figure}[ht]

\[
\begin{tikzpicture}[>=stealth] 

\node (t1) at (0,4) {\includegraphics[scale=0.2]{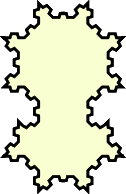}};
\node (t1s) at (2,4) {\includegraphics[scale=0.2]{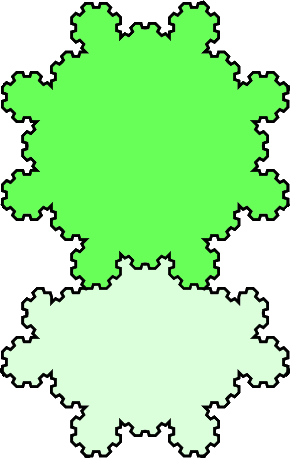}}
	edge[<-] (t1);
\node (t2) at (4,4) {\includegraphics[scale=0.2]{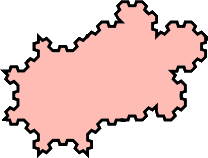}};
\node (t2s) at (6.8,4) {\includegraphics[scale=0.2]{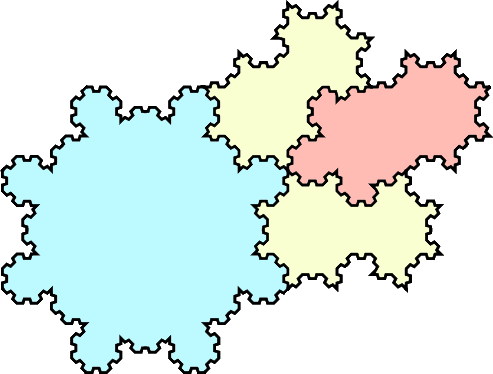}}
	edge[<-] (t2);
\node (t3) at (9.5,4) {\includegraphics[scale=0.2]{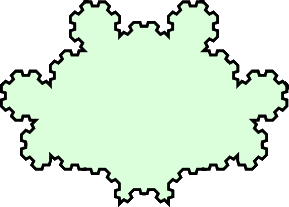}};
\node (t3s) at (13,4) {\includegraphics[scale=0.2]{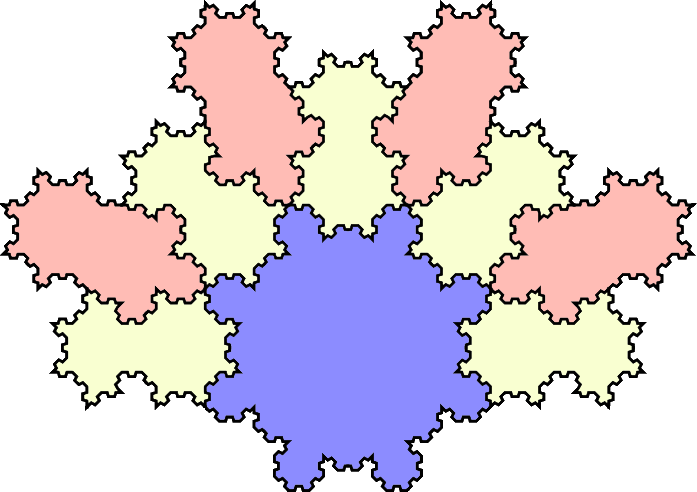}}
	edge[<-] (t3);
		
\node (t4) at (3,1) {\includegraphics[scale=0.2]{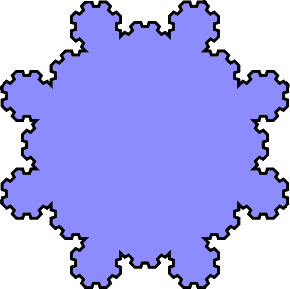}};
\node (t5) at (3,-1) {\includegraphics[scale=0.2]{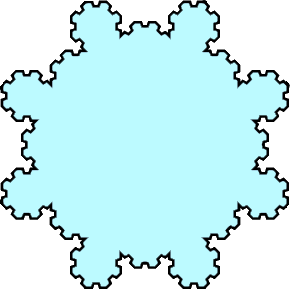}};
\node (t6) at (11,1) {\includegraphics[scale=0.2]{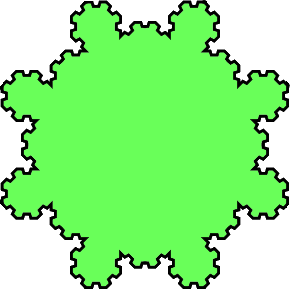}};
\node (t7) at (11,-1) {\includegraphics[width=1.6cm]{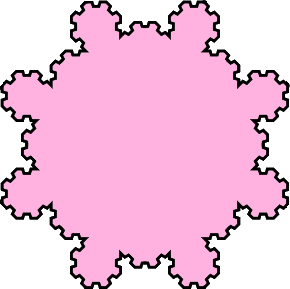}};

\node (t4s) at (7,0) {\includegraphics[scale=0.2]{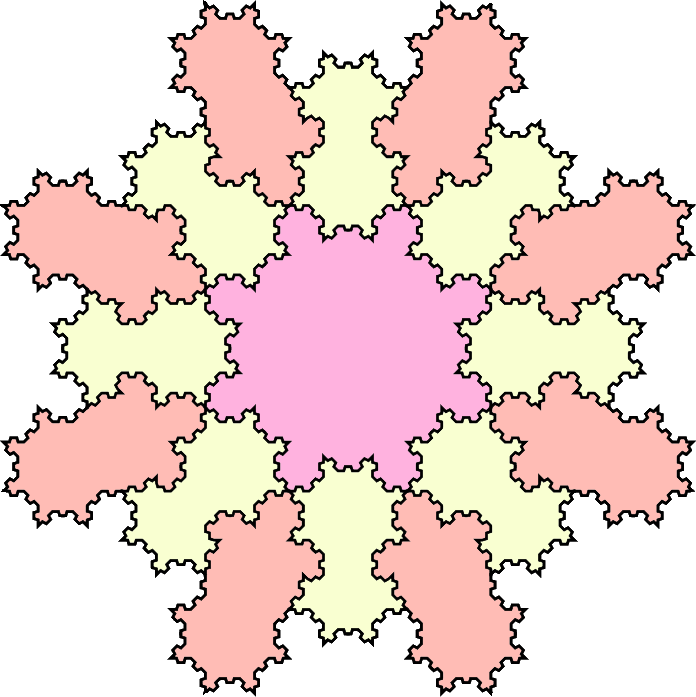}}
	edge[<-] (t4)
	edge[<-] (t5)
	edge[<-] (t6)
	edge[<-] (t7);

\end{tikzpicture}
\]

\caption{The substitution rules of this fractal realization of the Ammann-Beenker octagonal tiling.}
\label{octsubs}
\end{figure}

\subsection{Penrose's ``Pentaplexity" tiling \cite{pentaplexity}}  The simplest self-similar version of the Penrose tilings has a prototile set with forty triangles, the two tiles on the left of Figure \ref{penrosetrianglegraph}, their labelled reflections, and all rotation by $\pi/5$ of these four prototiles.  These triangles can be combined to produce either the kite and dart or the rhombus tilings, and all three prototile sets have matching rules that make them aperiodic tile sets.  
The kite/dart and rhombus versions are only pseudo-self-similar and not actually self-similar, so we've chosen to work with the triangle version instead.  One interesting note about the fractal tiling we obtain is that it is closely related to the pentaplexity version of the Penrose tiling that arises in the first six images of \cite{pentaplexity}.
\begin{figure}[ht]
\[
\scalebox{0.9}{
\begin{tikzpicture}
\node (a) at (0,4) [label=below:{$\alpha$ and $G_\alpha$}] {\includegraphics[width=3cm]{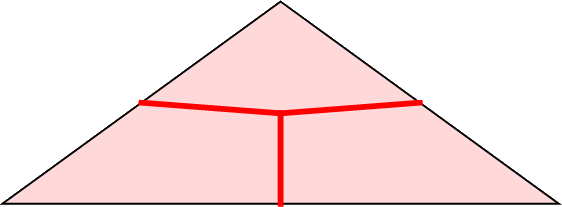}};
\node (b) at (4,4) [label=below:{$\rrule(G)_\alpha$}] {\includegraphics[width=3cm]{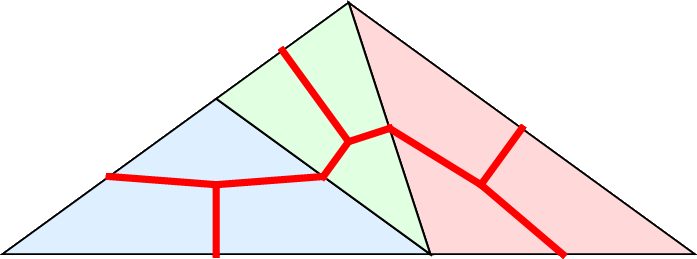}}
	edge[<-] (a);
\node (c) at (8,4) [label=below:{$S_\alpha$}] {\includegraphics[width=3cm]{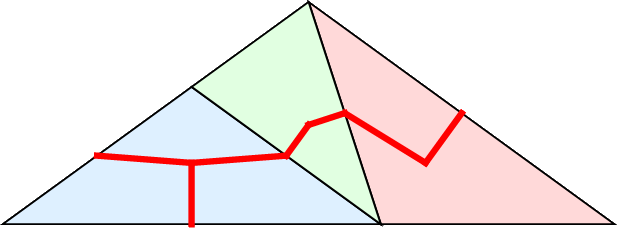}}
	edge[<-] (b);
\node (d) at (12,4) {\includegraphics[width=3cm]{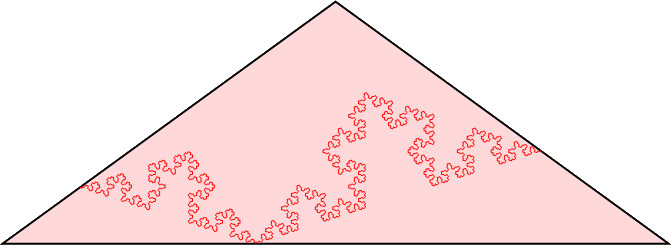}}
	edge[<-,dashed] (c);

\node (e) at (1.5,0) [label=below:{$\beta$ and $G_\beta$}] {\includegraphics[width=2cm]{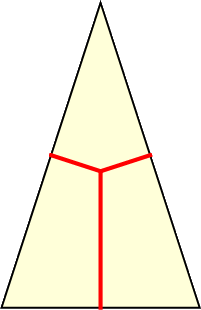}};
\node (f) at (4.5,0) [label=below:{$\rrule(G)_\beta$}] {\includegraphics[width=2cm]{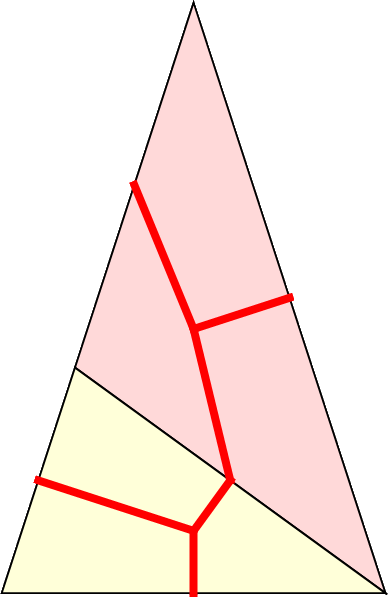}}
	edge[<-] (e);
\node (g) at (7.5,0) [label=below:{$S_\beta$}] {\includegraphics[width=2cm]{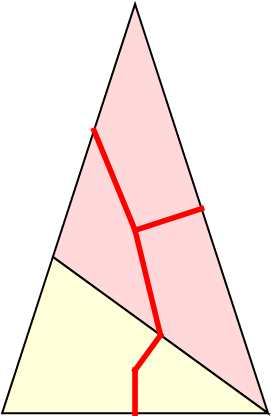}}
	edge[<-] (f);
\node (h) at (10.5,0)  {\includegraphics[width=2cm]{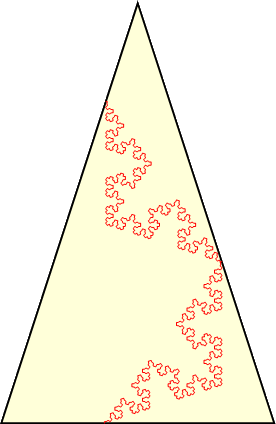}}
	edge[<-,dashed] (g);
\end{tikzpicture}}
\]
\caption{A recurrent pair for the Penrose triangles; note that each triangle come in two reflections and so do their geometric graphs.}
\label{penrosetrianglegraph}
\end{figure}
We choose an initial graph $G$ on the two prototiles as pictured in Figure \ref{penrosetrianglegraph} along with a choice of $S$ that produces a recurrent pair. We then extend these graphs to all forty prototiles by reflection and rotation.  What we gain in simplicity by working with one iteration of the substitution we pay for with an inability to construct $S$ in a way that satisfies the injectivity conditions for any level of $N$-subtiles. The recurrent pair $(G,S)$ does not give rise to an injective $\psiinf$ because one edge in each prototile is collapsed into a vertex, so Theorem \ref{prop:T infty is a tiling} does not apply.  Nonetheless $\psiinf(\partial\tzero)$ forms the boundary of a self-similar tiling $\Tinf$, a patch of which is shown in Figure \ref{bigPenrose}.
\begin{figure}[ht]
\includegraphics[width=5in]{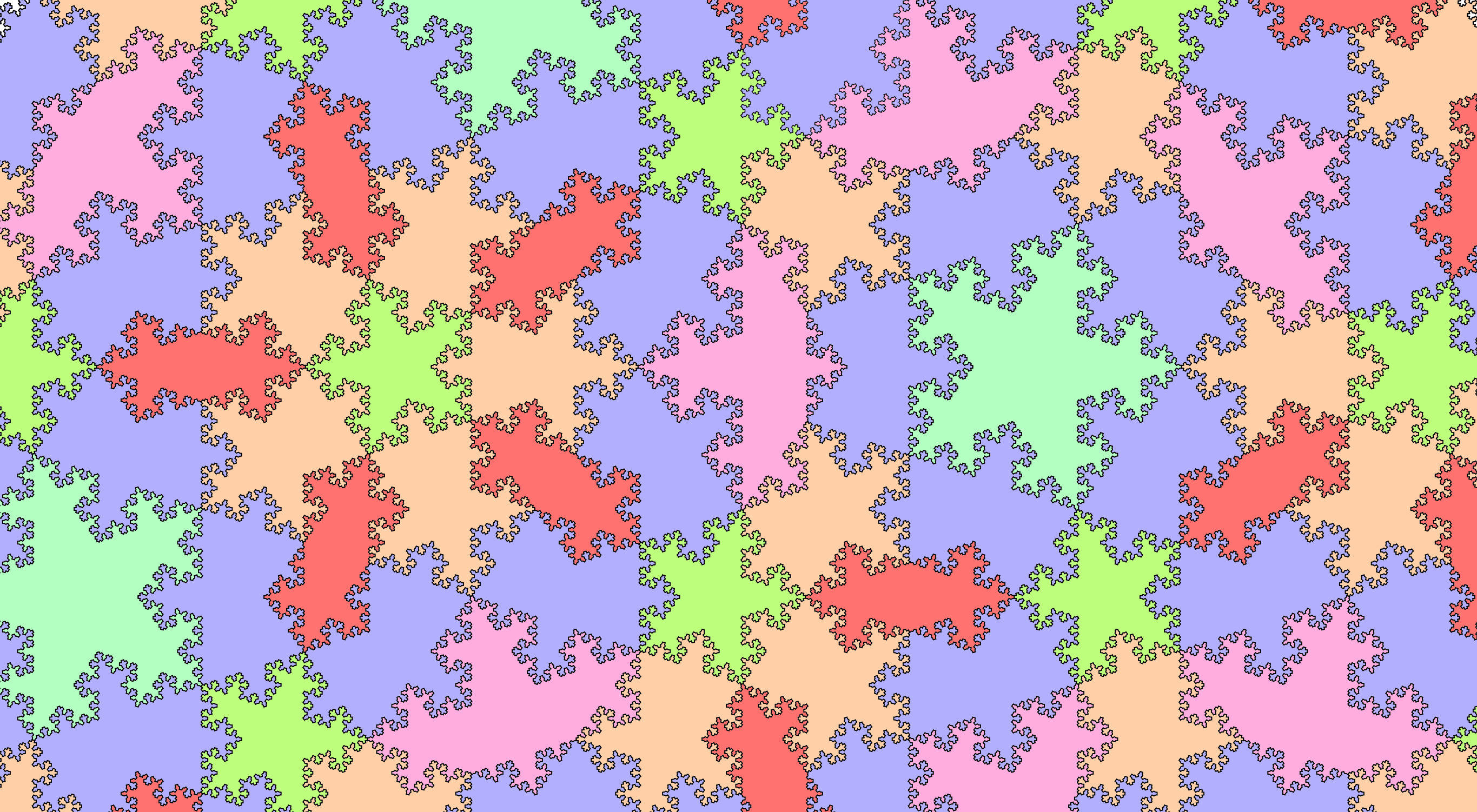}
\caption{A patch of the fractal pentaplexity tiling.}
\label{bigPenrose}
\end{figure}

\begin{figure}[ht]

\[
\begin{tikzpicture}[>=stealth] 

\node (t1) at (0,3) {\includegraphics[width=1.0cm]{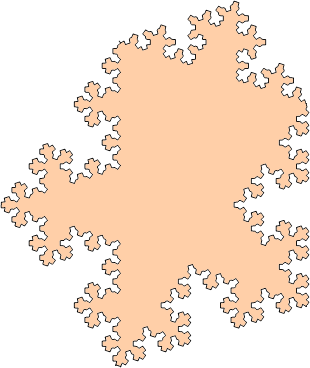}};
\node (t1s) at (2.5,3) {\includegraphics[width=1.6cm]{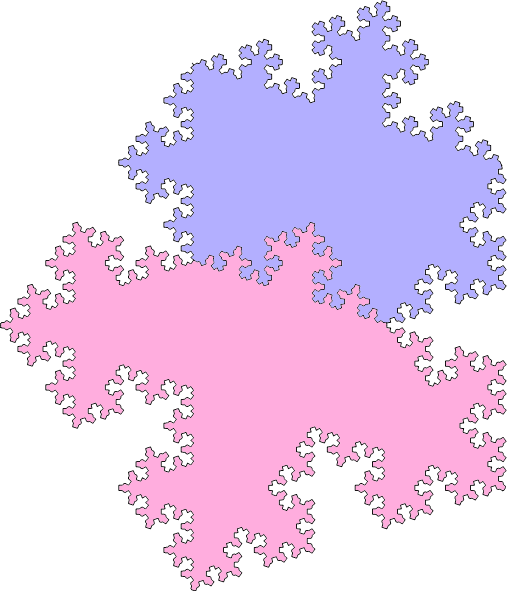}}
	edge[<-] (t1);
\node (t2) at (5,3) {\includegraphics[width=1.0cm]{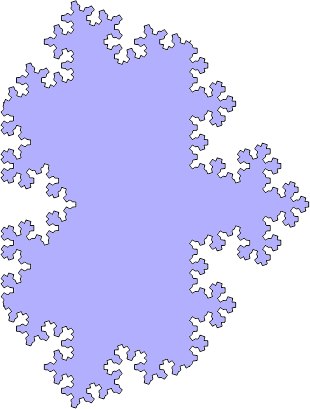}};
\node (t2s) at (7.5,3) {\includegraphics[width=1.6cm]{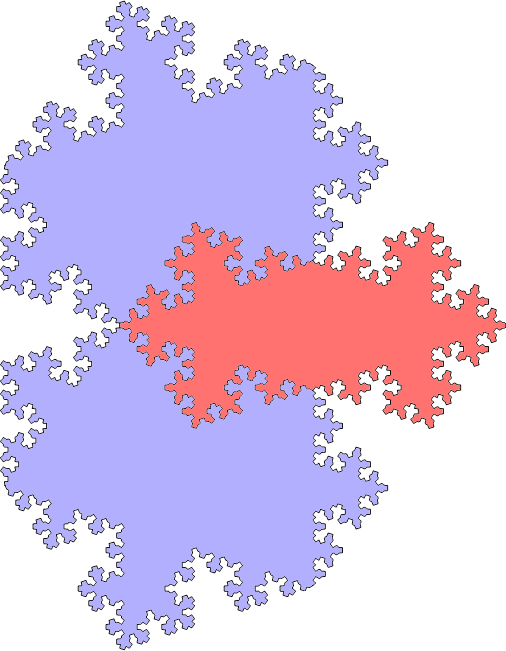}}
	edge[<-] (t2);
\node (t3) at (10,3) {\includegraphics[width=1.25cm]{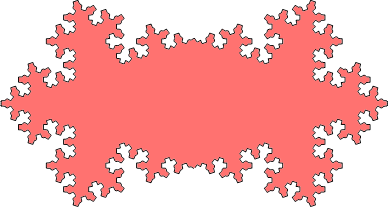}};
\node (t3s) at (12.5,3) {\includegraphics[width=2.0cm]{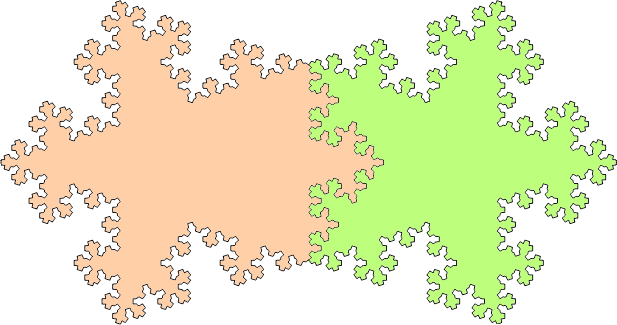}}
	edge[<-] (t3);
		
\node (t4) at (0,0) {\includegraphics[width=1.65cm]{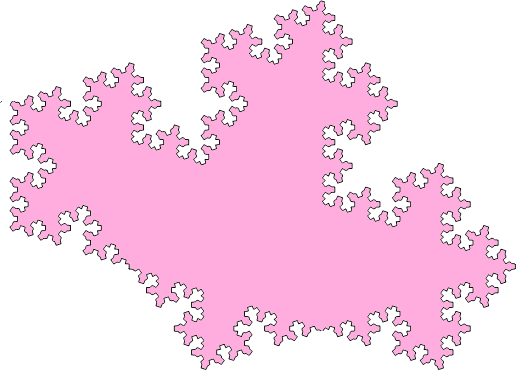}};
\node (t4s) at (3,0) {\includegraphics[width=2.65cm]{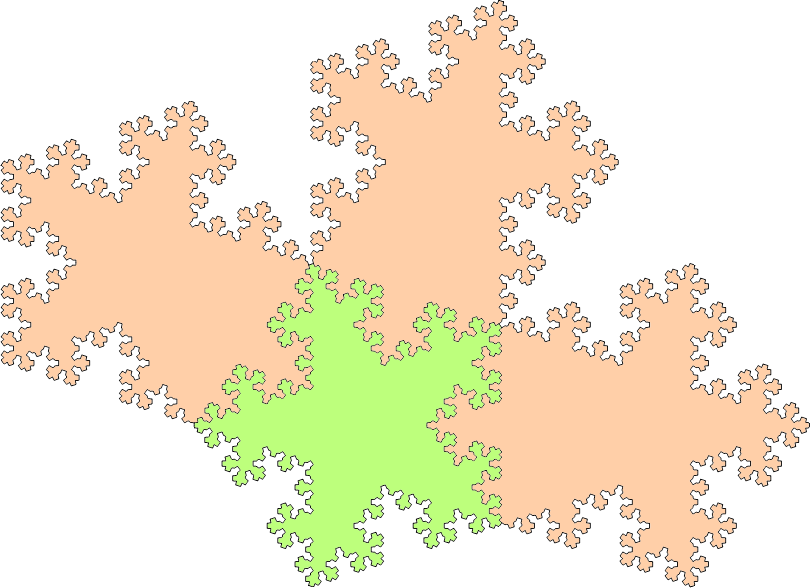}}
	edge[<-] (t4);
\node (t5) at (5.25,0) {\includegraphics[width=1.0cm]{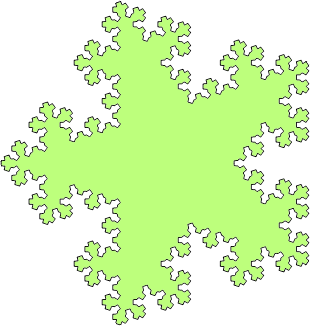}};
\node (t5s) at (7.5,0) {\includegraphics[width=1.6cm]{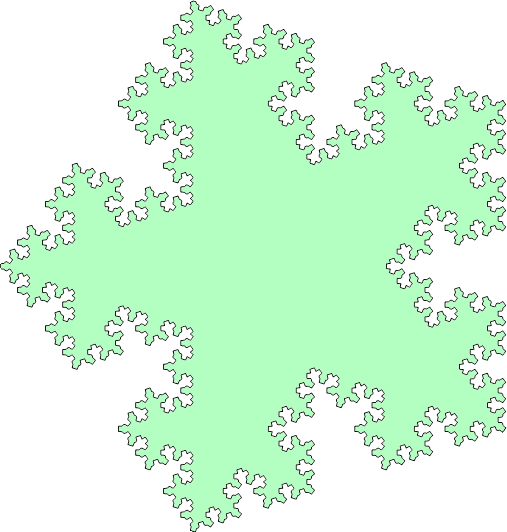}}
	edge[<-] (t5);
\node (t6) at (9.5,0) {\includegraphics[width=1.6cm]{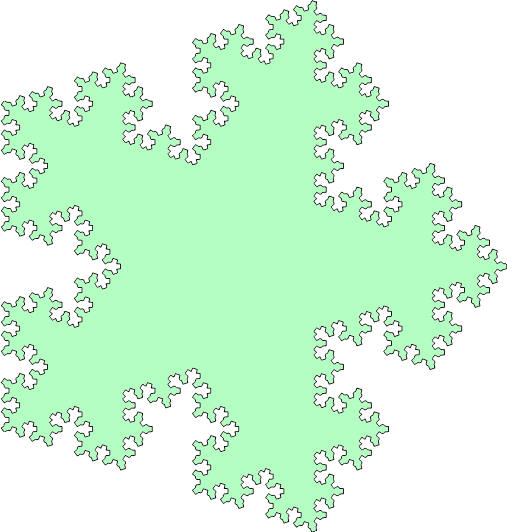}};
\node (t6s) at (12.5,0) {\includegraphics[width=2.56cm]{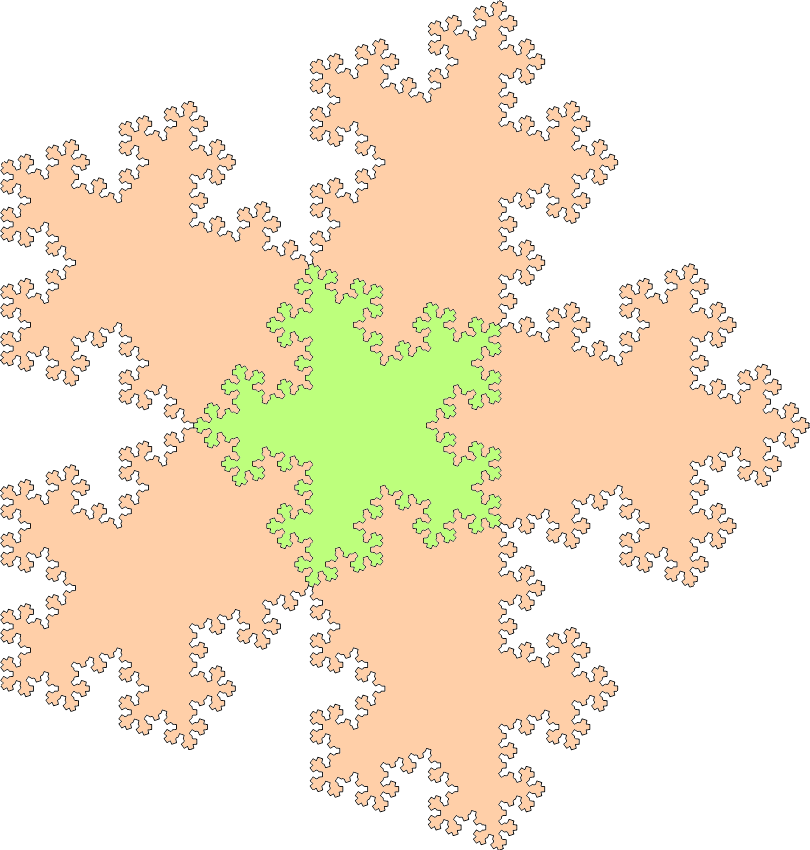}}
	edge[<-] (t6);

\end{tikzpicture}
\]

\caption{A fractal version of the pentaplexity substitution. }
\label{PenroseFractalSub}
\end{figure}

There are five tile types in this fractal Penrose prototile set (ignoring rotations), some of which are shown along with their substitutions in Figure \ref{PenroseFractalSub}.  The connection to the tiles developed by Penrose at the beginning of \cite{pentaplexity} were discussed in a 2013 talk \cite{BrunellSherwood}.

\subsection{Two ways to construct fractal realizations of the chair tiling}

If we endow the chair tiles with vertices from a tiling $T$ as described in section \ref{subsec:tilinggraphs}, then each tile is an octagon and the tiling is edge-to-edge but not singly edge-to-edge. Thus the existence theorem does not directly apply.  Nonetheless it is possible to use the existence algorithm in Section \ref{subsec:exist} to establish numerous fractal realizations by quasi-dual graphs.  Interestingly, none of these quasi-dual graphs can actually be duals.

Another approach is to label the tiles of $T$ by the number of natural edges they have.  It turns out that there are three types of tiles in this case and they have 4, 5, and 6 edges.  We can expand our prototile set to have twelve labels rather than four (once rotations are accounted for) and it is routine to write down the tiling substitution for each of the three tile types.  We exhibit a self-similar fractal realization for the dual tiling in this situation.

There are many other possible ways to manage the chair tiling, for instance by passing to the ``square chair" version \cite[p.16]{Lorenzo.book}, which subdivides each chair tile into three labelled squares. The square chair is single edge-to-edge so that fractal realizations are guaranteed by Theorem \ref{thm:existence}.  This is an MLD operation and, as above, yields equivalent hulls, but the geometry and combinatorics are not as well preserved. 

\begin{example}[Chairs as octagons]
We consider there to be four chair tiles, each octagons, with natural and inherited vertices as shown in Figure \ref{chair.vertices}.  As usual, we treat all rotations the same so our figures contain only one tile. 

We begin by noticing the simple reason why the dual tiling can never have a self-similar realization.  The dual graph $G_0$ has, on each prototile, a single interior vertex of degree 8. Trying to find a subgraph $S$ of $\rrule^N(G_0)$ that is equivalent to $G_0$ is doomed to failure: no matter which vertex of degree 8 is selected from the interior, there will be two edges emanating from that vertex that come back together in an adjacent subtile and cannot therefore both be part of a tree.  This illustrates how essential the singly edge-to-edge condition is to our construction.

It is possible to run the algorithm guaranteeing existence in Theorem \ref{thm:existence} to obtain a quasi-dual recurrent pair for the octagonal chair tiling even though the tiling is not singly edge-to-edge. The ultimate choice of $G$ has a minimum of two vertices on the interior of the prototile and the minimum number of iterations required to satisfy the injectivity conditions is 3.  Instead, we choose to exhibit a recurrent pair that is not quasi-dual and whose limiting fractal intersects prototile vertices but for which $\psiinf$ is injective nonetheless.  

The initial graph $G$ has 9 edges and two interior vertices, and two of the boundary vertices have degree 2; we show this graph on the left of Figure \ref{chairoctgraph}.  Conveniently, there is a subgraph of $\rrule(G)_\alpha$ that is isomorphic to $G$ which we extract as shown in the same figure. 
\begin{figure}[ht]
\[
\begin{tikzpicture}[scale=0.9]
\node (a) at (0,4) [label=below:{$\alpha$ and $G_\alpha$}] {\includegraphics[width=.8in]{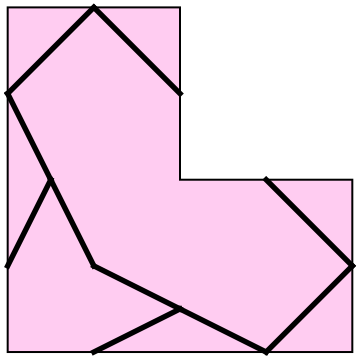}};
\node (b) at (4,4) [label=below:{$\rrule(G)_\alpha$}] {\includegraphics[width=.8in]{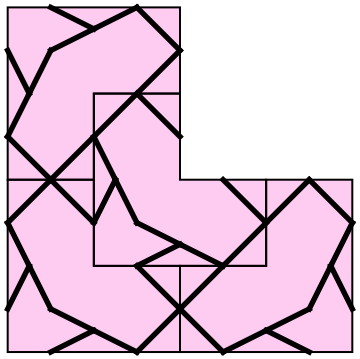}}
	edge[<-] (a);
\node (c) at (8,4) [label=below:{$S_\alpha$}] {\includegraphics[width=.8in]{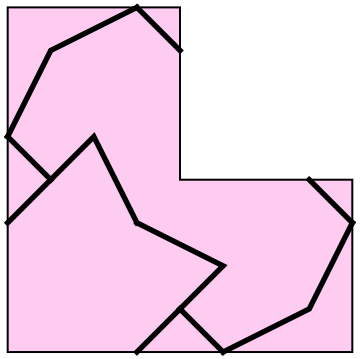}}
	edge[<-] (b);
\node (d) at (12,4) {\includegraphics[width=.8in]{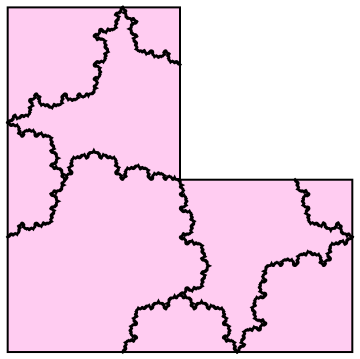}}
	edge[<-,dashed] (c);

\end{tikzpicture}
\]
\caption{A recurrent pair and limiting fractal for the chairs as octagons.}
\label{chairoctgraph}
\end{figure}

 On the far right of Figure \ref{chairoctgraph} we see the result of iterating the graph substitution induced by this recurrent pair.  The central edge has migrated to the boundary of the prototile and touches one of its vertices. The fact that there is no problem with the injectivity of $\psiinf$ can be seen by looking at the patch of fractiles superimposed on chairs shown in Figure \ref{fig:chairoctpatch}.

\begin{figure}[ht]
\includegraphics[width=5in]{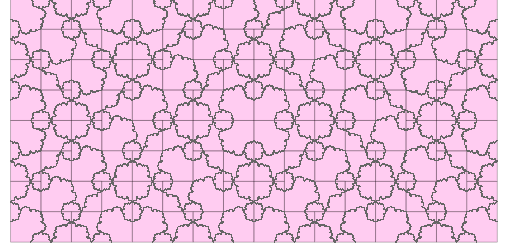}
\caption{A patch of the octagonal chair tiling with a fractal realization overlaid.}
\label{fig:chairoctpatch}
\end{figure}

\end{example}

\begin{example}[Colored chairs]
Figure \ref{fig:chairs} shows the three types of chairs along with their dual graphs, embedded in their final fractal form.  Although the graph edges appear to intersect prototile vertices, they do not: the natural vertices of these tile types do not include the graph vertices as they appear in the tiling.

\begin{figure}[ht]
\centering
\subfigure{\includegraphics[width=.8in]{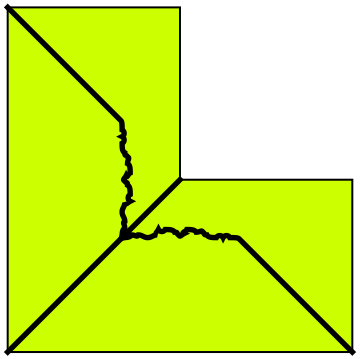}}\hspace{.35in}
\subfigure{\includegraphics[width=.8in]{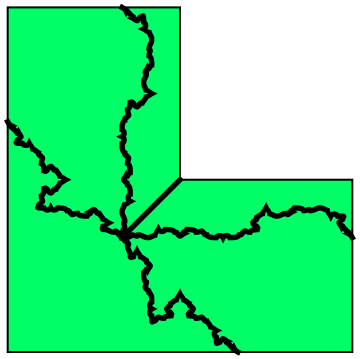}}\hspace{.35in}
\subfigure{\includegraphics[width=.8in]{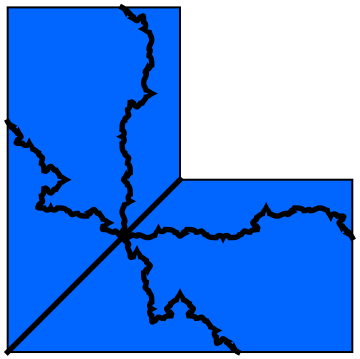}}
\caption{The three chair tiles and their fractal dual graphs}
\label{fig:chairs}
\end{figure}

The substitutions of the three tile types appear in Figure \ref{fig:chair subs}.  It should be noted that one substitution is not enough to find a recurrent pair for the dual substitution, however: none of the three graphs in $\rrule(G)$ contain combinatorially equivalent subgraphs.  Thus, we used $\rrule^2(G)$ to produce the graph $S$ for our recurrent pair.

\begin{figure}[ht]
\scalebox{0.8}{
\centering
\subfigure{\includegraphics[width=1in]{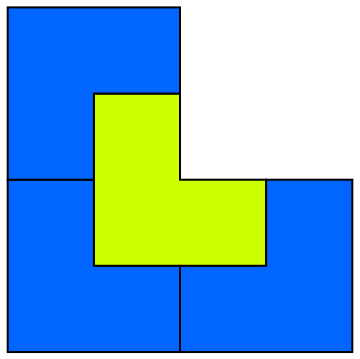}}\hspace{.35in}
\subfigure{\includegraphics[width=1in]{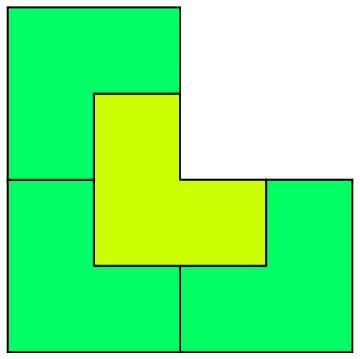}}\hspace{.35in}
\subfigure{\includegraphics[width=1in]{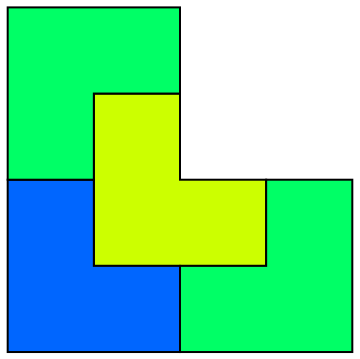}}}
\caption{Substitution of the three chairs (not to scale)}
\label{fig:chair subs}
\end{figure}

In Figure \ref{fig:chair dual its} we show a patch of the chair with the self-similar dual tiling overlaid atop it.

\begin{figure}[ht]
\includegraphics[width=5in]{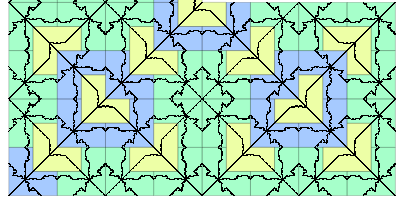}
\caption{A patch of the chair tiling with its fractal dual overlaid.}
\label{fig:chair dual its}
\end{figure}

\end{example}

\end{document}